\documentclass[11pt,reqno]{amsart}
\usepackage{mathrsfs}
\usepackage{amsgen}
\usepackage{amscd}
\usepackage{amsmath}
\usepackage{latexsym}
\usepackage{amsfonts}
\usepackage{amssymb}
\usepackage{amsthm}
\usepackage{accents}
\usepackage{graphicx}
\usepackage{color}
\usepackage{amsthm,amssymb}
\usepackage{graphicx}
\usepackage{enumerate}
\usepackage{amssymb,amsmath,graphicx,amsfonts,euscript}
\usepackage{color}
\usepackage{mathrsfs}
\usepackage{empheq}
\usepackage{harpoon}
\usepackage{cases}
\usepackage{cite}
\usepackage{multirow}

\usepackage[colorlinks=true,backref]{hyperref}
\hypersetup{urlcolor=blue, citecolor=green, linkcolor=blue}

\usepackage[margin=3cm, a4paper]{geometry}

\def\H{{\cal H}}
\def\F{\mathscr{F} }
\def\N{\mathbb{N}}
\def\R{\mathbb{R}}

\def\T{\mathbb{T}}
\def\C{\mathbb{C}}

\def\H2{H^2(\R^N)}
\def\L2{L^2(\R^N)}
\def\to{\rightarrow}

\newcommand{\myvec}[1]%
{\stackrel{\raisebox{-2pt}[0pt][0pt]{\small$\rightharpoonup$}}{#1}}

\def\jb#1{\langle#1\rangle} \def\norm#1{\|#1\|}
\def\normb#1{\bigg\|#1\bigg\|} 

\def\H{{\cal H}}

\def\H1{H^1(\R)}

\newcommand{\al}{\alpha}

   \newcommand{\I}{\infty}

 \newcommand{\Del}[1]{}

\numberwithin{equation}{section}

\newtheorem{thm}{Theorem}[section]

\newtheorem{lem}[thm]{Lemma}
\newtheorem{prop}[thm]{Proposition}
\newtheorem{definition}[thm]{Definition}
\newtheorem{conjecture}[thm]{Conjecture}
\theoremstyle{remark}
\newtheorem{remark}[thm]{Remark}
\newtheorem*{exam*}{Examples}

\newcommand{\EQ}[1]{\begin{align*}\begin{split} #1 \end{split}\end{align*}}
\newcommand{\EQn}[1]{\begin{align}\begin{split} #1 \end{split}\end{align}}
\newcommand{\EQnnsub}[1]{\begin{subequations}\begin{align} #1 \end{align}\end{subequations}}
\def\norm#1{\left\|#1\right\|}

\def\normb#1{\big\|#1\big\|}
\def\normbb#1{\Big\|#1\Big\|}

\def\jb#1{\langle#1\rangle}

\def\lsm{\lesssim}

\def\al{\alpha}

\begin{document}

	\setcounter{page}{1}

	\title[NLS with rough potential]{Regularization for the Schr\"{o}dinger equation with rough potential: high-dimensional case}

\author{Ruobing Bai}
\address{Ruobing Bai \newline School of Mathematics and Statistics\\
Henan University\\
Kaifeng 475004, China}
\email{baimaths@hotmail.com}
\thanks{}

	\author{Yajie Lian}
	\address{Yajie Lian \newline Center for Applied Mathematics\\
		Tianjin University\\
		Tianjin 300072, China}
	\email{yjlian@tju.edu.cn}
	\thanks{}

	\author{Yifei Wu}
	\address{Yifei Wu  \newline School of Mathematical Sciences\\
		Nanjing Normal University\\
		Nanjing 210046, China}
	\email{yerfmath@gmail.com}
	\thanks{}
	
	

	\begin{abstract}\noindent
		In this work, we investigate the regularization mechanisms of the Schr\"odinger equation with a spatial potential
		\begin{align*}
			i\partial_t u+\Delta u+\eta u =0,
		\end{align*}
		where $\eta$ denotes a given spatial potential. The regularity of solutions constitutes one of the central problems in the theory of dispersive equations. Recent works \cite{Bai-Lian-Wu-2024, M-Wu-Z24} have established the sharp regularization mechanisms for this model in the whole space $\R$ and on the torus $\T$, with $\eta$ being a rough potential.

The present paper extends the line of research to the high-dimensional setting with rough potentials $\eta \in L_x^r+L_x^{\infty}$.
More precisely, we first show that when $1\leq r <\frac d2$, there exists some $\eta \in L_x^r+L_x^{\infty}$ such that the equation is ill-posed in $H_x^{\gamma}$ for any $\gamma \in \R$. Conversely, when $\frac d2 \leq r \leq \infty$, the expected optimal regularity is given by
$$H_x^{\gamma_*}, \quad \gamma_*=\mbox{min}\{2+\frac d2-\frac dr, 2\}.$$
 We establish a comprehensive characterization of the regularity, with the exception of two dimensional endpoint case $d=2, r=1$.
Our novel theoretical framework combines several fundamental ingredients: the construction of counterexamples, the proposal of splitting normal form method, and the iterative Duhamel construction. Furthermore, we briefly discuss the effect of the interaction between rough potentials and nonlinear terms on the regularity of solutions.
		
	\end{abstract}
	
\maketitle
%
	
	\section{Introduction}
	\vskip 0.2cm

	In this paper, we study the linear Schr\"odinger equation with a ``rough" spatial potential
	\begin{equation}\label{eq:NLS}
		\left\{ \aligned
		&i\partial_t u(t, x)+\Delta u(t, x)+\eta(x) u(t,x)=0,
		\\
		&u(0,x)=u_0(x),
		\endaligned
		\right.
	\end{equation}
	where $u(t, x):\R^+\times \R^d\rightarrow \C$ is an unknown function, and $\eta:\R^d\rightarrow \C$ is a given spatial potential.

	
	The equation \eqref{eq:NLS} has a rich physical background, arising in the mathematical description of phenomena in nonlinear optics and plasma physics. In particular, it is often referred to as the disordered Schr\"odinger equation, where $\eta(x)$ represents a given complex-valued, random, or sufficiently irregular potential. This formulation is closely related to Anderson localization \cite{Anderson-1958}, a phenomenon that has been extensively studied and applied in diverse areas, including Metal-Insulator Transition, superconductors, suppressing epileptic seizures and so on.

The aim of this paper is to explore the regularization mechanisms of Schr\"odinger equations with irregular potentials $\eta\in L_x^r+L_x^{\infty}$. The regularity of solutions is a central issue in the study of the nonlinear dispersive equations when the potential is irregular. This regularity reveals how the interplay between nonlinearity, smooth initial conditions, and the roughness of the potential affects the localization phenomena. Moreover, as pointed out in \cite{M-Wu-Z24}, the regularity properties are essential for the design and analysis of numerical schemes of approximating solutions, where the smoothness ensures the convergence and accuracy of computational methods.
	
	
	
The equation \eqref{eq:NLS} exhibits two types of critical indices that play a fundamental role in the analysis of the well-posedness/regularization.

{\bf{Critical index for the potential}}.
Note that the class of solutions to equation \eqref{eq:NLS} is invariant under the scaling
\begin{align*}
u(t,x)\rightarrow u_{\lambda}(t,x)=u(\lambda^2t,\lambda x),
\end{align*}	
\begin{align*}
\eta(x)\rightarrow\eta^\lambda(x)= \lambda^2\eta(\lambda x),
\end{align*}
with $\lambda>0$, which maps the initial data
\begin{align*}
u(0)\rightarrow u_\lambda(0):=u_0(\lambda x).
\end{align*}
This scaling leaves the $L_x^{\frac d2}$-norm of the potential $\eta$ invariant, that is,
\begin{align*}
	\big\|\eta^\lambda\big\|_{L^{\frac d2}(\R^d)}=\big\|\eta\big\|_{L^{\frac d2}(\R^d)}.
\end{align*}
Hence, the space $L^{\frac d2}(\R
	^d)$ is the critical space for the potential in the sense of scaling. Accordingly, for $\eta\in L_x^r+L_x^{\infty}(\R^d)$, we call the potential $\eta$ supercritical, critical and subcritical, if $r<\frac d2$, $r=\frac d2$, and $r>\frac d2$, respectively. This suggests that the problem \eqref{eq:NLS} is ill-posed for some $\eta$ lying in the supercritical region, that is, $\eta\in L_x^r+L_x^{\infty}(\R^d)$, $r<\frac d2$. This will be rigorously established below.

{\bf{Critical index for the regularity}}.
The second critical index concerns the regularity of the solution for a fixed potential $\eta$. For $\eta\in L_x^r, r\in [1,2)$ and smooth function $f$, one expects the best that
$$
\eta f\in L_x^r(\R^d), \quad \mbox{ or }\quad \eta f\in H_x^{\alpha_*}(\R^d),\quad \alpha_*= \alpha_*(r)=\frac d2-\frac dr.
$$
Considering the inhomogeneous Schr\"odinger equation,
$$
iu_t+\Delta u=F,
$$
with $F\in L^\infty_t H^{\alpha_*}_x$, then  the corresponding expected optimal regularity of the solution is $\gamma_*=\gamma_*(r)=\alpha_*(r)+2=2+\frac d2-\frac dr$. This suggests that for $L_x^r$-potential, the best expected regularity of the solution to \eqref{eq:NLS} is $H_x^{\gamma_*}$. This also will be rigorously proved in the following.

The general form of \eqref{eq:NLS} is the nonlinear Schr\"odinger equation
\begin{equation}\label{eq:NLS-2}
		\left\{ \aligned
		&i\partial_t u(t, x)+\Delta u(t, x)+\eta(x) u(t,x)=\lambda |u(t,x)|^{p} u(t,x),
		\\
		&u(0,x)=u_0(x),
		\endaligned
		\right.
	\end{equation}
where $u(t, x):\R^+\times \R^d\rightarrow \C$ and $\lambda \in \R$.
In this case, an additional critical index arises from the nonlinearity. Without the potential term, there is a critical index $s_c=\frac d2-\frac 2p$, which reads the {\bf{scaling critical index}}. This index arises from the invariance of the $\dot{H}_x^{s_c}$ norm under the scaling transformation,
\begin{align*}
u(t,x)\rightarrow u_{\lambda}(t,x)=\lambda^{\frac 2p}u(\lambda^2t,\lambda x), \mbox{ for }\lambda>0.
\end{align*}
Combining this with the critical indices above, for $\eta\in L_x^r+L_x^{\infty}, r\ge \frac d2$, the best expectation for well-posedness is
$$
 u\in H_x^s(\R^d),\quad \mbox{max}\{s_c, 0\}\le s\le \gamma_*.
$$
This will be further discussed in Section \ref{sec:114}.

We nextly briefly review some existing results on the regularization theory for nonlinear Schr\"odinger (NLS) equation \eqref{eq:NLS-2}. When $\eta$ is random or sufficiently rough, the regularization theory for the equation \eqref{eq:NLS-2} remains underdeveloped, with only a few notable results. The most relevant findings are summarized as follows.
 Cazenave \cite{Cazenave--03} proved that if $\eta\in L^{\infty}$ is real-valued, the equation \eqref{eq:NLS-2} is globally well-posed in $H_x^1(\R^d)$ for small data, where $d\geq 1$. In the same work, Cazenave also established the local well-posedness in $H_x^2(\R^d)$ when $\eta \in L_x^2+L_x^{\infty}(\R^d)$ and $d\geq 1$\footnote{This result is not true for $d>4$, see Theorem \ref{theorem-main 1} below for further details.}.

 In the recent work \cite{M-Wu-Z24}, Mauser, Zhao, and the third author considered the periodic case in one dimension. Their results represent the first sharp well-posedness results for this model. Based on the theoretical theory, the authors designed a
low-regularity integrator tailored to rough potentials, for which they proved convergence
rates with sharp regularity dependence.


Moreover, in \cite{Bai-Lian-Wu-2024}, we studied the equation \eqref{eq:NLS} with potentials $\eta \in L_x^r+L_x^{\infty}(\R)$ for any $r \ge 1$. Specifically, we proved that the equation is globally well-posed in $H^{\tfrac{3}{2}-}(\R)$ when $r=1$; $H^{\tfrac{5}{2}-\tfrac{1}{r}}(\R)$ when $1<r \le 2$, and $H^2(\R)$ when $r>2$,
while in each case there exists some $\eta$ such that it becomes ill-posed in the corresponding space $H^{\tfrac{3}{2}}(\R)$, $H^{\tfrac{5}{2}-\tfrac{1}{r}+}(\R)$, and $H^{2+}(\R)$, respectively. Here and below, we denote $a\pm:=a\pm\epsilon$ for arbitrary small $\epsilon>0$. The analysis relies on commutator estimates, local smoothing effects, and the normal form method.


There are also related results for the stochastic NLS equations.
 For temporally stochastic but spatially regular potentials, Bouard and Debussche \cite{Bouard-De-99} studied the stochastic NLS with a multiplicative noise and demonstrated that for some subcritical nonlinearities, the $L^2(\R^d)$ solution is almost surely global and unique, using the fixed point argument.
	For spatial white noise potentials, Debussche and Weber \cite{De-We-19} proved that the defocusing NLS equation \eqref{eq:NLS-2} with smooth initial data has a global solution almost surely in $H^1(\T^2)$. They also proved that the focusing NLS equation \eqref{eq:NLS-2} has the same result under the additional smallness condition, which is based on a renormalization of this equation and the conserved quantities. Later, Debussche and Martin \cite{De-Ma-19} applied the methods from \cite{De-We-19} to study the subcritical defocusing NLS equation with spatial white noise on the full space $\R^2$. They obtained that if $p<2$, this equation has a local solution almost surely in some weighted Besov space, and if $p<1$, the solution is global. Furthermore, the interesting work by Babin, Ilyin and Titi \cite{Babin-Ilyin-Titi-2011} established the unconditional well-posedness results for the periodic KdV equation in $\dot{H}^s$, $s\geq 0$, which provided a new insight into regularization mechanisms for nonlinear dispersive partial differential equations (PDEs) in the periodic setting.

	This work continues \cite{Bai-Lian-Wu-2024, M-Wu-Z24} by extending the analysis of the regularization mechanisms for \eqref{eq:NLS} to $\R^d$ with $d\ge 2$. The one-dimensional case can be handled via commutator estimates and the local smoothing effect; in higher dimensions, however, the markedly weaker local smoothing renders this approach ineffective.
	
	\subsection{Main results}
Before presenting our main results, we give the definitions of well-posedness and ill-posedness.
	\begin{definition}[Well-posedness]\label{Def1}
		The well-posedness of a time dependent PDE can be defined as follows: Denote by $C(I, X_0)$ the space of continuous functions from the time interval $I$ to the topological space $X_0$. We say that the Cauchy problem is locally well-posed in $C(I, X_0)$ if the following properties hold:
		
		(1) For every $u_0\in X_0$, there exists a strong solution defined on a maximal time interval $I=[0, T_{max})$, with $T_{max}\in (0, +\infty]$.

		(2) There exists some auxiliary space $X$, such that strong solution to this problem is unique in $C(I, X_0)\cap X$ .

		(3) The solution map $u_0\mapsto u[u_0]$ is continuous from $X_0$ to $X_0$.
	\end{definition}
	If any above condition fails, the Cauchy problem \eqref{eq:NLS} is said to be ill-posed in space $X_0$. In this work, we refer to the violation of the third condition (around zero solution).
	
More precisely, let $\Phi_t: X \to X$ be the solution flow map of a Cauchy problem in the function space $X$. We say the problem is ill-posed in $X$ if the flow map $\Phi_t$ fails to be continuous at some point $u_0 \in X$. Equivalently, there exist $u_{0,n} \to u_0$ in $X$ and a time $t>0$ such that $\Phi_t(u_{0,n}) \not\to \Phi_t(u_0)$ in $X$.

	We now turn to the well-posedness/regularity results for equation \eqref{eq:NLS} when $\eta \in L_x^{r}+L_x^{\infty}(\R^d)$. Our first task is to determine the admissible range of
$r$ for which well-posedness can be expected. It is straightforward to show that there exists some $\eta\in L^r+ L_x^{\infty}(\R^d)$ with $1\leq r<\frac d2$, then the equation is ill-posed in $H^\gamma(\R^d)$ for any $\gamma\in \R$. This fact reflects the supercritical nature of such potentials under scaling. The precise result is given in the following theorem.

%
%
	\begin{thm}\label{theorem-main 1}

Let $d\geq 3$, $1\leq r<\frac d2$, there exists some $\eta  \in L_x^{r}+L_x^{\infty}(\R^d)$ such that for any $\gamma\in \R$, \eqref{eq:NLS} is ill-posed in $H_x^{\gamma}(\R^d)$ .

	\end{thm}

%
%
%

We make the following remarks concerning the above result.

 \begin{remark}

The condition $r\geq \frac d2$ is the natural regime for proving well-posedness of \eqref{eq:NLS}, which matches to the \textit{critical index for the potential} discussed as above.

\end{remark}
Next, we present the well-posedness results for \eqref{eq:NLS} under the subcritical and critical potentials. For convenience, we denote the best expected regularity index
$$
\gamma_*(r)=2+\frac d2-\frac dr,
$$
which we abbreviate as $\gamma_*$.
In what follows, we define the ``sharp well-posedness in $H_x^\gamma$ for $\eta \in Y_0$ (some spatial function space)'' to mean that the problem is well-posed in $H_x^\gamma$ for any $\eta \in Y_0$, but ill-posed in $H_x^{\gamma+}$ for some $\eta \in Y_0$.
\begin{thm}\label{theorem-one}
		Let $d=2, 3$, $\frac d2<r\leq 2$, and $\eta \in L_x^{r}+L_x^{\infty}(\R^d)$, then \eqref{eq:NLS} is sharp globally well-posed in $H_x^{\gamma_*}(\R^d)$.

	\end{thm}

\begin{thm}\label{theorem-two}
		Let $d=3, 4$, $r=\frac d2$, and $\eta \in L_x^{r}+L_x^{\infty}(\R^d)$, then \eqref{eq:NLS} is sharp globally well-posed in $H_x^{\gamma_*-}(\R^d)$.

	\end{thm}

\begin{thm}\label{theorem-three}
		Let $d\geq 2$, $r\geq\frac d2$ and $r>2$, and $\eta \in L_x^{r}+L_x^{\infty}(\R^d)$, then \eqref{eq:NLS} is sharp globally well-posed in $H_x^{2}(\R^d)$.
\end{thm}

In summary, our results can be shown in the following table

\begin{table}[h]
\centering
\setlength{\tabcolsep}{8pt}
\renewcommand{\arraystretch}{1.3}
\begin{tabular}{|c|c|c|c|c|}
\hline
\multirow{2}{*}{ $d=2$}
  & \multicolumn{3}{c|}{ $1<r\leq 2$} &  $r>2$ \\ \cline{2-5}
& \multicolumn{3}{c|}{ $H^{3-\frac 2r}$} &  $H^2$\\ \hline

\multirow{2}{*}{ $d=3$}
  &  $1<r<\frac 32$ &  $r=\frac 32$ & $\frac 32<r\leq 2$ &  $r>2$ \\ \cline{2-5}
&  $\mbox{Ill-posedness}$ & $H^{\frac 32-}$ & $H^{\frac 72-\frac 3r}$ &  $H^2$ \\ \hline

\multirow{2}{*}{ $d=4$}
  & \multicolumn{2}{c|}{ $1<r<2$} &  $r=2$ &  $r>2$ \\ \cline{2-5}
& \multicolumn{2}{c|}{ $\mbox{Ill-posedness}$} &  $H^{2-}$ &  $H^2$\\ \hline

\multirow{2}{*}{ $d\geq 5$}
  & \multicolumn{3}{c|}{ $1<r<\frac d2$} &  $r\geq \frac d2$\\ \cline{2-5}
& \multicolumn{3}{c|}{ $\mbox{Ill-posedness}$} &  $H^2$ \\ \hline
\end{tabular}
\end{table}

Taken together, the above results leave only one unsolved case: $d=2$ and $r=1$. We conjecture that ill-posedness occurs in this setting.
	\begin{conjecture}
		For any $\gamma\in \R$, there exists $\eta \in L_x^1+L_x^{\infty}(\R^2)$ such that \eqref{eq:NLS} is ill-posed in $H_x^{\gamma}(\R^2)$.
	\end{conjecture}
	The most probable reason supporting this conjecture is the failure of the endpoint Strichartz estimate in $L^2_tL^\infty_x$ for two dimensions. At present, due to technical limitations, we are unable to resolve this problem, which appears to be substantially more difficult than the one-dimensional case.

\begin{remark}

\begin{enumerate}

\item By the above several theorems, we observe that the regularity of the solution is essentially determined by
 the regularity of the potential. Moreover, increasing the integrality of the potential $\eta$ leads to a corresponding increase in the regularity of the solution.
However, once the integrality of $\eta$ reaches a certain threshold, further improvements in $\eta$ no longer translate into higher regularity of the solution.

\item A typical example for the potential $\eta \in L_x^r+L_x^{\infty}$ is
 $$\eta=|x|^{-a}\in L_x^r+L_x^{\infty}, \quad 0\leq a<\frac dr.$$
%
\end{enumerate}
\end{remark}

	\subsection{A discussion on the effects of nonlinearity} \label{sec:114}
We now briefly discuss the effect of nonlinearities on the regularity of solutions to equation \eqref{eq:NLS}. As observed, the term $\eta u$ and the nonlinear term $|u|^{p} u$ interact with each other, influencing the regularity of the solution. Specifically, we consider the following classical nonlinear Schr\"odinger equation,
	\begin{equation}\label{eq:C-NLS}
		\left\{ \aligned
		&i\partial_t u(t, x)+\Delta u(t, x)=\lambda |u(t,x)|^{p} u(t,x), \quad (t, x)\in \R\times \R^d,
		\\
		&u(0,x)=u_0(x),
		\endaligned
		\right.
	\end{equation}
whose scaling critical index is given by
$
s_c=\frac d2-\frac 2p.
$
 By the work of Cazenave and Weissler \cite{Cazenave-Weissler-1990}, the equation \eqref{eq:C-NLS} is locally well-posed in $H_x^s(\R^d)$, for $s\geq s_c$. If $s<s_c$, the equation \eqref{eq:C-NLS} is ill-posed in $H_x^s$, see the ill-posedness results in \cite{C-C-T-03, KenigPonceVega-CPAM-01}.

Next, we summarize the well-posedness results for the equation \eqref{eq:NLS-2} with potential $\eta\in L_x^r+L_x^{\infty}$. Recall that the equation \eqref{eq:NLS-2} is
\begin{equation}\label{eq:NLS-2-2}
		\left\{ \aligned
		&i\partial_t u(t, x)+\Delta u(t, x)+\eta(x) u(t,x)=\lambda |u(t,x)|^{p} u(t,x),
		\\
		&u(0,x)=u_0(x),
		\endaligned
		\right.
	\end{equation}
where the sign of $\lambda$ does not affect local well-posedness. By combining the known well-posedness results for the classical NLS equation \eqref{eq:C-NLS} with our main theorems in this paper, we obtain the following claim without proof.

{\bf Claim:} 
Let $r\geq \frac d2$, $(r, d)\neq (1, 2)$, and $\epsilon $ be arbitrary small positive constant. Denote
$$
\tilde{\gamma}_*:=\min\{2, 2+\frac d2-\frac dr+o_-\},
$$
where 	$$
	o_-=\left\{ \aligned
	&-\epsilon, \quad r=\frac d2,
	\\
	&0, \quad r>\frac d2.
	\endaligned
	\right.
	$$
The following statements hold:

(1) If
$\mbox{max}\{s_c, 0\}\leq s\leq  \tilde{\gamma}_*,\mbox{ and } s<p+1,$
then \eqref{eq:NLS-2-2} is locally well-posed in $H_x^s$;

(2) If $\tilde{\gamma}_*<s_c$, then for any $\gamma \in \R$, there exists some $\eta \in L_x^r+L_x^{\infty}$ such that \eqref{eq:NLS-2-2} is ill-posed in $H_x^{\gamma}$.

Case (2) exactly exists, for instance, when $\frac d2\leq r\leq\infty$, $p=2$, and $d>6$.
%
%
%
%
The proof of this claim is rather direct relying on the fractional chain rule (see Proposition A.1. in \cite{Visan-2007}), the standard techniques employed in the well-posedness theory for the classical NLS, and the arguments developed in this paper. Moreover, if the potential $\eta$ is real-valued, we further assert that the equation \eqref{eq:NLS-2-2} is globally well-posed in the aforementioned space $H_x^s$, as such a potential generally does not influence the global well-posedness in this setting.

	\subsection{Difficulty, novelty, ideas of proof.} \label{sec:1.2}

In establishing the well-posedness of the equation \eqref{eq:NLS} with a rough potential, the primary difficulty is that we cannot impose any derivative on the potential. In fact,
$$
|\nabla|^\alpha\big(\eta u\big),\quad\alpha>0,
$$
is not well-defined when $\eta$ belongs merely to $L^r+ L_x^{\infty}(\mathbb{R}^d)$ with $r \ge 1$. Moreover, the usual Strichartz estimates provide no global smoothing effect for general initial data. This same issue, commonly referred to as a ``loss of derivative'', has also been encountered in earlier studies on the torus $\mathbb{T}$ and on the one-dimensional line $\mathbb{R}$ (see \cite{Bai-Lian-Wu-2024, M-Wu-Z24}). To overcome this difficulty, we introduce several novel techniques, described below. We now outline the key ideas and observations of the proof for the case $\frac{d}{2} < r \le \max\{2,\,\frac{d}{2}\}$ with $d = 2,3$.


	
	$\bullet$ \textit{Better performance of the time derivative $\partial_t u$.} We consider the equation for $v$ (the time derivative of $u$) instead of directly working with $u$. This is motivated by the observation that $v$ has significantly better space-time properties than $\Delta u$.
More precisely, for some $p > 2$, one can show
$$
|\nabla|^{\gamma_*-2}v\in L^{q}_tL^{p}_x, \quad \text{whereas} \quad  |\nabla|^{\gamma_*}u\notin L^{q}_tL^{p}_x.
$$
Here $\gamma_*=2+\frac d2-\frac dr$.
Indeed, the first one has been proved in Section \ref{v-spacetime-830}. If the second one is true, then from the relation between $v$ and $u$,
\EQn{\label{v-u-relation}
 u=i (-\Delta)^{-1} v+(-\Delta)^{-1}(\eta u),
 }
it follows that $|\nabla|^{\gamma_*-2}(\eta u)\in L^{q}_tL^{p}_x$. However, this is against the Sobolev embedding due to the low regularity of the potential $\eta$. In fact, one can find some initial data $u_0$ and potentials $\eta_N$ such that
$$
\big\||\nabla|^{\gamma_*-2}(\eta_N e^{it \Delta } u_0)\big\|_{L^{q}_tL^{p}_x}\rightarrow +\infty,   \mbox{ as } N\rightarrow \infty.
$$
This indicates that obtaining space-time bounds for $u$ alone necessitates reducing the regularity index $\gamma_*$. Consequently, directly analyzing the equation for $u$ may not achieve the desired regularity. Therefore, we turn to the equation satisfied by $v$,
\begin{align}\label{v-equation}
iv_t+\Delta v+\eta v=0, \mbox{ with } v_0\in H^{\gamma_*-2},
\end{align}
and carry out the analysis in the space $L^\infty_t H^{\gamma_* - 2}_x \cap X$, where $X$ is an appropriate auxiliary space. This approach ultimately allows us to improve the regularity of $u$ in the $H^{\gamma_*}$ norm.

$\bullet$ \textit{Loss of derivative and regularity gain.} As described above, the term $\langle \nabla\rangle^\gamma(\eta u)$ is not well-defined for $\gamma>0$ when $\eta$ is only in $L^r_x+ L_x^{\infty}$ with $r\ge 1$. We must therefore recover this lost regularity by fully exploiting the dispersive properties and smoothing effects from the structure of the equation. To do so, we employ the normal form method introduced by Shatah \cite{Sha85CPAM} and also the differentiation by parts from Babin, Ilyin and Titi \cite{Babin-Ilyin-Titi-2011} to compensate for the derivative loss.
The most complex scenario is as follows:
	\EQ{
		\int_{0}^{t}e^{i(t-s)\Delta}\langle\nabla \rangle^{\gamma}\big(\eta_H u_L\big)ds,
	}
	where ``H" and ``L" indicate the high- and low-frequency components, respectively. The above integral can be rewritten as
	\EQ{
		e^{it\Delta}\F_\xi^{-1}\int_{0}^{t} \int_{\xi=\xi_1+\xi_2}e^{is\Phi}\langle\xi\rangle^{\gamma}\widehat{\eta_{H}}(\xi_1)\widehat{\tilde{u}_{L}}(s,\xi_2)d\xi_1ds,
	}
	where the profile $\tilde{u}:=e^{-it\Delta}u$, and the phase function
	\EQ{
		\Phi(\xi, \xi_2):=|\xi|^2-|\xi_2|^2.
	}
	This integral is temporal non-resonant in the sense that $\Phi \thicksim |\xi_1|^2$. Integrating by parts in $s$ yields several terms in the form of a multilinear operator (called the normal form transform):
	\EQn{\label{T-D}
		T(\langle\nabla\rangle^{\gamma}\eta_{H},u_{L})
		:=\F_\xi^{-1}\int_{\xi=\xi_1+\xi_2}\frac{1}{i\Phi} \widehat{\langle\nabla\rangle^{\gamma}\eta_{H}}(\xi_1)\widehat{u_{L}}(\xi_2)d\xi_1d\xi_2.
	}
Heuristically,
	\EQ{
		T(\langle\nabla\rangle^{\gamma}\eta_{H},u_{L}) &\sim \jb{\nabla}^{-2+\gamma} \eta_{H}\cdot u_{L},
	}
	where $\jb{\nabla}^{-2}$ is derived due to the factor $1/\Phi$, which compensates for the derivative loss $\jb{\nabla}^{\gamma}$, for some $0<\gamma\leq \gamma_*$.

However, the normal form method above breaks down when we consider the equation for $v$ in $H^{\gamma_*-2}_x$ with $\gamma_* < 2$. In this setting, we turn to examine the Duhamel term
\begin{align}\label{Duhamel-v}
I\triangleq\int_0^te^{i(t-s) \Delta}(\eta v)d s.
\end{align}
If we naively apply the standard normal form method repeatedly to $I$, we roughly obtain
the following route of the transformed nonlinearity:
\begin{align}\label{route-CF}
(\eta v_H)_L &\longrightarrow \big(\eta \langle\nabla\rangle^{-2}\big(\eta v_H\big)\big)_L\notag\\
&\longrightarrow \big(\eta \langle\nabla\rangle^{-2}\big(\eta\langle\nabla\rangle^{-2}\big(\eta v_H\big)\big)\big)_L\notag\\
&\longrightarrow \big(\eta \langle\nabla\rangle^{-2}\big(\eta \cdots \langle\nabla\rangle^{-2}\big(\eta v_H\big)\big)\big)_L.
\end{align}
This is probably impossible to close in $H_x^s$ for $s<0$, because the last $v_H$ cannot obtain any negative derivative. In other words, this direct normal form method fails to control the solution in the negative Sobolev space $H^{\gamma_* - 2}_x$.


%
%

$\bullet$ \textit{Iterative Duhamel construction.} As noted above, one cannot directly deduce $$v \in H^{\gamma_* - 2}_x \implies I \in H^{\gamma_* - 2}_x.$$ To overcome this, we develop the ``Iterative Duhamel construction". Specifically, denote the partial sum $$S_N := \sum_{n=0}^N e^{it\Delta} I_n,$$
where the terms $I_n$ are defined recursively by
$$
I_n = i \int_0^t e^{-i\rho \Delta} \Big(\eta\, e^{i\rho \Delta} I_{n-1}\Big)\, d\rho, \quad n \ge 1; \qquad I_0 = v_0.
$$
We shall prove that $S_N$ converges to the unique solution to equation (1.2) in the space $H^{\gamma_* - 2}_x$. While Bejenaru and Tao \cite{B-T-06} established the statement assuming quantitative well-posedness of the equation, we eliminate this hypothesis by the following three steps:

\begin{enumerate}
\item $S_N \in L^\infty_t H^{\gamma_* -2}_x$ for each $N$;\\

\item $\|S_N - S_{N'}\|_{L^\infty_t H^{\gamma_* - 2}_x} \to 0$ as $N,N' \to \infty$;\\

\item $S_N$ converges to the unique solution to equation (1.2) in a weaker space $Y \supset L^\infty_t H^{\gamma_* - 2}_x$.
\end{enumerate}
(1) and (2) guarantee the convergence of $S_N$. (3) shows that the limit of $S_N$ is exact the solution $v$ to equation \eqref{v-equation}.

The key estimate for proving (1) and (2) is that there exists some $\theta > 0$ such that for all $n \ge 1$,
\begin{align}\label{key-e}
\big\|\,e^{it\Delta} I_n \big\|_{H^{\gamma_* - 2}_x \cap X} \;\lesssim\; \sum_{k=1}^n 2^{-\theta k}\, \big\|\,e^{it\Delta} I_{\,n-k} \big\|_{H^{\gamma_* - 2}_x \cap X}.
\end{align}
This estimate is based on the observation that
$$
I_n=\sum_{k=1}^n (T_N)^kJ_{n-k},
$$
where $J_{n-k}$ has a similar form of $I_{n-k}$, where $T_N$ is a pseudo-differential operator of order $2$ with negative principal symbol $|\xi|^{-2}$.

By induction on $n$, \eqref{key-e} yields
$$
\|I_n\|_{H^{\gamma_* - 2}_x} \;\lesssim\; 2^{-\theta n}\,\|v_0\|_{H^{\gamma_* - 2}_x}.
$$
For step (3), the crucial ingredient is that we can initially prove a weak result that $$u \in L^\infty_t H^\alpha_x,$$ for some $\alpha < \gamma_*$ defined in Section \ref{low-one}. 

$\bullet$ \textit{Splitting normal form method.}
In deriving the estimate \eqref{key-e} for $n\ge 3$ via an iteration analogous to \eqref{route-CF}, one encounters terms of the form
\EQ{
  \F_\xi^{-1}\int_{\xi=\sum_{j=1}^4\xi_j}
  \frac{1}{i\Phi(\xi,\xi-\xi_1)}\frac{1}{i\Phi(\xi,\xi_3+\xi_4)}\,
  \widehat{\eta}(\xi_1)\widehat{\eta}(\xi_2)\widehat{\eta}(\xi_3)\widehat{e^{is\Delta}I_{n-3}}(\xi_4)\,
  d\xi_1\,d\xi_2\,d\xi_3.
}
The associated multiplier
\[
\frac{1}{\Phi(\xi,\,\xi-\xi_1)} \cdot \frac{1}{\Phi(\xi,\,\xi_3+\xi_4)}
\]
is too intricate for a direct application of the Coifman--Meyer multiplier theorem without incurring a derivative loss. Although each factor separately behaves like $|\nabla|^{-2}$, the coupling between the terms $\xi-\xi_1$ and $\xi_3+\xi_4$ prevents one from simultaneously converting both factors into derivatives and makes it, when regarded as a whole multiplier,  not satisfy the conditions of the Coifman--Meyer multiplier theorem. Consequently, the procedure in \eqref{route-CF} breaks down and the optimal regularity cannot be reached.

To overcome this difficulty, we split the phase in the normal form step. Rather than using the standard integration-by-parts identity
\[
e^{is\phi}=\partial_{s}(e^{is\phi})\frac{1}{i\phi},
\]
we split the phase function $\phi$ as $\phi = \phi_1 + \phi_2$, and in the regime $|\phi_1| \ll |\phi_2|$ apply the following formula
\begin{align}\label{SNM}
e^{is \phi}=e^{is \phi_1+is \phi_2}=\partial_{s}(e^{is \phi_2})\frac {e^{is \phi_1}}{i\phi_2}.
\end{align}
This splitting averages out the prospective derivative loss: unlike in the standard normal form, the factor $1/\phi_2$ may be regarded directly as $|\nabla|^{-2}$ without invoking Coifman--Meyer multiplier theorem. However, every time we perform integration by parts, an additional term is produced. For instance, applying the ``splitting normal form'' once via \eqref{SNM} to $I_n$ yields
\begin{align}\label{1-SNM}
(\eta e^{is\Delta}I_{n-1, H})_L \;\longrightarrow\;
|\nabla|^2\!\big(\eta |\nabla|^{-2}e^{is\Delta}I_{n-1}\big)
\;+\;
\big(\eta |\nabla|^{-2}\big(\eta e^{is\Delta}I_{n-2,H}\big)\big)_L.
\end{align}
The first term is readily controlled by Schur's test, while the second is handled by iterating the splitting step.
Iterating produces a cascade analogous to \eqref{route-CF}:
\begin{align}\label{route-SNM}
(\eta e^{is\Delta}I_{n-1, H})_L
&\longrightarrow \big(\eta |\nabla|^{-2}\big(\eta e^{is\Delta}I_{n-2,H}\big)\big)_L \notag\\
&\longrightarrow \big(\eta |\nabla|^{-2}\big(\eta|\nabla|^{-2}\big(\eta e^{is\Delta}I_{n-3,H}\big)\big)\big)_L \notag\\
&\longrightarrow \big(\eta |\nabla|^{-2}\big(\eta \cdots |\nabla|^{-2}\big(\eta e^{is\Delta}I_0\big)\big)\big)_L.
\end{align}
This iterative splitting normal form controls the derivative loss and thereby yields \eqref{key-e}.

	\subsection{Organization of the paper} The rest of the paper is organized as follows. In Section 2, we give some basic notations, and lemmas that will be used in this paper. In Section 3, the ill-posedness in any Sobolev space for supercritical potentials is established. Section 4 is devoted to the resonant and non-resonant decomposition of the Duhamel term based on the normal form transform. Sections 5, 6, 7 are devoted to the proof of Theorems \ref{theorem-one}, \ref{theorem-two}, and \ref{theorem-three}, respectively.


	\section{Preliminary}\label{sec:notations}
	
	\subsection{Notations}
	For any $a\in \mathbb{R}$, $a\pm:=a\pm\epsilon$ for arbitrary small $\epsilon>0$. For any $z\in \mathbb{C}$, we define $\mbox{Re}z$ and $\mbox{Im}z$ as the real and imaginary part of $z$, respectively. $|\nabla|^{\alpha}=(-\Delta)^{\frac {\alpha}2}$. $\langle \cdot\rangle =(1+|\cdot|^2)^{\frac 12}$. We write $X \lesssim Y$ or $Y \gtrsim X$ to indicate $X \leq CY$ for
	some constant $C>0$. If $X \leq CY$ and $Y \leq CX$, we write $X\sim Y$. If $X\leq 2^{-5}Y$, we denote $X\ll Y$ or $Y\gg X$. Throughout the whole paper, the letter $C$ will denote suitable positive constant that may vary from line to line. Moreover, we use ``\emph{R.H.S} of $(\cdot)$'' to represent the part on the right-hand side of $(\cdot)$.
	
We use the following norm to denote the sum of two Banach spaces $X_1$ and $X_2$,
\EQ{
\norm{u}_{X_1+X_2}=\inf\{\norm{u_1}_{X_1}+\norm{u_2}_{X_2}: u_1\in X_1, u_2\in X_2, u=u_1+u_2\}.
}

We also use the following norm to denote the mixed spaces $L_t^qL_x^r(I\times \R^d)$,
	\begin{align*}
		\|u\|_{L_t^qL_x^r(I\times \R^d)}=\Big(\int_I \|u\|_{L_x^r( \R^d)}^qdt\Big)^{\frac{1}{q}}.
	\end{align*}
For simplicity, we often use $L_t^{q}L_x^r$ to denote $L_t^{q}L_x^r(I\times \R^d)$; if the time interval $I$ needs to be emphasized, we specify it as $L_t^{q}L_x^r(I)$ instead.

We use $\widehat{f}$ or $\mathscr{F}f$ to denote the Fourier transform of $f$:
$$
\mathscr{F}f (\xi)= \widehat f(\xi)
	= \int_{\R^d} e^{- i x\cdot \xi} f(x) dx.
$$
We also define
$$
\mathscr{F}^{-1}g (x)= \int_{\R^d} e^{ i x\cdot \xi} g(\xi)d\xi.
$$
	The Hilbert space $H^s(\R^d)$ is a Banach space of elements such that
	$\mathscr \langle\xi\rangle^s \widehat u  \in L^2(\R^d)$, and equipped with the norm $\|u\|_{H^s}= \|\langle\xi\rangle^s  \widehat{ u}  (\xi)  \|_{L^2}$. We also have an embedding inequality
	that $\|u\|_{H^{s_1}}\lesssim \|u\|_{H^{s_2}}$ for any $s_1 \leq s_2$, $s_1, s_2\in \R$.

We take a cut-off function $\chi_{a\leq |\cdot| \leq b}(x) \in C_c^{\infty}(\R^d)$ for $b>a>\frac 14$ such that
	$$
	\chi_{a\leq |\cdot| \leq b}(x) =\left\{ \aligned
	&1, \quad a\leq |x| \leq b,
	\\
	&0, \quad |x|\leq a-\frac 14 \mbox{ or } |x|\geq b+\frac 14.
	\endaligned
	\right.
	$$
	 We take a cut-off function $\phi \in C_c^{\infty}(0, \infty)$ such that
	$$
	\phi(r)=\left\{ \aligned
	&1, \quad r\leq 1,
	\\
	&0, \quad r\geq 2.
	\endaligned
	\right.
	$$

Next, we give the definition of Littlewood-Paley dyadic projection operator.
For dyadic number $N\in 2^{\mathbb{N}}$, when $N\geq 1$, let $\phi_{\leq N}(r)=\phi(N^{-1}r)$. Then, we define
$\phi_{1}(r):=\phi(r)$, and $\phi_{N}(r)=\phi_{\leq N}(r)-\phi_{\leq \frac N2}(r)$ for any $N\geq 2$. We define the inhomogeneous Littlewood-Paley dyadic operator
$$
f_1=P_{ 1}f:=\mathscr{F}^{-1}(\phi_{ 1}(|\xi|)\widehat{f}(\xi)),
$$
and for any $N\geq 2$,
$$
f_{ N}=P_{ N}f:=\mathscr{F}^{-1}(\phi_{ N}(|\xi|)\widehat{f}(\xi)).
$$
Then, by definition, we have $f=\sum_{N\in 2^{\N}}f_N$. Moreover, we also define the following:
	\begin{align*}
	&f_{\leq N}=P_{\leq N}f:=\mathscr{F}^{-1}(\phi_{\leq N}(|\xi|)\widehat{f}(\xi)),\\
	&f_{ \ll N}=P_{ \ll N}f:=\mathscr{F}^{-1}(\phi_{ \leq  N}(2^5|\xi|)\widehat{f}(\xi)),\\
&f_{\lsm N}=P_{\lsm N}f:=\mathscr{F}^{-1}(\phi_{ \leq N}(2^{-5}|\xi|)\widehat{f}(\xi)).
	\end{align*}
	We also define that $f_{\geq N}=P_{\geq N}f:=f-f_{\leq N}$, $f_{\gg N}=P_{\gg N}f:=f-P_{\lsm N}f$, and $f_{\gtrsim N}=P_{\gtrsim N}f:=f-P_{ \ll N}f$.

	Next, we show the Triebel-Lizorkin Spaces $F_{p}^{\alpha, q}$ with the corresponding norm as follows,
	\EQ{
		\|u\|_{F_{p}^{\alpha, q}}=\|u\|_{L_x^p}+\|N^{\alpha}P_Nu\|_{ L_x^p l_{N\in 2^{\mathbb{N}}}^q}.
	}
	For any $1\leq p <\infty$, we define $l_{N}^p=l_{N\in 2^{\mathbb{N}}}^p$ by its norm,
	$$
	\|c_N\|_{l_{N\in 2^{\mathbb{N}}}^p}^p:=\sum_{N\in 2^{\mathbb{N}}}|c_N|^p.
	$$
	For $p =\infty$, we define $l_{N}^{\infty}=l_{N\in 2^{\mathbb{N}}}^{\infty}$ by its norm
	$$
	\|c_N\|_{l_{N\in 2^{\mathbb{N}}}^{\infty}}:=\sup_{N\in 2^{\mathbb{N}}}|c_N|.
	$$
	In this paper, we also use the following abbreviations
	$$
	\sum_{N\geq M}:=\sum_{N, M\in 2^{\mathbb{N}}: N\geq M}, \quad \sum_{N\gtrsim M}:=\sum_{N, M\in 2^{\mathbb{N}}: N\geq 2^{-5}M},\mbox{ and } \sum_{N\ll M}:=\sum_{N, M\in 2^{\mathbb{N}}: N\leq 2^{-5}M}.
	$$
Finally, we give the definition of the Schr\"odinger-admissible pair.
Let $d\geq 1$ and the pair $(q, r)$ satisfy
		\begin{align*}
			2\leq q, r\leq \infty  ,\quad \frac{2}{q}+\frac{d}{r}=\frac{d}{2},\quad and\quad (q,r,d)\ne(2,\I,2),
		\end{align*}
then we say that the pair $(q, r)$ is Schr\"odinger-admissible.

	\subsection{Basic lemmas}

	\quad In this section, we state some preliminary estimates that will be used in our later sections.
	Firstly, we introduce the following Bernstein estimates that will be used frequently.
	\begin{lem}[Bernstein estimates\label{lem:Bernstein}]
		For any $1\leq p \leq q \leq \infty$, $s\geq 0$, and $f\in L_x^p(\R^d)$,
		\begin{align*}
			\|P_{\geq N} f\|_{L_x^p(\R^d)}&\lesssim  N^{-s}\||\nabla|^{ s}P_{\geq N} f\|_{L_x^p(\R^d)},\\
			\||\nabla|^{ s}P_{\leq N} f\|_{L_x^p(\R^d)}&\lesssim N^s\|P_{\leq N} f\|_{L_x^p(\R^d)},\\
			\||\nabla|^{\pm s}P_N f\|_{L_x^p(\R^d)}&\sim N^{\pm s}\|P_N f\|_{L_x^p(\R^d)},\\
			\|P_{\leq N} f\|_{L_x^q(\R^d)}&\lesssim  N^{\frac dp-\frac dq}\|P_{\leq N} f\|_{L_x^p(\R^d)},\\
			\|P_N f\|_{L_x^q(\R^d)}&\lesssim  N^{\frac dp-\frac dq}\|P_N f\|_{L_x^p(\R^d)}.
		\end{align*}
	\end{lem}

	\begin{lem}[Schur's test\label{lem:Schur}]
		For any $a>0$, let sequences $\{a_N\}, \{b_N\}\in l_{N\in 2^{\N}}^2$, then we have
		\EQ{
			\sum_{ N\geq N_1}\Big({\frac{N_1}{N}}\Big)^a a_N b_{N_1}\lesssim\|a_N\|_{l_N^2}\|b_N\|_{l_N^2}.
		}
	\end{lem}
	Next, we recall the well-known Strichartz's estimates.
	
	\begin{lem}\label{lem:strichartz}
		(Strichartz's estimates, see \cite{Keel-Tao-1998, Cazenave--03, Strichartz-1977, Ginibre-Velo-1985}) Let $I\subset \R$ be a time interval. Let $(q_j, r_j), j=1,2,$ be Schr\"odinger-admissible,
		then the following statements hold:
		\begin{align}\label{1.2222}
			\|e^{it\Delta}f\|_{L_t^{q_j}L_x^{r_j}(I\times{\R^d})}\lesssim\|f\|_{L^2(\R^d)};
		\end{align}
		and
		\begin{align}\label{1.222234}
			\Big\|\int_0^t e^{i(t-s)\Delta}F(s)ds\Big\|_{L_t^{q_1}L_x^{r_1}(I\times{\R^d})}\lesssim \|F\|_{L_t^{q_2'}L_x^{r_2'}(I\times{\R^d})},
		\end{align}
		where $\frac{1}{q_2}+\frac{1}{q_2'}=\frac{1}{r_2}+\frac{1}{r_2'}=1$.
	\end{lem}

		We also need the following Littlewood-Paley theory, see Remark 2.2.2 in \cite{Gra-14}.
	\begin{lem}[Littlewood-Paley theory\label{lem:littlewood-Paley}]
		Let $1<p<\infty$, for any $\alpha\in \R$, we have
		\EQ{
			\|f\|_{F_{p}^{\alpha, 2}}\sim \|\langle\nabla\rangle^{\alpha}f\|_{L_x^p}.
		}
	\end{lem}

	\begin{lem}[Multilinear Coifman-Meyer multiplier estimates, see \cite{Co-Me-91}\label{lem:Coifman-Meyer}]
		Let the function $m$ on ${(\R^n)^k}$ be bounded and let $T_m$ be the corresponding m-linear multiplier operator on $\R^n (n\geq 1)$
		\begin{align*}
			T_m(f_1,\cdots, f_k)(x)=\int_{(\R^n)^k}m(\eta_1,\cdots,\eta_k)\widehat{f_1}(\eta_1)\cdots\widehat{f_k}(\eta_k)e^{ix\cdot(\eta_1+\cdots+\eta_k)}d\eta_1\cdots d\eta_k.
		\end{align*}
		If $L$ is sufficiently large and $m$ satisfies
		\begin{align*}
			\Big|\partial_{\eta_1}^{\al_1}\cdots\partial_{\eta_k}^{\al_k}m(\eta_1,\cdots,\eta_k)\Big|\lesssim_{\al_1,\cdots, \al_k}(|\eta_1|+\cdots+|\eta_k|)^{-(|\al_1|+\cdots+|\al_k|)},
		\end{align*}
		for multi-indices $\al_1,\cdots, \al_k$ satisfying $|\al_1|+\cdots+|\al_k|\leq L$. Then, for $1< p<\infty$, $1<p_1, \cdots, p_k\leq\infty$ and $\frac 1p=\frac{1}{p_1}+\cdots+\frac{1}{p_k}$, we have
		\begin{align*}
			\|T_m(f_1,\cdots, f_k)\|_{L_x^p}\leq C\|f_1\|_{L_x^{p_1}}\cdots\|f_k\|_{L_x^{p_k}}.
		\end{align*}
	\end{lem}
	
	The Coifman-Meyer Multiplier Theorem is reduced to the Mihlin-H\"{o}rmander Multiplier
Theorem when $k=1$ and $1<p<\infty$.
	
	In order to prove the ill-posedness results for the equation \eqref{eq:NLS}, we need the following lemma.
	\begin{lem}\label{ill-posed}(See \cite{B-T-06}).
		Consider a quantitatively well-posed abstract equation in spaces $D$ and $S$,
		\begin{align*}
			u=L(f)+N_k(u,\ldots, u),
		\end{align*}
		which means for all $f\in D$, $u_1, \ldots ,u_k\in S$ and for some constant $C>0$,
		\begin{align*}
			\|L(f)\|_{S}\leq C\|f\|_{D}, \quad \|N_k(u_1, \ldots ,u_k)\|_{S}\leq C\|u_1\|_{S}\ldots\|u_k\|_{S}.
		\end{align*}
		Here $(D, \|\|_{D})$ is a Banach space with initial data and $(S, \|\|_{S})$ is a Banach space of space-time functions. Define
		\begin{align*}
			A_1(f):=L(f), \quad A_n(f):=\sum_{n_1,\ldots ,n_k\geq 1, n_1+\ldots+n_k=n}N_k(A_{n_1}(f), \ldots ,A_{n_k}(f)), n>1.
		\end{align*}
		Then for some $C_1>0$, all $f, g\in D$ and all $n\geq 1$,
		\begin{align*}
			\|A_n(f)-A_n(g)\|_{S}\leq C_1^n\|f-g\|_{D}(\|f\|_{D}+\|g\|_{D})^{n-1}.
		\end{align*}
	\end{lem}

	\section{Supercritical case: the proof of Theorem \ref{theorem-main 1}}

%
%
%
%
%

In this part, we aim to prove that for any $\gamma \in \R$, there exists $\eta\in L_x^{r}+ L_x^{\infty}(\R^d)$ with $1\leq r<\frac d2$, $d\geq 3$, such that the equation \eqref{eq:NLS} is ill-posed in $H_x^{\gamma}(\R^d)$. Let
\EQ{
A (u_0)(t)\triangleq\int_0^t e^{-i\rho\Delta}(\eta e^{i\rho\Delta}u_0)d\rho.
}
Let the parameters $M, N, K_0\gg 1$ be determined later, and satisfy $M=K_0 N$. Next, on one hand, we choose the initial data
\EQ{
u_0(x):=\mathscr{F}^{-1}\big(N^{-\frac d2-\gamma }\chi_{N\leq |\cdot|\leq 2N}(\xi)\big)(x).
}
Then we have
\EQ{\|u_0\|_{H_x^{\gamma}}^2=\|\langle\xi\rangle^\gamma \widehat{u_0}(\xi)\|_{L_{\xi}^2}^2\sim N^{-d}\int_{N}^{2N}r^{d-1}dr\sim 1.
}
On the other hand, we choose the potential
\EQ{
\eta(x)=M^{\frac dr}\mathscr{F}^{-1}\big(\chi_{\frac 12\leq |\cdot|\leq 2}(\xi)\big)(M x).
}
Then we have
\EQ{
\widehat{\eta}(\xi)=M^{-d(1-\frac 1r)}\chi_{\frac 12\leq |\cdot|\leq 2}\big(\frac{\xi}{M}\big).
}
Note $\chi_{\frac 12\leq |\cdot|\leq 2}(\xi)$ is a Schwartz function, hence
\EQ{
\norm{\eta}_{L_x^r}=\norm{\mathscr{F}^{-1}\big(\chi_{\frac 12\leq |\cdot|\leq 2}(\xi)\big)}_{L_x^r}<\infty.
}
Next, we aim to prove that for any $T>0$ and $\gamma\in \R$,
\EQ{
\sup\limits_{t\in [0,T]}\|A(u_0)(t)\|_{H_x^{\gamma}(\R^d)}\rightarrow \infty, \mbox{ as }N\rightarrow \infty.
}
For our purpose, we set
$$
t\triangleq \frac 1{M^2},
$$
and
$$
\Omega=\{\xi: \sqrt{\frac {\pi}{3}}M\leq|\xi|\leq \sqrt{\frac {\pi}{2}}M\}.
$$
For $A(u_0)$, by the integration-by-parts and the definition of $u_0$ and $\eta$, we have
\begin{align}\label{621656}
\widehat{A (u_0)}(\xi)=&M^{-d(1-\frac 1r)}N^{-\frac d2-\gamma }\int_0^t \int_{\xi=\xi_1+\xi_2}e^{is(|\xi|^2-|\xi_2|^2)}\chi_{\frac 12\leq |\cdot|\leq 2}\big(\frac{\xi_1}{M}\big)\chi_{N\leq |\cdot|\leq 2N}(\xi_2)d\xi_2ds\notag\\
=&M^{-d(1-\frac 1r)}N^{-\frac d2-\gamma }\int_{\xi=\xi_1+\xi_2}\frac{e^{it(|\xi|^2-|\xi_2|^2)}-1}{i(|\xi|^2-|\xi_2|^2)}\chi_{\frac 12\leq |\cdot|\leq 2}\big(\frac{\xi_1}{M}\big)\chi_{N\leq |\cdot|\leq 2N}(\xi_2)d\xi_2.
\end{align}
Hence, taking the real part of $\widehat{A (u_0)}(\xi)$, we obtain
\begin{align}\label{6121}
\mbox{Re}\widehat{A (u_0)}(\xi)=M^{-d(1-\frac 1r)}N^{-\frac d2-\gamma }\int_{\xi=\xi_1+\xi_2}\frac{\sin[t(|\xi|^2-|\xi_2|^2)]}{|\xi|^2-|\xi_2|^2}\chi_{\frac 12\leq |\cdot|\leq 2}\big(\frac{\xi_1}{M}\big)\chi_{N\leq |\cdot|\leq 2N}(\xi_2)d\xi_2.
\end{align}
Noting that $t|\xi|^2\in(\frac {\pi}3, \frac {\pi}2)$ for any $\xi\in \Omega$, and $t|\xi_2|^2\sim K_0^{-2}\leq \frac 14$ for $K_0$ large enough, by the mean value theorem, we have
\begin{align}\label{621}
\sin[t(|\xi|^2-|\xi_2|^2)]=\sin (t|\xi|^2)+O(t|\xi_2|^2)
\geq \frac 14.
\end{align}
By the estimates \eqref{6121} and \eqref{621}, we obtain
\begin{align}\label{6122}
\mbox{Re}\widehat{A (u_0)}(\xi)\geq \frac 14 M^{-d(1-\frac 1r)}N^{-\frac d2-\gamma }\int_{\xi=\xi_1+\xi_2}\frac{1}{|\xi|^2-|\xi_2|^2}\chi_{\frac 12\leq |\cdot|\leq 2}\big(\frac{\xi_1}{M}\big)\chi_{N\leq |\cdot|\leq 2N}(\xi_2)d\xi_2>0.
\end{align}
Further, the above inequality yields that
\begin{align*}
\|A (u_0)\|_{H_x^{\gamma}(\R^d)}= \norm{\langle\xi\rangle^{\gamma}\widehat{A(u_0)}(\xi)}_{L_{\xi}^2(\R^d)}\geq \norm{\langle\xi\rangle^{\gamma}\mbox{Re}\widehat{A(u_0)}(\xi)}_{L_{\xi}^2(\R^d)}.
\end{align*}
Finally, combing the estimate \eqref{6122}, we get
\begin{align*}
\|A (u_0)\|_{H_x^{\gamma}(\R^d)}\geq& CM^{\gamma-d(1-\frac 1r)}N^{-\frac d2-\gamma }\norm{\int_{\R^d}\frac{1}{|\xi|^2-|\xi_2|^2}\chi_{N\leq |\cdot|\leq 2N}(\xi_2)d\xi_2}_{L_{\xi}^2(\Omega)}\\
\geq & CM^{\gamma-2-d(1-\frac 1r)}N^{-\frac d2-\gamma }\norm{\int_{\R^d}\chi_{N\leq |\cdot|\leq 2N}(\xi_2)d\xi_2}_{L_{\xi}^2(\Omega)}\\
\geq & CM^{\gamma-2-d(1-\frac 1r)}N^{-\frac d2-\gamma }M^{\frac d2}N^d\\
=&CK_0^{\gamma-2-d(\frac 12-\frac 1r)}N^{-2+\frac dr},
\end{align*}
where $C>0$ and $K_0>0$ are finite. Hence, by $-2+\frac dr>0$, we obtain that for any $T>0$ and $\gamma\in \R$
\begin{align}\label{Bu0}
\sup\limits_{t\in [0,T]}\|A (u_0)\|_{H_x^{\gamma}(\R^d)}\rightarrow \infty, \mbox{ as }N\rightarrow \infty.
\end{align}
The proof of ill-posedness is done by applying Lemma \ref{ill-posed}. This completes the proof of Theorem \ref{theorem-main 1}.

	\section{Resonant and non-resonant decomposition}
We now turn to the well-posedness analysis for \eqref{eq:NLS} with critical or subcritical index, that is $\eta \in L_x^{r}+L_x^{\infty}(\R^d)$, where $r=\frac d2$ or $r>\frac d2$. In the following, we only need to consider $\eta \in L_x^r$. Indeed, for $\eta=\eta_1+\eta_2$, with $\eta_1\in L_x^r$ and $\eta_2\in L_x^{\infty}$, we denote
$$
\Phi_j(u):=\int_0^t e^{i(t-\rho)\Delta}(\eta_j u)d\rho.
$$
Then we shall prove $\Phi_1(u)$ and $\Phi_2(u)$ are closed in $H_x^{\tilde{\gamma}_*}$ and $H_x^2$, respectively. Here $\tilde{\gamma}_*=\mbox{min}\{2, 2+\frac d2-\frac dr+o_-\}$. Since $\tilde{\gamma}_*\leq 2$, $\Phi_1(u)$ and $\Phi_2(u)$ are both closed in $H_x^{\tilde{\gamma}_*}$.

 The key difficulty in closing the estimates arises when the potential $\eta$ exhibits high-frequency components while the solution $u$ remains low-frequency, how do we transfer the derivative when it acts on potential $\eta$. In order to overcome this difficulty, we shall use the technique of the resonant and non-resonant decomposition. We now proceed to describe this decomposition in detail.
	
	By Duhamel's formula, the integral equation for \eqref{eq:NLS} is
	\EQn{\label{Duhamel-NLS}
		u(t)=e^{it\Delta}u_0+i\int_{0}^t e^{i(t-\rho)\Delta}(\eta u)(\rho)d\rho.
	}
	Next, we apply the normal form transform to give a suitable resonant and non-resonant decomposition for the integral term in \eqref{Duhamel-NLS}. Firstly, we give the following definition.
	
	\begin{definition}
		\label{defn:normal-form}
		Let $\alpha\in \R$, and $N_0\in 2^\N$, denote the multiplier
		\EQ{
			m(\xi_1, \xi_2):=\frac{\langle\xi\rangle^{\alpha}\langle\xi_1\rangle^{2-\alpha}}{|\xi|^2-|\xi_2|^2}\phi_{\geq N_0}(|\xi|)\phi_{\ll 1}\Big(\frac{|\xi_2|}{|\xi|}\Big),
		}
		where $\xi=\xi_1+\xi_2$. Using this notation, we give the following definitions:
		\begin{enumerate}
			\item (Boundary term)
			We define the normal form transform of functions $f, g$ by
			\EQ{
				\mathcal{B}(f, g)(x) := \int_{\xi=\xi_1+\xi_2} e^{ix(\xi_1+\xi_2)} m(\xi_1, \xi_2)  \widehat{f}(\xi_1)\widehat{g}(\xi_2)d\xi_1d\xi_2.
			}
			\item (Resonance term and low frequency term)
			We define the resonance part and some remainder terms of the term $\eta u$ by
			\EQ{
				\mathcal{R}(\eta, u):= P_{\leq N_0}(\eta u)+P_{\geq N_0}\sum_{M\gtrsim N}P_N(\eta P_Mu).
			}
		\end{enumerate}
	\end{definition}
	\begin{remark}
	It is easy to check that the multiplier $m$ satisfies the conditions of Coifman-Meyer's multiplier in Lemma \ref{lem:Coifman-Meyer}.
	\end{remark}
	Using the notations in the above definition, we can rewrite $\langle\nabla \rangle^{\alpha} u(t,x)$ in the following form.
	\begin{lem}\label{resoant-decom}
Let $\alpha\in \R$. Let $u(t,x)$ be defined in \eqref{Duhamel-NLS}, the bilinear operator $\mathcal{B}$ and the function $\mathcal{R}(\eta, u)$ be defined in Definition \ref{defn:normal-form}. Then we have
		\EQn{\label{Duhamel-NLS-normal}
			\langle\nabla \rangle^{\alpha} u(t,x)=&\langle\nabla \rangle^{\alpha}e^{it\Delta}u_0(x)
			-e^{it\Delta}\mathcal{B}(\langle\nabla \rangle^{-2+\alpha}\eta, u_0(x))\\
&+\mathcal{B}(\langle\nabla \rangle^{-2+\alpha}\eta, u(t,x))\\
			&+i\int_0^te^{i(t-\rho)\Delta}\langle\nabla \rangle^{\alpha}\mathcal{R}(\eta, u(\rho, x))d\rho\\
			&-i\int_0^te^{i(t-\rho)\Delta}\mathcal{B}(\langle\nabla \rangle^{-2+\alpha}\eta, \eta u)(\rho, x)d\rho.
		}
	\end{lem}
The proof of this lemma can be referred to Lemma 4.3 in \cite{Bai-Lian-Wu-2024}. Here we omit the details.

\section{Subcritical case: the proof of Theorem \ref{theorem-one}}\label{low-regularity-830}
	\subsection{Low regularity for a subcritical index $\frac d2<r\leq 2$}\label{low-one}

In this part, we use Lemma \ref{resoant-decom} to derive a weak regularity result for \eqref{eq:NLS} in the subcritical index regime. First, we make the choices of some parameters:

(1) For any $\frac d2<r\leq 2$, $d=2, 3$, define $\varepsilon_0$ be an arbitrary small constant such that
$$0<\varepsilon_0<\frac 2d-\frac 1r.$$

(2) Define the regularity index $\alpha$ as follows,
$$
	\alpha=\left\{ \aligned
	&4-\frac {4}r-2\varepsilon_0, \quad d=2,
	\\
	&\frac 92-\frac {6}r, \quad d=3.
	\endaligned
	\right.
$$

(3) Define the Schr\"odinger admissible pair $(q_0, r_0)$ as follows,
$$
	(q_0, r_0)=\left\{ \aligned
	&(2+, \infty-), \quad d=2,
	\\
	&(2, 6), \qquad d=3.
	\endaligned
	\right.
$$
More precisely, denote $(2+, \infty-):=(\frac 2{1-2\varepsilon_0}, \frac 1{\varepsilon_0})$.

(4) Define the Schr\"odinger admissible pair $(q_1, r_1)$ by the following,
\begin{align*}
\frac 2 {q_1}=-\frac d2+\frac d{r_0}+\frac dr, \mbox{ and } \frac 1 {r_1}=1-\frac 1{r_0}-\frac 1r.
\end{align*}

We define the auxiliary space $X(I)$ for $I\subset\R^+$ by the following norms,
	\EQn{\label{3251-2}
		\|u\|_{X(I)}:=\|u\|_{L_t^{\infty}L_x^{2}(I)} + \|u\|_{L_t^{q_1}L_x^{r_1}(I)}.
	}

We first establish the global well-posedness of \eqref{eq:NLS} in the space $H_x^{\alpha}(\R^d)$. This constitutes a weak regularity result, as $\alpha < \gamma^* = 2 + \frac{d}{2} - \frac{d}{r}$, where $\gamma^*$ denotes the expected critical regularity index.
\begin{prop}\label{weak-subcritical}
Let $d=2, 3$, $\frac d2<r\leq 2$, and $\eta\in L_x^{r}(\R^d)$, then \eqref{eq:NLS} is globally well-posed in $H_x^{\alpha}(\R^d)$.
\end{prop}
Next we provide the key estimates to prove this proposition.

%
%
%
%


%

	\subsubsection{Boundary terms}
	\begin{lem}[Boundary terms]\label{lem:nonlinear-estimate-boundary-22}
		Let $d=2, 3$, and $\frac d2<r\leq 2$. Let $I \subset\R^+$ be an interval containing $0$. Then, for any $N_0\in2^\N$,
		\EQn{\label{32511-22}
			\normb{e^{it\Delta}\mathcal{B}(\langle\nabla \rangle^{-2+\alpha}\eta, u_0)}_{X(I)} \lsm \|P_{\geq N_0}\eta\|_{L_x^{r}}\|\langle\nabla \rangle^{\alpha}u\|_{X(I)},
		}
		and
		\EQn{\label{32512-22}
			\norm{\mathcal{B}(\langle\nabla \rangle^{-2+\alpha}\eta, u(t))}_{X(I)} \lsm  \|P_{\geq N_0} \eta\|_{L_x^{r}}\|\langle\nabla \rangle^{\alpha}u\|_{X(I)}.
		}
	\end{lem}
	\begin{proof}
		By Strichartz's estimates, we have
		\EQn{\label{32513-22}
			\normb{e^{it\Delta}\mathcal{B}(\langle\nabla \rangle^{-2+\alpha}\eta, u_0)}_{X(I)} \lsm \normb{\mathcal{B}(\langle\nabla \rangle^{-2+\alpha}\eta, u_0)}_{L_x^2}.
		}
		Due to the restriction of applying the Sobolev inequality, we prove this lemma from the following three cases: $\alpha<\frac d2$, $\alpha=\frac d2$, and $\alpha>\frac d2$.
		
		When $\alpha<\frac d2$, noting $-2+\frac dr<0$, by Lemma \ref{lem:Coifman-Meyer} and Sobolev's inequality,
		\EQ{
			\normb{\mathcal{B}(\langle\nabla \rangle^{-2+\alpha}\eta, u_0)}_{L_x^2}\lsm &\|\langle\nabla \rangle^{-2+\alpha}P_{\geq N_0}\eta\|_{L_x^{\frac d{\alpha}}}\|u_0\|_{L_x^{\frac d{\frac d2-\alpha}}}\\
			\lsm &\|\langle\nabla \rangle^{-2+\alpha+\frac dr-\alpha}P_{\geq N_0}\eta\|_{L_x^{r}}\|u\|_{L_t^{\infty}H_x^{\alpha}}\\
			\lsm &\|P_{\geq N_0}\eta\|_{L_x^{r}}\|u\|_{L_t^{\infty}H_x^{\alpha}}.
		}
		When $\alpha=\frac d2$, noting $-2+\frac dr+d\varepsilon_0<0$,
		\EQ{
			\normb{\mathcal{B}(\langle\nabla \rangle^{-2+\alpha}\eta, u_0)}_{L_x^2}\lsm &\|\langle\nabla \rangle^{\frac d2-2}P_{\geq N_0}\eta\|_{L_x^{\frac 2{1-2\varepsilon_0}}}\|u_0\|_{L_x^{\frac 1{\varepsilon_0}}}\\
 \lsm &\|\langle\nabla \rangle^{\frac d2-2+\frac dr-\frac {d(1-2\varepsilon_0)}2}P_{\geq N_0}\eta\|_{L_x^{r}}\|u\|_{L_t^{\infty}H_x^{\alpha}}\\
  \lsm &\|P_{\geq N_0}\eta\|_{L_x^{r}}\|u\|_{L_t^{\infty}H_x^{\alpha}}.
		}
		When $\alpha>\frac d2$, noting $\alpha-2-\frac d2+\frac dr< 0$, we have
		\EQ{
			\normb{\mathcal{B}(\langle\nabla \rangle^{-2+\alpha}\eta, u_0)}_{L_x^2}\lsm &\|\langle\nabla \rangle^{-2+\alpha}P_{\geq N_0}\eta\|_{L_x^{2}}\|u_0\|_{L_x^{\infty}}\\
			\lsm &\|\langle\nabla \rangle^{-2+\alpha+\frac dr-\frac d2}P_{\geq N_0}\eta\|_{L_x^{r}}\|u\|_{L_t^{\infty}H_x^{\alpha}}\\
			\lsm &\|P_{\geq N_0}\eta\|_{L_x^{r}}\|u\|_{L_t^{\infty}H_x^{\alpha}}.
		}
		Hence, \eqref{32511-22} follows from \eqref{32513-22} and the above three estimates.
		
		Next, we give the proof of \eqref{32512-22}. First, following the same procedure as above, we conclude that
		\EQ{
			\norm{\mathcal{B}(\langle\nabla \rangle^{-2+\alpha}\eta, u(t))}_{L_t^{\infty}L_x^{2}} \lsm  \|P_{\geq N_0} \eta\|_{L_x^{r}}\|u\|_{L_t^{\infty}H_x^{\alpha}}.
		}
		It is reduced to control term $\norm{\mathcal{B}(\langle\nabla \rangle^{-2+\alpha}\eta, u(t))}_{L_t^{q_1}L_x^{r_1}}$.
Here, we note that
		\EQn{\label{3262-22}
			\alpha-\frac d{r_1}>0, \mbox{and}-2+\alpha+\frac dr-\frac d{r_1}=0.
		}
By Lemma \ref{lem:Coifman-Meyer} and Sobolev's inequality,
		\EQ{
			\norm{\mathcal{B}(\langle\nabla \rangle^{-2+\alpha}\eta, u(t))}_{L_t^{q_1}L_x^{r_1}} \lsm & \| \langle\nabla\rangle^{-2+\alpha}P_{\geq N_0}\eta\|_{L_x^{r_1}}\|u\|_{L_t^{q_1}L_x^{\infty}}\\
			\lsm &\| \langle\nabla\rangle^{-2+\alpha+\frac dr-\frac d{r_1}}P_{\geq N_0} \eta\|_{L_x^{r}}\|\langle\nabla\rangle^{\alpha}u\|_{L_t^{q_1}L_x^{r_1}}\\
			\lsm &\|P_{\geq N_0} \eta\|_{L_x^{r}}\|\langle\nabla\rangle^{\alpha}u\|_{L_t^{q_1}L_x^{r_1}}.
		}
	This gives \eqref{32512-22}. Hence, we finish the proof of this lemma.
	\end{proof}

\subsubsection{Resonance term and low frequency term}
	\begin{lem}\label{lem:nonlinear-estimate-resonance-22}
Let $d=2, 3$, and $\frac d2<r\leq 2$. Let $I=[0, T) \subset\R^+$ be an interval. Then, for any $N_0\in2^\N$,
		\EQ{
			\normbb{\int_{0}^t e^{i(t-\rho)\Delta}\langle\nabla \rangle^{\alpha}\mathcal{R}(\eta, u) d\rho}_{X(I)} \lsm  T^{\frac 1{q_0'}-\frac 1{q_1}}N_0^{\alpha} \|\eta\|_{L_x^{r}}\|\langle\nabla \rangle^{\alpha}u\|_{X(I)}.
		}
	\end{lem}
	\begin{proof}
By the definition of $\mathcal{R}(\eta, u)$ in Definition \ref{defn:normal-form}, we have
\EQnnsub{
\Big\|\int_{0}^t& e^{i(t-\rho)\Delta}\langle\nabla \rangle^{\alpha}\mathcal{R}(\eta, u) d\rho\Big\|_{X(I)}\notag\\
\lsm &\Big\|\int_{0}^t e^{i(t-\rho)\Delta}\langle\nabla \rangle^{\alpha} P_{\leq N_0}(\eta u)d\rho\Big\|_{X(I)}\label{64001-22}\\
&+\Big\|\int_{0}^t e^{i(t-\rho)\Delta}\langle\nabla \rangle^{\alpha} P_{\geq N_0}\sum_{M\gtrsim N}P_N(\eta P_Mu)d\rho\Big\|_{L_t^{\infty}L_x^2}\label{64002-22}\\
&+\Big\|\int_{0}^t e^{i(t-\rho)\Delta}\langle\nabla \rangle^{\alpha} P_{\geq N_0}\sum_{M\gtrsim N}P_N(\eta P_Mu)d\rho\Big\|_{L_t^{q_1}L_x^{r_1}}\label{64003-22}.
}
For \eqref{64001-22}, by Lemmas \ref{lem:Bernstein}, \ref{lem:strichartz}, we get
\begin{align}\label{65001-22}
\eqref{64001-22}\lsm N_0^{\alpha}\|\eta u\|_{L_t^{q_0'}L_x^{r_0'}}
\lsm T^{\frac 1{q_0'}-\frac 1{q_1}}N_0^{\alpha} \|\eta\|_{L_x^{r}}\|u\|_{L_t^{q_1}L_x^{r_1}},
\end{align}
where $\frac 1{q_0'}-\frac 1{q_1}>0$.

For \eqref{64002-22}, by the duality, Lemmas \ref{lem:Schur}, \ref{lem:strichartz},
\begin{align}\label{32514-22}
 \eqref{64002-22}
\lsm& \sup_{h: \norm{h}_{L_x^2}\leq 1}\normbb{\Big\langle\int_{0}^t\langle\nabla \rangle^{\alpha} e^{i(t-\rho)\Delta}P_N\sum _{N\lsm M}(\eta P_Mu )d\rho, h\Big\rangle}_{L_t^{\infty}}\notag\\
\lsm &\sup_{h:  \norm{h}_{L_x^2}\leq 1}\sum _{N\lsm M}\frac{\langle N \rangle^{\alpha}}{\langle M \rangle^{\alpha}}\normbb{\int_{0}^t e^{i(t-\rho)\Delta}P_N(\eta \langle M \rangle^{\alpha}P_Mu )d\rho}_{L_t^{\infty}L_x^2}\|P_N h\|_{L_x^2}\notag\\
\lsm &\sup_{h:  \norm{h}_{L_x^2}\leq 1}\sum _{N\lsm M}\frac{\langle N \rangle^{\alpha}}{\langle M \rangle^{\alpha}} \norm{\eta \langle M \rangle^{\alpha}P_Mu}_{L_t^{q_0'}L_x^{r_0'}}\|P_N h\|_{L_x^2}\notag\\
\lsm &\norm{\eta \langle M \rangle^{\alpha}P_Mu}_{l_M^2L_t^{q_0'}L_x^{r_0'}}.
\end{align}
By Sobolev's inequality, Minkowski's inequality, and Lemma \ref{lem:littlewood-Paley}, we get
\begin{align}\label{32515-22}
\norm{\eta \langle M \rangle^{\alpha}P_Mu}_{l_M^2L_t^{q_0'}L_x^{r_0'}}\lsm& T^{\frac 1{q_0'}-\frac 1{q_1}}\|\eta\|_{L_x^{r}}\|\langle M \rangle^{\alpha}P_Mu\|_{L_t^{q_1}L_x^{r_1}l_M^2}\notag\\
\lsm & T^{\frac 1{q_0'}-\frac 1{q_1}}\|\eta\|_{L_x^r}\|u\|_{L_t^{q_0}F_{r_1}^{\alpha, 2}}\notag\\
\lsm & T^{\frac 1{q_0'}-\frac 1{q_1}}\|\eta\|_{L_x^r}\|\langle \nabla \rangle^{\alpha}u\|_{L_t^{q_1}L_x^{r_1}}.
\end{align}
Hence, by \eqref{32514-22} and \eqref{32515-22}, we get
\begin{align}\label{32516-22}
 \eqref{64002-22}\lsm T^{\frac 1{q_0'}-\frac 1{q_1}}\|\eta\|_{L_x^r}\|\langle \nabla \rangle^{\alpha}u\|_{L_t^{q_1}L_x^{r_1}}.
\end{align}
For \eqref{64003-22}, following the same argument as used in \eqref{32514-22} and \eqref{32515-22},
\begin{align}\label{32517-22}
\eqref{64003-22}
\lsm& \sup_{h: \norm{h}_{L_x^{{r_1}'}}\leq 1}\normbb{\Big\langle\int_{0}^t\langle\nabla \rangle^{\alpha} e^{i(t-\rho)\Delta}P_N\sum _{N\lsm M}(\eta P_Mu )d\rho, h\Big\rangle}_{L_t^{q_1}}\notag\\
\lsm &\sup_{h: \norm{h}_{L_x^{{r_1}'}}\leq 1}\sum _{N\lsm M}\frac{\langle N \rangle^{\alpha}}{\langle M \rangle^{\alpha}}\normbb{\int_{0}^t e^{i(t-\rho)\Delta}P_N(\eta \langle M \rangle^{\alpha}P_Mu )d\rho}_{L_t^{q_1}L_x^{r_1}}\|P_N h\|_{L_x^{{r_1}'}}\notag\\
\lsm &\sup_{h: \norm{h}_{L_x^{{r_1}'}}\leq 1}\sum _{N\lsm M}\frac{\langle N \rangle^{\alpha}}{\langle M \rangle^{\alpha}} \norm{\eta \langle M \rangle^{\alpha}P_Mu}_{L_t^{q_0'}L_x^{r_0'}}\|P_N h\|_{L_x^{{r_1}'}}\notag\\
\lsm& \norm{\eta \langle M \rangle^{\alpha}P_Mu}_{l_M^2L_t^{q_0'}L_x^{r_0'}}\notag\\
\lsm & T^{\frac 1{q_0'}-\frac 1{q_1}}\|\eta\|_{L_x^r}\|\langle \nabla \rangle^{\alpha}u\|_{L_t^{q_1}L_x^{r_1}}.
\end{align}
Hence, by the estimates \eqref{65001-22}, \eqref{32516-22}, and \eqref{32517-22}, we finish the proof of this lemma.
\end{proof}

	\subsubsection{High-order terms}
	\begin{lem}[Higher order terms]\label{lem:nonlinear-estimate-higher-order-22d}
Let $d=2, 3$, and $\frac d2<r\leq 2$. Let $I=[0, T)\subset\R^+$. Then
\begin{align*}
\normbb{\int_0^te^{i(t-\rho)\Delta}\mathcal{B}(\langle\nabla \rangle^{-2+\alpha}\eta, \eta u)(\rho, x)d\rho}_{X(I)} \lsm& T^{\frac 1{q_0'}-\frac 1{q_1}}\|\eta\|_{L_x^r}^2\|\langle\nabla \rangle^{\alpha}u\|_{X(I)}.
\end{align*}
\end{lem}
\begin{proof}
By Strichartz's estimates, we get
\begin{align}
\normbb{\int_0^t e^{i(t-\rho)\Delta}\mathcal{B}(\langle\nabla \rangle^{-2+\alpha}\eta, \eta u)(\rho, x)d\rho}_{X(I)}
\lsm \norm{\mathcal{B}(\langle\nabla \rangle^{-2+\alpha}\eta, \eta u)}_{L_t^{q_0'}L_x^{r_0'}}.\label{65002-22}
\end{align}
Recalling from \eqref{3262-22} that $\alpha r_1>d$ and $-2+\alpha+\frac dr-\frac d{r_1}=0$, by Lemma \ref{lem:Coifman-Meyer}, Sobolev's and H\"{o}lder's inequalities, we have
\begin{align}\label{eq:nonlinear-estimate-higher-order-22}
\eqref{65002-22}\lsm & T^{\frac 1{q_0'}-\frac 1{q_1}}\|\langle\nabla \rangle^{-2+\alpha} P_{\geq N_0} \eta\|_{L_x^{r_1}}\|\eta\|_{L_x^{r}}\|u\|_{L_t^{q_1}L_x^{\infty}}\notag\\
\lsm & T^{\frac 1{q_0'}-\frac 1{q_1}}\|\langle\nabla \rangle^{-2+\alpha+\frac dr-\frac d{r_1} } P_{\geq N_0}\eta\|_{L_x^{r}}\|\eta\|_{L_x^r}\|\langle\nabla \rangle^{\alpha}u\|_{L_t^{q_1}L_x^{r_1}}\notag\\
\lsm &T^{\frac 1{q_0'}-\frac 1{q_1}}\|\eta\|_{L_x^r}^2\|\langle\nabla \rangle^{\alpha}u\|_{L_t^{q_1}L_x^{r_1}}.
\end{align}
This proves this Lemma.
\end{proof}

Based on the above several lemmas, we are now in a position to prove Proposition \ref{weak-subcritical}.
	\begin{proof}[\bf{Proof of Proposition \ref{weak-subcritical}}]
	We firstly prove the local well-posedness. By Strichartz's estimate, we have
		\EQn{\label{446-22}
			\norm{\langle\nabla \rangle^{\alpha}e^{it\Delta} u_0}_{X(I)}\leq C\|u_0\|_{H_x^{\alpha}}:=R.
		}
		Moreover, for any $0<\delta\ll 1$, by $\eta \in L_x^r(\R^d)$ for $\frac  d2 <r\leq  2$, we choose $N_0\in 2^{\N}$ large enough such that
		\EQn{\label{447-22}
			\|P_{\geq N_0} \eta\|_{L_x^{r}}\leq \delta.
		}
		Denote the operator $\Phi$ by the following form,
		\EQ{
			\langle\nabla\rangle^{\alpha}\Phi(u)=&\langle\nabla \rangle^{\alpha}e^{it\Delta}u_0(x)
			-e^{it\Delta}\mathcal{B}(\langle\nabla \rangle^{-2+{\alpha}}\eta, u_0(x))\\
&+\mathcal{B}(\langle\nabla \rangle^{-2+{\alpha}}\eta, u(t,x))\\
			&+i\int_0^te^{i(t-\rho)\Delta}\langle\nabla \rangle^{\alpha}\mathcal{R}(\eta, u(\rho, x))d\rho\\
			&-i\int_0^te^{i(t-\rho)\Delta}\mathcal{B}(\langle\nabla \rangle^{-2+{\alpha}}\eta, \eta u)(\rho, x)d\rho.
		}
		Take the working space as
		\EQ{
			B_{R}:=\{u\in C(I; H_x^{\alpha}(\R)):\|\langle\nabla \rangle^{\alpha} u\|_{X(I)}\leq 2R\}.
		}
		Next, we aim to prove $\Phi$ is a contraction mapping in $B_{R}$. Hence, we need to collect the estimates of $\langle\nabla\rangle^{\alpha}\Phi(u)$ in $X(I)$.

		By Lemma \ref{lem:nonlinear-estimate-boundary-22},
		\EQn{\label{442-22}
			\normb{e^{it\Delta}\mathcal{B}(\langle\nabla \rangle^{-2+\alpha}\eta, u_0)}_{X(I)}\lsm \delta R,
		}
		and
		\EQn{\label{443-22}
			\norm{\mathcal{B}(\langle\nabla \rangle^{-2+\alpha}\eta, u(t))}_{X(I)}\lsm \delta R.
		}
		By Lemma \ref{lem:nonlinear-estimate-resonance-22}, there exists $\gamma>0$, such that
		\EQn{\label{444-22}
			\normbb{\int_{0}^t e^{i(t-\rho)\Delta}\langle\nabla \rangle^{\alpha}\mathcal{R}(\eta, u) d\rho}_{X(I)} \lsm& T^{\gamma}RN_0^{\alpha}\|\eta\|_{L_x^r}.
		}
		By Lemma \ref{lem:nonlinear-estimate-higher-order-22d},
		\EQn{\label{445-22}
			\normbb{\int_0^te^{i(t-\rho)\Delta}\mathcal{B}(\langle\nabla \rangle^{-2+\alpha}\eta, \eta u)d\rho}_{X(I)} \lsm& T^{\gamma}R\|\eta\|_{L_x^r}^2.
		}
		By the estimates \eqref{442-22}-\eqref{445-22}, for any $u\in B_{R}$, there exists a constant $C=C(\|\eta\|_{L_x^r})$, such that
		\begin{align}\label{448-22}
			\norm{\langle\nabla\rangle^\alpha\Phi(u)}_{X(I)}\leq  R+C\delta R+CT^{\gamma}RN_0^{\alpha}+CT^{\gamma}R.
		\end{align}
		First, by \eqref{447-22}, there exists $N_0=N_0(\delta, \|\eta\|_{L_x^r})$, such that
		\EQ{
			C\delta \leq \frac 14.
		}
	Then, we take $T=T(N_0, \|\eta\|_{L_x^r})$ small enough, such that
\EQ{
CT^{\gamma}N_0^{\alpha}+CT^{\gamma}\leq \frac 12.
}
		Therefore, by the above inequalities, we have
		\EQ{
			\norm{\langle\nabla\rangle^\alpha \Phi(u)}_{X(I)}\leq 2R.
		}
		Hence, we have that $\Phi:B_{R}\rightarrow B_{R}$. Therefore, the local well-posedness follows from the contraction mapping principle.

We emphasize that the lifespan $T$ obtained above depends only on $\norm{\eta}_{L_x^r}$. This allows us to extend the local solution $u$ globally. In fact, let $u\in C([0,T^*);H_x^{\alpha})$ be the solution of equation \eqref{eq:NLS} with the maximal lifespan $[0,T^*)$. Let $0<\epsilon_0<T$, where $T=T( \norm{\eta}_{L_x^{r}})$ is the lifespan established in the local well-posedness argument.
Assume by contradiction that $T^*<+\infty$. By the local well-posedness theory, we have $u\in C([0,T^*-\epsilon_0);H_x^{\alpha})$ and $\norm{u(T^*-\epsilon_0)}_{H_x^{\alpha}} \lsm \norm{u_0}_{H_x^{\alpha}} $. Applying the local existence argument again at $T^*-\epsilon_0$, we can extend the solution to the interval $[0, T^*-\epsilon_0+T)$. Since $T^*-\epsilon_0+T>T^*$, this contradicts to the definition of $T^*$. Therefore, $T^*=+\infty$. This finishes the proof of the global well-posedness in $H_x^{\alpha}$.
	\end{proof}

\subsection{Global well-posedness in $H_x^{2+\frac d2-\frac dr}(\R^d)$}\label{high-one}

To further improve the regularity, the argument above appears to be no longer applicable, now we employ an alternative approach to achieve it.
Define $s\triangleq\frac d2-\frac dr$, where $\frac d2<r\leq2$ and $d=2, 3$. Building on the results from subsection \ref{low-one}, we have that the solution $u$ is global in $H_x^{\alpha}$ and for any $T>0$,
\EQn{\label{62011-22}
\norm{u}_{L_t^{\infty}H_x^{\alpha}([0, T)\times \R^d)}\leq  C(T) \norm{u_0}_{H_x^{\alpha}}\leq C(T)  \norm{u_0}_{H_x^{2+s}}.
}
Denote $v=\partial_t u$, then from \eqref{eq:NLS}, $v$ satisfies the following equation
\begin{equation}\label{eq:NLS-222}
		\left\{ \aligned
		&i\partial_t v+\Delta v+\eta v=0, \qquad t\in \R^+ \mbox{ and } x\in \R^d,
		\\
		&v(0,x)=i(\Delta u_0+\eta u_0)\triangleq v_0.
		\endaligned
		\right.
	\end{equation}
We now present two key observations.

\noindent $\bullet$ Claim 1: $v_0\in H_x^s$.

Indeed,
by $u_0\in H_x^{2+\frac d2-\frac dr}$ (that is $u_0\in H_x^{2+s}$), $\eta \in L_x^r$ with $\frac d2<r\leq 2$, and the Sobolev and H\"{o}lder inequalities, noting that $2+s>\frac d2$, we have
\EQ{
\norm{v_0}_{H_x^s}&=\norm{\Delta u_0+\eta u_0}_{H_x^s}\\
&\lsm \norm{u_0}_{H_x^{2+s}}+\norm{\eta u_0}_{L_x^{\frac {2d}{d-2s}}}\\
&\lsm  \norm{u_0}_{H_x^{2+s}}+\norm{\eta }_{L_x^r}\norm{ u_0}_{L_x^{\infty}}\\
&\lsm \norm{u_0}_{H_x^{2+s}}+\norm{\eta }_{L_x^r}\norm{ u_0}_{H_x^{2+s}}.
}
\noindent $\bullet$ Claim 2: $v\in C([0, T); H_x^{s})$ implies $u\in C([0, T); H_x^{2+s})$.

Indeed, by the high and low frequency decomposition,
\begin{align}\label{high-low-22}
\norm{u}_{L_t^{\infty}H_x^{2+s}}\leq \norm{P_{<N_0}u}_{L_t^{\infty}H_x^{2+s}} +\norm{P_{\geq N_0}u}_{L_t^{\infty}H_x^{2+s}},
\end{align}
where $N_0\in 2^{\N}$ will be determined later.

For $ \norm{P_{<N_0}u}_{L_t^{\infty}H_x^{2+s}} $, by Lemma \ref{lem:Bernstein} and \eqref{62011-22},
\begin{align}\label{low-22}
\norm{P_{<N_0}u}_{L_t^{\infty}H_x^{2+s}}\leq C N_0^{2+s-\alpha}\norm{u}_{L_t^{\infty}H_x^{\alpha}}\leq C(T) N_0^{2+s-\alpha}\norm{u_0}_{H_x^{2+s}}.
\end{align}
For $ \norm{P_{\geq N_0}u}_{L_t^{\infty}H_x^{2+s}} $, noting that
\EQ{
\Delta u=-i v-\eta u,
}
we have
\begin{align}\label{1023-22}
\norm{P_{\geq N_0}u}_{L_t^{\infty}H_x^{2+s}}  \leq & \norm{v}_{L_t^{\infty}H_x^{s}} +\norm{P_{\geq N_0}(\eta u)}_{L_t^{\infty}H_x^{s}}\notag\\
\leq & C(T)\norm{v_0}_{H_x^{s}} +\norm{P_{\geq N_0}(\eta u)}_{L_t^{\infty}H_x^{s}}.
\end{align}
The Sobolev and H\"{o}lder inequalities, together with Lemma \ref{lem:Bernstein}, yield
\begin{align}\label{high-1-22}
\norm{P_{\geq N_0}(\eta u)}_{L_t^{\infty}H_x^{s}}\lsm &\norm{P_{\geq N_0}(\eta u)}_{L_t^{\infty}L_x^{r}}\notag\\
\lsm& \norm{P_{\gtrsim N_0}\eta }_{L_x^{r}}\norm{u}_{L_{t,x}^{\infty}}+\norm{\eta }_{L_x^{r}}\norm{P_{\gtrsim N_0} u}_{L_{t,x}^{\infty}}\notag\\
\lsm& \norm{P_{\gtrsim N_0}\eta }_{L_x^{r}}\norm{u}_{L_{t}^{\infty}H_x^{\frac d2+}}+\norm{\eta }_{L_x^{r}}\norm{P_{\gtrsim N_0} u}_{L_{t}^{\infty}H_x^{\frac d2+\epsilon_0}}\notag\\
\lsm &\norm{P_{\gtrsim N_0}\eta }_{L_x^{r}}\norm{u}_{L_t^{\infty}H_x^{2+s}}+N_0^{-2-s+\frac d2+\epsilon_0}\norm{\eta }_{L_x^{r}}\norm{u}_{L_t^{\infty}H_x^{2+s}},
\end{align}
where $0<\epsilon_0<2-\frac dr$.

Hence, combining the above two estimates, we get
\begin{align}\label{high-22}
\norm{P_{\geq N_0}u}_{L_t^{\infty}H_x^{2+s}}  \leq & C(T) \norm{v_0}_{H_x^s}+C\norm{P_{\gtrsim N_0}\eta }_{L_x^{r}}\norm{u}_{L_t^{\infty}H_x^{2+s}}\notag\\
&+CN_0^{-2-s+\frac d2+\epsilon_0}\norm{\eta }_{L_x^{r}}\norm{u}_{L_t^{\infty}H_x^{2+s}}.
\end{align}
Now, we take $N_0=N_0(\norm{\eta }_{L_x^{r}})$ large enough such that
\begin{align}\label{N0-22}
C\norm{P_{\gtrsim N_0}\eta }_{L_x^{r}}+CN_0^{-2-s+\frac d2+\epsilon_0}\norm{\eta }_{L_x^{r}}\leq \frac 12,
\end{align}
where $-2-s+\frac d2+\epsilon_0=-2+\frac dr+\epsilon_0<0$ for any $\frac d2<r\leq 2$ and $0<\epsilon_0<2-\frac dr$.

Collecting the estimates \eqref{high-low-22}, \eqref{low-22}, \eqref{high-22}, and \eqref{N0-22}, we get
\begin{align*}
\norm{u}_{L_t^{\infty}H_x^{2+s}} \leq C (T)N_0^{2+s-\alpha}\norm{u_0}_{H_x^{2+s}}+C(T)\norm{v_0}_{H_x^s}+\frac 12 \norm{u}_{L_t^{\infty}H_x^{2+s}}.
\end{align*}
Further, combining Claim 1, this implies
 Claim 2.

Based on the above two Claims, it is reduced to prove the Cauchy problem \eqref{eq:NLS-222} is globally well-posed in $H_x^{s}$, $s=\frac d2-\frac dr$, where $\frac d2<r\leq 2$, $d=2, 3$.

For our purpose, we firstly give the following result via an iterated Duhamel construction.
\begin{prop}\label{prop:v-2d}
Let $ N\in \N$, and $S_N\triangleq\sum_{n=0}^{N}e^{it\Delta}I_n$, where the terms $I_n$  are defined recursively by
\begin{align*}
I_n=i\int_0^te^{-i\rho \Delta}(\eta e^{i\rho \Delta}I_{n-1})d \rho, {\mbox{ for }}n\geq1; \quad I_0=v_0.
\end{align*} Let $s=\frac d2-\frac dr$, where $\frac d2<r\leq2$, and $d=2, 3$. Then there exist $T=T(\|\eta\|_{L_x^r})>0$, and $v\in C([0, T); H_x^s(\R^d))$, such that
$$
\lim_{N\to \infty}S_N=v, \mbox{ in }H_x^s,
$$
where $v$ is the unique solution to equation \eqref{eq:NLS-222}.
\end{prop}
We will give the proof of Proposition \ref{prop:v-2d} in the following subsection. Now, we prove the global well-posedness of \eqref{eq:NLS} in $H_x^{2+\frac d2-\frac dr}(\R^d)$
assuming that Proposition \ref{prop:v-2d} holds.
	\begin{proof}[\bf{Proof of global well-posedness in $H_x^{2+\frac d2-\frac dr}$}]
By Claim 2, the global well-posedness of \eqref{eq:NLS} in $ H_x^{2+\frac d2-\frac dr}(\R^d)$ reduces to the global well-posedness of \eqref{eq:NLS-222} in $H_x^s$, where $s=\frac d2-\frac dr$.
By Proposition \ref{prop:v-2d}, we construct the local solution of \eqref{eq:NLS-222} in $H_x^s$. Noting that the lifespan $T$ of local solution $v$ depends only on $\|\eta\|_{L_x^r}$, we can easily extend it globally. We omit the details.
\end{proof}

\subsection{Proof of Proposition \ref{prop:v-2d}}\label{v-spacetime-830}
Next, let us focus on the proof of Proposition \ref{prop:v-2d}. Now, we need the following structural lemma.
\begin{lem}\label{In-form}
Let $I_n=i\int_0^te^{-i\rho \Delta}(\eta e^{i\rho \Delta}I_{n-1})d \rho$ for $n\geq 1$, and $I_0=v_0$. For any $M, N, N_0\in 2^{\N}$, define the operator $T_{N}$ as follows,
\begin{align*}
T_{N}f&=\eta \sum_{M:  M\gg N, M\geq N_0}|\nabla|^{-2}P_{M} f.
\end{align*}
Further, for $k\in \N$, define the operator $T_N^k$ by the following,
\begin{align*}
T_N^k f=(T_N)^kf, \mbox{ with } T_N^0 f=f.
\end{align*}
Then, we have that for any $n\geq 1$,
\begin{align}\label{I}
I_n=&\sum_{k=0}^{n-1}i\int_0^te^{-i\rho\Delta }\sum_{N}P_{N}T_N^k(\eta e^{i\rho\Delta }P_{\leq N_0}I_{n-1-k})d\rho\notag\\
&+\sum_{k=0}^{n-1}i\int_0^te^{-i\rho\Delta }\sum_{M\lsm N}P_{N}T_N^k(\eta  e^{i\rho\Delta }P_{M\geq N_0}I_{n-1-k})d\rho\notag\\
&+\sum_{k=0}^{n-1}i|\nabla|^2\int_0^te^{-i\rho\Delta }\sum_{M\gg N}P_{N}T_N^k(\eta   |\nabla|^{-2}e^{i\rho\Delta }P_{M\geq N_0}I_{n-1-k})d\rho\notag\\
&-\sum_{k=0}^{n-1}e^{-it\Delta }\sum_{M\gg N}P_{N}T_N^k(\eta   |\nabla|^{-2} e^{it\Delta }P_{M\geq N_0}I_{n-1-k})\notag\\
&+\sum_N P_N T_N^n v_0.
\end{align}
\end{lem}
\begin{proof}
For any $n\geq 1$ and $0\leq k\leq n-1$, denote
$$
I_{n, k}:=i\int_0^te^{-i\rho \Delta}\sum_{N}P_NT_N^k(\eta e^{i\rho\Delta}I_{n-k-1})d\rho,
$$
and
\begin{align*}
J_{n, k}:=&i\int_0^te^{-i\rho\Delta }\sum_{N}P_{N}T_N^k(\eta e^{i\rho\Delta }P_{\leq N_0}I_{n-1-k})d\rho\notag\\
&+i\int_0^te^{-i\rho\Delta }\sum_{M\lsm N}P_{N}T_N^k(\eta  e^{i\rho\Delta }P_{M\geq N_0}I_{n-1-k})d\rho\notag\\
&+i|\nabla|^2\int_0^te^{-i\rho\Delta }\sum_{M\gg N}P_{N}T_N^k(\eta   |\nabla|^{-2}e^{i\rho\Delta }P_{M\geq N_0}I_{n-1-k})d\rho\notag\\
&-e^{-it\Delta }\sum_{M\gg N}P_{N}T_N^k(\eta   |\nabla|^{-2} e^{it\Delta }P_{M\geq N_0}I_{n-1-k}).
\end{align*}
We now assert that the following recurrence relations hold:

{\emph{$(A_1):$ }} for any $n\geq 1$, $I_n=I_{n, 0}$;

{\emph{$(A_2):$ }} for any $n\geq 2$ and $0\leq k\leq n-2$, $I_{n, k}=J_{n, k}+I_{n, k+1}$;

{\emph{$(A_3):$ }} for any $n\geq 1$, $I_{n, n-1}=J_{n, n-1}+\sum_N P_N T_N^n v_0.$

It now suffices to prove the three identities listed above, from which the lemma immediately follows. Indeed, when $n=1$, by {\emph{$(A_1)$ }}and {\emph{$(A_3)$}},
$$
I_1=I_{1, 0}=J_{1,0}+\sum_N P_N T_N^1 v_0.
$$
This gives \eqref{I} with $n=1$. When $n\geq 2$, by {\emph{$(A_1)$}}, {\emph{$(A_2)$}}, and {\emph{$(A_3)$ }},
$$
I_n=I_{n,0}=\sum_{k=0}^{n-1}J_{n, k}+\sum_N P_N T_N^n v_0.
$$
This gives \eqref{I} with $n\geq2$.

Next, we focus on the proof of {\emph{$(A_1)$}}, {\emph{$(A_2)$}}, and {\emph{$(A_3)$ }}.
The proof of {\emph{$(A_1)$}} follows directly from the definition of $T_N^0$. To prove {\emph{$(A_2)$}}, applying the high-low frequency decomposition, for any $n\geq 1$, $0\leq k\leq n-1$, and $N_0\in 2^{\N}$,
\begin{align}\label{In}
I_{n, k}=&i\int_0^te^{-i\rho \Delta}\sum_{N}P_NT_N^k(\eta e^{i\rho\Delta}I_{n-k-1})d\rho\notag\\
=&i\int_0^te^{-i\rho \Delta}\sum_{N}P_NT_N^k(\eta e^{i\rho\Delta}P_{\leq N_0}I_{n-k-1})d\rho\notag\\
&+i\int_0^t e^{-i\rho \Delta}\sum_{M\lsm N}P_NT_N^k(\eta e^{i\rho \Delta}P_{M\geq N_0}I_{n-k-1})d\rho\notag\\
&+i\int_0^t e^{-i\rho \Delta}\sum_{M\gg N}P_NT_N^k(\eta e^{i\rho \Delta}P_{M\geq N_0}I_{n-k-1})d\rho.
\end{align}
Denote $I_{n, k}^{h}$ as follows,
\begin{align}\label{Jh12d}
I_{n, k}^{h}:=i\int_0^te^{-i\rho \Delta}\sum_{M\gg N}P_{N}T_N^k(\eta e^{i\rho \Delta}P_{M\geq N_0}I_{n-k-1})d\rho.
\end{align}
Denote the multiplier
$$m(\myvec{\xi}):=\prod_{j=2}^{k+1}\frac{\phi_{\gg 1}\big(\frac{|\eta_j|}{|\xi|}\big)\phi_{\geq N_0}(|\eta_j|)}{|\eta_j|^2},$$
where ${\myvec{\xi}}=(\xi_1, \xi_2, \cdots, \xi_{k+2})$, $\xi=\sum_{l=1}^{k+2}\xi_l$ and $\eta_j=\sum_{l=j}^{k+2}\xi_l$.

Now, by the Fourier transformation, we have
\begin{align*}
\widehat{I_{n, k}^{h}}(\xi)=&i\int_0^t\int_{\xi=\sum_{l=1}^{k+2}\xi_l}e^{i\rho (|\xi|^2-|\xi_{k+2}|^2)}m(\myvec{\xi})
\phi_{\gg 1}\Big(\frac {|\xi_{k+2}|}{|\xi|}\Big)\phi_{\geq N_0}(|\xi_{k+2}|)\\
&\qquad\qquad \cdot \prod_{l=1}^{k+1}\widehat{\eta}(\xi_l)\widehat{I_{n-k-1}}(\xi_{k+2})d\xi_1 d\xi_2 \cdots d\xi_{k+1}d\rho.
\end{align*}
Note that
$$
\partial_t I_{n-k-1}= ie^{-it\Delta}(\eta e^{it\Delta}I_{n-k-2}),
$$
and $I_{n-k-1}(0, x)=0$ for $0\leq k\leq n-2$.

Hence, by the integration-by-parts, we get
\begin{align}
\widehat{I_{n, k}^{h}}(\xi)
=& i\int_{\xi=\sum_{l=1}^{k+2}\xi_l}\int_0^te^{i\rho |\xi|^2} m(\myvec{\xi})\frac{ \phi_{\gg 1}\Big(\frac {|\xi_{k+2}|}{|\xi|}\Big)\phi_{\geq N_0}(|\xi_{k+2}|)}{-i |\xi_{k+2}|^2}\notag\\
 &\qquad\cdot\prod_{l=1}^{k+1}\widehat{\eta}(\xi_l) \widehat{I_{n-k-1}}(\rho, \xi_{k+2})d(e^{-i\rho |\xi_{k+2}|^2})d\xi_1 d\xi_2 \cdots d\xi_{k+1} \notag\\
 =&i\int_{\xi=\sum_{l=1}^{k+2}\xi_l}e^{it (|\xi|^2-|\xi_{k+2}|^2)}m(\myvec{\xi})\frac{ \phi_{\gg 1}\Big(\frac {|\xi_{k+2}|}{|\xi|}\Big)\phi_{\geq N_0}(|\xi_{k+2}|)}{-i |\xi_{k+2}|^2}\notag\\
 &\qquad \cdot \prod_{l=1}^{k+1}\widehat{\eta}(\xi_l)\widehat{I_{n-k-1}}(t, \xi_{k+2})d\xi_1 d\xi_2 \cdots d\xi_{k+1}\label{Jnk1}\\
 &-i\int_0^t\int_{\xi=\sum_{l=1}^{k+2}\xi_l}e^{i\rho (|\xi|^2-|\xi_{k+2}|^2)}m(\myvec{\xi})\frac{i|\xi|^2 \phi_{\gg 1}\Big(\frac {|\xi_{k+2}|}{|\xi|}\Big)\phi_{\geq N_0}(|\xi_{k+2}|)}{-i |\xi_{k+2}|^2}\notag\\
 &\qquad \cdot \prod_{l=1}^{k+1}\widehat{\eta}(\xi_l)\widehat{I_{n-k-1}}(\rho, \xi_{k+2})d\xi_1 d\xi_2 \cdots d\xi_{k+1} d\rho\label{Jnk2}\\
  &-i\int_0^t\int_{\xi=\sum_{l=1}^{k+2}\xi_l}ie^{i\rho |\xi|^2}m(\myvec{\xi})\frac{ \phi_{\gg 1}\Big(\frac {|\xi_{k+2}|}{|\xi|}\Big)\phi_{\geq N_0}(|\xi_{k+2}|)}{-i |\xi_{k+2}|^2}\notag\\
 &\qquad \cdot \prod_{l=1}^{k+1}\widehat{\eta}(\xi_l)\mathscr{F}(\eta e^{i\rho \Delta }I_{n-k-2})(\xi_{k+2})d\xi_1 d\xi_2 \cdots d\xi_{k+1}d\rho\label{Jnk3}.
\end{align}
We can rewrite $I_{n, k}^{h}$ in the physical space as follows,
\begin{align}\label{Jnh-3}
I_{n, k}^{h}=\mathscr{F}^{-1}\eqref{Jnk1}+\mathscr{F}^{-1}\eqref{Jnk2}+\mathscr{F}^{-1}\eqref{Jnk3},
\end{align}
where
\EQn{\label{Jnh-4}
\mathscr{F}^{-1}\eqref{Jnk1}=&-e^{-it \Delta}\sum_{M\gg N}P_{N}T_N^k(\eta |\nabla|^{-2} e^{it \Delta}P_{M\geq N_0}I_{n-k-1}),\\
\mathscr{F}^{-1}\eqref{Jnk2}=&i|\nabla|^{2}\int_0^te^{-i\rho \Delta}\sum_{M\gg N}P_{N}T_N^k(\eta |\nabla|^{-2} e^{i\rho \Delta}P_{M\geq N_0}I_{n-k-1})d\rho,\\
\mathscr{F}^{-1}\eqref{Jnk3}=&i\int_0^t e^{-i\rho \Delta}\sum_{N}P_NT_N^{k+1}(\eta e^{i\rho\Delta}I_{n-k-2})d\rho=I_{n, k+1}.
}
Collecting \eqref{In}, \eqref{Jh12d}, \eqref{Jnh-3} and \eqref{Jnh-4}, we complete the proof of {\emph{$(A_2)$}}.

Finally, we turn to prove {\emph{$(A_3)$}}. Noting that if $k=n-1$, $I_{n-k-1}(0, x)=I_0(x)=v_0(x)$, we have
\begin{align}
\widehat{I_{n, n-1}^{h}}(\xi)
 =&i\int_{\xi=\sum_{l=1}^{n+1}\xi_l}e^{i\rho (|\xi|^2-|\xi_{n+1}|^2)}m(\myvec{\xi})\frac{ \phi_{\gg 1}\Big(\frac {|\xi_{n+1}|}{|\xi|}\Big)\phi_{\geq N_0}(|\xi_{n+1}|)}{-i |\xi_{n+1}|^2}\notag\\
 &\qquad\qquad \cdot \prod_{l=1}^{n}\widehat{\eta}(\xi_l)\widehat{v_{0}}(\xi_{n+1})d\xi_1 d\xi_2 \cdots d\xi_{n} \Big|_{\rho=0}^{\rho=t}\label{Jnk1-1}\\
 &-i\int_0^t\int_{\xi=\sum_{l=1}^{n+1}\xi_l}e^{i\rho (|\xi|^2-|\xi_{n+1}|^2)}m(\myvec{\xi})\frac{i|\xi|^2 \phi_{\gg 1}\Big(\frac {|\xi_{n+1}|}{|\xi|}\Big)\phi_{\geq N_0}(|\xi_{n+1}|)}{-i |\xi_{n+1}|^2}\notag\\
 &\qquad\qquad \cdot \prod_{l=1}^{n}\widehat{\eta}(\xi_l)\widehat{v_0}(\xi_{n+1})d\xi_1 d\xi_2 \cdots d\xi_{n} d\rho\label{Jnk2-1}.
\end{align}
We rewrite $I_{n, n-1}^{h}$ in the physical space as follows,
\EQn{\label{Jnh-5}
I_{n, n-1}^{h}=&-e^{-it \Delta}\sum_{M\gg N}P_{N}T_N^{n-1}(\eta |\nabla|^{-2} e^{it \Delta}P_{M\geq N_0}I_{0})\\
&+\sum_{M\gg N}P_{N}T_N^{n-1}(\eta |\nabla|^{-2} P_{M\geq N_0}I_{0})\\
&+i|\nabla|^{2}\int_0^te^{-i\rho \Delta}\sum_{M\gg N}P_{N}T_N^{n-1}(\eta |\nabla|^{-2} e^{i\rho \Delta}P_{M\geq N_0}I_{0})d\rho.
}
Note that the second term can be further expressed by the following
\EQn{\label{Jnh-6}
\sum_{M\gg N}P_{N}T_N^{n-1}(\eta |\nabla|^{-2} P_{M\geq N_0}I_{0})=\sum_{ N}P_{N}T_N^{n}v_0.
}
Hence, by \eqref{In}, \eqref{Jh12d}, \eqref{Jnh-5}, and \eqref{Jnh-6}, we complete the proof of {\emph{$(A_3)$}}. This ends the proof of this lemma.
\end{proof}

Next, we first give the estimate for the operator $T_{N}^k$. Before this, recall the definition of $r_1$ that
$$
\frac 1 {r_1}=1-\frac 1{r_0}-\frac 1r,$$
where $r_0= \infty-, \mbox{ if } d=2;  \quad   6, \mbox{ if }d=3$. Then, we have
$$
	r_1'=\left\{ \aligned
	&r-, \qquad d=2,
	\\
	&\frac {6r}{6+r}, \quad d=3.
	\endaligned
	\right.
$$
\begin{lem}\label{lem:Tnk-es}
Let $\frac d2<r\leq 2$ and $f\in L_x^{r_1'}$, under the same assumptions on $T_{N}^k$ as in Lemma \ref{In-form}, for any $k\in \N$ and any $N, N_0\in 2^{\N}$, there exists $C=C(r)>1$ such that the following inequality holds,
\begin{align}\label{Op-Tk-1}
\norm{T_{N}^kf}_{L_x^{r_1'}}\leq C^k \min\{(2^5 N)^{k(\frac dr-2)}, N_0^{k(\frac dr-2)}\}\norm{\eta}_{L_x^r}^k\norm{f}_{L_x^{r_1'}}.
\end{align}
\end{lem}
\begin{proof}
By the definition of $T_{N}$, H\"older's inequality, Lemmas \ref{lem:Bernstein} and \ref{lem:littlewood-Paley}, we have that there exists $C=C(r)>1$ such that,
\begin{align*}
\norm{T_{N}f}_{L_x^{r_1'}}\leq & C \norm{\eta}_{L_x^{r}} \sum_{M:  M\gg N, M\geq N_0}\norm{|\nabla|^{-2}P_{M} f}_{L_x^{r_0}}\notag \\
\leq & C\norm{\eta}_{L_x^{r}} \sum_{M:  M\geq 2^5 N, M\geq N_0} M^{\frac dr-2}\norm{P_{M} f}_{L_x^{r_1'}}\notag \\
\leq  & C \min\{(2^5 N)^{\frac dr-2}, N_0^{\frac dr-2}\}\norm{\eta}_{L_x^{r}}\norm{ f}_{L_x^{r_1'}}.
\end{align*}
Hence, for any $k\geq 1$, we have
\begin{align*}
\norm{T_{N}^k f}_{L_x^{r_1'}}\leq  \|T_{N}T_{N}^{k-1} f\|_{L_x^{r_1'}}\leq  C \min\{(2^5 N)^{\frac dr-2}, N_0^{\frac dr-2}\}\norm{\eta}_{L_x^{r}}\norm{ T_{N}^{k-1} f}_{L_x^{r_1'}}.
\end{align*}
Further, by iterating the above inequality, we get
\begin{align*}
\norm{T_{N}^k f}_{L_x^{r_1'}}\leq C^k \min\{(2^5 N)^{k(\frac dr-2)}, N_0^{k(\frac dr-2)}\}\norm{\eta}_{L_x^r}^k\norm{f}_{L_x^{r_1'}}.
\end{align*}
We finish the proof of this lemma.\end{proof}

%

Now, we give the following estimates for $I_n$ with $n\geq 1$. Recall that
$$
	(q_0, r_0)=\left\{ \aligned
	&(2+, \infty-), \quad d=2,
	\\
	&(2, 6), \qquad d=3.
	\endaligned
	\right.
$$
\begin{lem}\label{2d-v}
Let $s=\frac d2-\frac dr$ with $\frac d2<r\leq 2$, $\eta\in L_x^r$, and $v_0\in H_x^s$, then for any $N_0\in 2^{\N}$, there exist $0<T\leq N_0^{-2}$ and $C_0=C_0(r)>1$, such that the following inequalities hold,
\begin{align}\label{I0}
\norm{\langle \nabla\rangle ^ s e^{it \Delta }I_0}_{L_t^{\infty}L_x^2 \cap L_t^{q_0}L_x^{r_0}([0, T))}\leq C_0 \norm{v_0}_{H_x^s};
\end{align}
and for any $n\geq 1$,
\begin{align}\label{In-all-523}
\norm{\langle \nabla\rangle^s e^{it \Delta} I_{n}}_{L_t^{\infty}L_x^2\cap L_t^{q_0}L_x^{r_0}([0, T))}\leq  (2C_0)^{n}N_0^{n(\frac dr-2)}\|\eta\|_{L_x^r}^{n} \| v_0\|_{H_x^s}.
\end{align}
\end{lem}

\begin{proof}
$\bullet$ { Estimates on $I_{0}$.} Recall that $I_0=v_0$, by Strichartz's estimates, the validity of \eqref{I0} follows immediately.

$\bullet$ { Estimates on $I_{n}, n\geq 1$.} In what follows, for notational brevity, we always omit $\sup\limits_{h: \norm{h}_{L_x^2}\leq 1}$, $\sup\limits_{h: \norm{h}_{L_x^{r_0'}}\leq 1}$ in the front of dual's identity $\norm{\cdot}_{L_x^2}:=\sup\limits_{h:\norm{h}_{L_x^2}\leq 1}\langle\cdot, h\rangle$, $\norm{\cdot}_{L_x^{r_0}}:=\sup\limits_{h:\norm{h}_{L_x^{r_0'}}\leq 1}\langle\cdot, h\rangle$, respectively.
Recall that for any $n\geq 1$,
\begin{align}\label{In-es}
I_n=\sum_{k=0}^{n-1}( I_{n, k}^{(1)}+I_{n, k}^{(2)}+I_{n, k}^{(3)}+I_{n, k}^{(4)})+I_{n}^{(5)},
\end{align}
where
\begin{align*}
I_{n, k}^{(1)}=&i\int_0^te^{-i\rho\Delta }\sum_{N}P_{N}T_N^k(\eta e^{i\rho\Delta }P_{\leq N_0}I_{n-1-k})d\rho;\notag\\
I_{n, k}^{(2)}=&i\int_0^te^{-i\rho\Delta }\sum_{M\lsm N}P_{N}T_N^k(\eta  \cdot e^{i\rho\Delta }P_{M\geq N_0}I_{n-1-k})d\rho;\notag\\
I_{n, k}^{(3)}=&i|\nabla|^2\int_0^te^{-i\rho\Delta }\sum_{M\gg N}P_{N}T_N^k(\eta  \cdot |\nabla|^{-2}e^{i\rho\Delta }P_{M\geq N_0}I_{n-1-k})d\rho;\notag\\
I_{n, k}^{(4)}=&-e^{-it\Delta }\sum_{M\gg N}P_{N}T_N^k(\eta  \cdot |\nabla|^{-2} e^{it\Delta }P_{M\geq N_0}I_{n-1-k});\notag\\
I_{n}^{(5)}=&\sum_N P_N T_N^n v_0.
\end{align*}
{\emph{1) On $I_{n, k}^{(1)}$.}} By Strichartz's estimates, Sobolev's and H\"older's inequalities, and Lemma \ref{lem:Tnk-es}, we get
\begin{align}\label{In167}
\|\langle \nabla\rangle^s e^{it \Delta} I_{n, k}^{(1)}\|_{L_t^{\infty}L_x^2\cap L_t^{q_0}L_x^{r_0}}\lsm & \|\langle \nabla\rangle^s\sum_{N}P_{N} T_N^k(\eta e^{it\Delta }P_{\leq N_0}I_{n-1-k})\|_{L_t^{q_2'}L_x^{r_2'}}\notag\\
\lsm &\|\sum_{N}P_{N} T_N^k( \eta e^{it\Delta }P_{\leq N_0}I_{n-1-k})\|_{L_t^{q_2'}L_x^{r_1'}}\notag\\
\lsm &\|P_{\ll N_0} T_N^k( \eta e^{it\Delta }P_{\leq N_0}I_{n-1-k})\|_{L_t^{q_2'}L_x^{r_1'}}\notag\\
&+\sum_{N: N\geq 2^{-5}N_0}\|P_{N} T_N^k( \eta e^{it\Delta }P_{\leq N_0}I_{n-1-k})\|_{L_t^{q_2'}L_x^{r_1'}}\notag\\
\lsm &T^{\frac 1{q_2'}-\frac 1{q_0}}C^kN_0^{k(\frac dr-2)}\|\eta\|_{L_x^r}^k\|\eta e^{it\Delta }P_{\leq N_0}I_{n-1-k}\|_{L_t^{q_0}L_x^{r_1'}}\notag\\
\lsm &T^{1-\frac d4}C^kN_0^{k(\frac dr-2)}N_0^{\frac dr-\frac d2}\|\eta\|_{L_x^r}^{k+1}\|\langle \nabla\rangle^s e^{it\Delta }I_{n-1-k}\|_{L_t^{q_0}L_x^{r_0}},
\end{align}
where $(q_2, r_2)=(\frac {2r_0} d, \frac {2r_0}{r_0-2})$ is the Schr\"odinger admissible pair.

Now, we choose $T=T(N_0)>0$, such that
\begin{align}\label{741}
T^{\frac 12}N_0\leq 1.
\end{align}
 Thus, we get
 $$T^{1-\frac d4}N_0^{k(\frac dr-2)}N_0^{\frac dr-\frac d2}\leq N_0^{(k+1)(\frac dr-2)}.$$
 Therefore, we have
\begin{align}\label{In1}
\|\langle \nabla\rangle^s e^{it \Delta} I_{n, k}^{(1)}\|_{L_t^{\infty}L_x^2\cap L_t^{q_0}L_x^{r_0}}\lsm C^k N_0^{(k+1)(\frac dr-2)}\|\eta\|_{L_x^r}^{k+1}\|\langle \nabla\rangle^s e^{it\Delta }I_{n-1-k}\|_{L_t^{q_0}L_x^{r_0}}.
\end{align}

{\emph{2) On $I_{n, k}^{(2)}$.}}
When $r=2$, $s=\frac d2-\frac dr=0$. In this case, by Strichartz's estimates, we have
\begin{align}\label{Ink2-s0}
\norm{\langle \nabla\rangle^s e^{it \Delta} I_{n, k}^{(2)}}_{L_t^{\infty}L_x^2\cap L_t^{q_0}L_x^{r_0}}\lsm & \sum_{N:N\geq 2^{-5}N_0}\|P_{N} T_N^k( \eta e^{it\Delta }P_{\lsm N}P_{\geq N_0}I_{n-1-k})\|_{L_t^{q_1'}L_x^{r_1'}}\notag\\
\lsm &T^{\frac 1{q_1'}-\frac 1{q_0}}C^kN_0^{k(\frac dr-2)}\|\eta\|_{L_x^r}^k\|\eta e^{it\Delta }P_{\geq N_0}I_{n-1-k}\|_{L_t^{q_0}L_x^{r_1'}}\notag\\
\lsm &T^{1-\frac d4}C^kN_0^{k(\frac dr-2)}\|\eta\|_{L_x^r}^{k+1}\|e^{it\Delta }P_{\geq N_0}I_{n-1-k}\|_{L_t^{q_0}L_x^{r_0}}\notag\\
\lsm & C^k N_0^{(k+1)(\frac dr-2)}\|\eta\|_{L_x^r}^{k+1}\|\langle \nabla\rangle^s e^{it\Delta }I_{n-1-k}\|_{L_t^{q_0}L_x^{r_0}}.
\end{align}
When $\frac d2<r<2$, $s=\frac d2-\frac dr<0$. In this case, by the duality, Strichartz's estimates, and Lemmas \ref{lem:Schur}, \ref{lem:Tnk-es}, we have
\begin{align}\label{Ink2}
\norm{\langle \nabla\rangle^s e^{it \Delta} I_{n, k}^{(2)}}_{L_t^{\infty}L_x^2}\lsm& \normbb{\Big\langle\int_0^t \langle \nabla\rangle^se^{i(t-\rho) \Delta}\sum _{M\lsm N}P_{N}T_N^k(\eta   e^{i\rho\Delta }P_{M\geq N_0}I_{n-1-k})d\rho, h\Big\rangle}_{L_t^{\infty}}\notag\\
\lsm &\sum _{M\lsm N}\frac{\langle N \rangle^{s}}{\langle M \rangle^{s}} \norm{T_N^k(\eta   e^{it\Delta }\langle M \rangle^{s}P_{M\geq N_0}I_{n-1-k})}_{L_t^{q_1'}L_x^{r_1'}}\|P_N h\|_{L_x^2}\notag\\
\lsm &\sum _{M\lsm N}\frac{\langle N \rangle^{s}}{\langle M \rangle^{s}}C^k N_0^{k(\frac dr-2)}\|\eta\|_{L_x^r}^k \norm{\eta e^{it\Delta }\langle M \rangle^{s}P_{M\geq N_0}I_{n-1-k}}_{L_t^{q_1'}L_x^{r_1'}}\|P_N h\|_{L_x^2}\notag\\
\lsm &C^k N_0^{k(\frac dr-2)}\|\eta\|_{L_x^r}^k\norm{\eta \langle M\rangle^{s}e^{it \Delta}P_{M\geq N_0} I_{n-1-k}}_{l_M^2L_t^{q_1'}L_x^{r_1'}}\| P_N h\|_{l_N^2L_x^2}\notag\\
\lsm &T^{1-\frac d{2r}} C^k N_0^{k(\frac dr-2)}\|\eta\|_{L_x^r}^{k+1}\norm{\langle \nabla\rangle^{s}e^{it \Delta} I_{n-1-k}}_{L_t^{q_0}L_x^{r_0}}\notag\\
\lsm & C^k N_0^{(k+1)(\frac dr-2)}\|\eta\|_{L_x^r}^{k+1}\norm{\langle \nabla\rangle^{s}e^{it \Delta} I_{n-1-k}}_{L_t^{q_0}L_x^{r_0}}.
\end{align}
Similarly, $\norm{\langle \nabla\rangle^s e^{it \Delta} I_{n, k}^{(2)}}_{L_t^{q_0}L_x^{r_0}}$ and $\norm{\langle \nabla\rangle^s e^{it \Delta} I_{n, k}^{(2)}}_{L_t^{\infty}L_x^2}$ can be controlled by the same bound.
Hence, we get
\begin{align}\label{Ink2-3}
\norm{\langle \nabla\rangle^s e^{it \Delta} I_{n, k}^{(2)}}_{L_t^{\infty}L_x^2\cap L_t^{q_0}L_x^{r_0}}\lsm C^k N_0^{(k+1)(\frac dr-2)}\|\eta\|_{L_x^r}^{k+1}\norm{\langle \nabla\rangle^{s}e^{it \Delta} I_{n-1-k}}_{L_t^{q_0}L_x^{r_0}}.
\end{align}
{\emph{3) On $I_{n, k}^{(3)}$.}}
 Noting that $2+s>0$, by the duality, Strichartz's estimates, and Lemma \ref{lem:Schur}, we have
\begin{align}\label{Ink4}
\norm{\langle \nabla\rangle^s e^{it \Delta} I_{n, k}^{(3)}}_{L_t^{\infty}L_x^2}\lsm&\normbb{\Big\langle\langle \nabla\rangle^{2+s}\int_0^t e^{i(t-\rho) \Delta}\sum _{M\gg N}P_{N}T_N^k(\eta   |\nabla|^{-2}e^{i\rho\Delta }P_{M\geq N_0}I_{n-1-k})d\rho, h\Big\rangle}_{L_t^{\infty}}\notag\\
\lsm &\sum _{M\gg N} \frac{\langle N\rangle^{2+s}}{\langle M\rangle^{2+s}}\Big\|\int_0^t e^{i(t-\rho) \Delta}T_N^k(\eta  \langle M\rangle^{2+s}\notag\\
&\qquad\qquad\quad \cdot|\nabla|^{-2}e^{i\rho\Delta }P_{M\geq N_0}I_{n-1-k})d\rho\Big\|_{L_t^{\infty}L_x^2}\|P_N h\|_{L_x^2}\notag\\
\lsm &\sum _{M\gg N} \frac{\langle N\rangle^{2+s}}{\langle M\rangle^{2+s}} \|T_N^k(\eta  |\nabla|^{-2}e^{it\Delta }\langle M\rangle^{2+s}P_{M\geq N_0}I_{n-1-k})\|_{L_t^{q_1'}L_x^{r_1'}}\|P_N h\|_{L_x^2}\notag\\
\lsm & C^k N_0^{k(\frac d r-2)}\|\eta\|_{L_x^r}^{k}  \|\eta  |\nabla|^{-2}e^{it\Delta }\langle M\rangle^{2+s}P_{M}P_{\geq N_0}I_{n-1-k}\|_{l_M^2L_t^{q_1'}L_x^{r_1'}}\notag\\
\lsm &T^{1-\frac d{2r}} C^k N_0^{k(\frac d r-2)}\|\eta\|_{L_x^r}^{k+1}   \| |\nabla|^{-2}e^{it\Delta }\langle M\rangle^{2+s}P_{M}P_{\geq N_0}I_{n-1-k}\|_{L_t^{q_0}L_x^{r_0}l_M^2}\notag\\
\lsm & C^k  N_0^{(k+1)(\frac d r-2)}\|\eta\|_{L_x^r}^{k+1} \|  \langle \nabla\rangle^{s}e^{it\Delta}I_{n-1-k}\|_{L_t^{q_0}L_x^{r_0}}.
\end{align}
Similarly, $\norm{\langle \nabla\rangle^s e^{it \Delta} I_{n, k}^{(3)}}_{L_t^{q_0}L_x^{r_0}}$ and $\norm{\langle \nabla\rangle^s e^{it \Delta} I_{n, k}^{(3)}}_{L_t^{\infty}L_x^2}$ can be controlled by the same bound.
Hence, we have
\begin{align}\label{Ink4-2}
\norm{\langle \nabla\rangle^s e^{it \Delta} I_{n, k}^{(3)}}_{L_t^{\infty}L_x^{2}\cap L_t^{q_0}L_x^{r_0}}\lsm C^k  N_0^{(k+1)(\frac d r-2)}\|\eta\|_{L_x^r}^{k+1} \|  \langle \nabla\rangle^{s}e^{it\Delta}I_{n-1-k}\|_{L_t^{q_0}L_x^{r_0}}.
\end{align}
{\emph{4) On $I_{n, k}^{(4)}$.}}
 By the duality, Lemmas \ref{lem:Bernstein}, \ref{lem:Schur}, \ref{lem:littlewood-Paley} and \ref{lem:Tnk-es}, we have
\begin{align}\label{Ink3}
\norm{\langle \nabla\rangle^s e^{it \Delta} I_{n, k}^{(4)}}_{L_x^2}\lsm &\Big\langle  \langle \nabla\rangle^s\sum _{M\gg N}P_{N}T_N^k(\eta   |\nabla|^{-2} e^{it\Delta }P_{M\geq N_0}I_{n-1-k}), h \Big\rangle\notag\\
\lsm &  \sum _{M\gg N} \frac{\langle N \rangle^{\frac d{r_0}}}{\langle M \rangle^{\frac d{r_0}}}\|  \langle \nabla\rangle^{s-\frac d{r_0}}P_{N}T_N^k(\eta   |\nabla|^{-2} e^{it\Delta }\langle M \rangle^{\frac d{r_0}} P_{M\geq N_0}I_{n-1-k})\|_{L_x^2} \|P_Nh\|_{L_x^2}\notag\\
\lsm&\sum _{M\gg N} \frac{\langle N \rangle^{\frac d{r_0}}}{\langle M \rangle^{\frac d{r_0}}} \|  T_N^k(\eta   |\nabla|^{-2} e^{it\Delta }\langle M \rangle^{\frac d {r_0}} P_{M\geq N_0}I_{n-1-k})\|_{L_x^{r_1'}} \|P_Nh\|_{L_x^2}\notag\\
\lsm & C^k N_0^{k(\frac dr-2)}\|\eta\|_{L_x^r}^{k}  \sum _{M\gg N} \frac{\langle N \rangle^{\frac d{r_0}}}{\langle M \rangle^{\frac d{r_0}}} \|  \eta   |\nabla|^{-2} e^{it\Delta }\langle M \rangle^{\frac d{r_0}} P_{M\geq N_0}I_{n-1-k}\|_{L_x^{r_1'}} \|P_Nh\|_{L_x^2}\notag\\
\lsm&  C^k N_0^{k(\frac dr-2)}\|\eta\|_{L_x^r}^{k+1} \|    |\nabla|^{-2} e^{it\Delta }\langle M \rangle^{\frac d{r_0}} P_{M\geq N_0}I_{n-1-k}\|_{l_M^2 L_x^{r_0}}\notag\\
\lsm & C^k N_0^{k(\frac dr-2)}N_0^{\frac {d} r-2}\|\eta\|_{L_x^r}^{k+1} \|  \langle \nabla\rangle^{s}e^{it\Delta}I_{n-1-k}\|_{L_x^2}\notag\\
\lsm & C^k N_0^{(k+1)(\frac d r-2)}\|\eta\|_{L_x^r}^{k+1} \|  \langle \nabla\rangle^{s}e^{it\Delta}I_{n-1-k}\|_{L_x^2}.
\end{align}
Hence,
\begin{align}\label{Ink3-1}
\norm{\langle \nabla\rangle^s e^{it \Delta} I_{n, k}^{(4)}}_{L_t^{\infty}L_x^2}\lsm C^k N_0^{(\frac d r-2)(k+1)}\|\eta\|_{L_x^r}^{k+1} \|  \langle \nabla\rangle^{s}e^{it\Delta}I_{n-1-k}\|_{L_t^{\infty}L_x^2}.
\end{align}
Similarly, we have
\begin{align}\label{Ink3-2}
\norm{\langle \nabla\rangle^s e^{it \Delta}  I_{n, k}^{(4)}}_{L_x^{r_0}}\lsm & \Big\langle  \langle \nabla\rangle^s\sum _{M\gg N} P_{N}T_N^k(\eta   |\nabla|^{-2} e^{it\Delta }P_{M\geq N_0}I_{n-1-k}), h \Big\rangle\notag\\
\lsm &   \sum _{M\gg N} \frac{\langle N \rangle^{\frac d2}}{\langle M \rangle^{\frac d2}}\|  \langle \nabla\rangle^{s-\frac d2}P_{N}T_N^k(\eta   \langle M\rangle^{\frac d2} |\nabla|^{-2} e^{it\Delta }P_{M\geq N_0}I_{n-1-k})\|_{L_x^{r_0}} \|P_Nh\|_{L_x^{r_0'}}\notag\\
\lsm& \sum _{M\gg N} \frac{\langle N \rangle^{\frac d2}}{\langle M \rangle^{\frac d2}}\| P_{N}T_N^k(\eta   \langle M\rangle^{\frac d2} |\nabla|^{-2} e^{it\Delta }P_{M\geq N_0}I_{n-1-k})\|_{L_x^{r_1'}}\|P_Nh\|_{L_x^{r_0'}}\notag\\
\lsm &C^k N_0^{k(\frac d r-2)}\|\eta\|_{L_x^r}^{k} \sum _{M\gg N} \frac{\langle N \rangle^{\frac d2}}{\langle M \rangle^{\frac d2}} \|  \eta   \langle M\rangle ^{\frac d2} |\nabla|^{-2} e^{it\Delta }P_{M\geq N_0}I_{n-1-k}\|_{L_x^{r_1'}}\|P_Nh\|_{L_x^{r_0'}}\notag\\
\lsm  & C^k N_0^{k(\frac d r-2)}\|\eta\|_{L_x^r}^{k}  \|  \eta   \langle M\rangle^{\frac d2} |\nabla|^{-2} e^{it\Delta }P_{M\geq N_0}I_{n-1-k}\|_{L_x^{r_1'}l_M^2}\notag\\
\lsm& C^k N_0^{k(\frac d r-2)}\|\eta\|_{L_x^r}^{k+1} \|  \langle M\rangle^{\frac d2} |\nabla|^{-2} e^{it\Delta }P_{M\geq N_0}I_{n-1-k}\|_{L_x^{r_0}l_M^2}\notag\\
\lsm &C^k N_0^{(k+1)(\frac dr-2)}\|\eta\|_{L_x^r}^{k+1}  \|  \langle \nabla\rangle^{s}e^{it\Delta}I_{n-1-k}\|_{L_x^{r_0}}.
\end{align}
Hence,
\begin{align}\label{Ink3-3}
\norm{\langle \nabla\rangle^s e^{it \Delta} I_{n, k}^{(4)}}_{L_t^{q_0}L_x^{r_0}}\lsm C^k N_0^{(k+1)(\frac d r-2)}\|\eta\|_{L_x^r}^{k+1}  \|  \langle \nabla\rangle^{s}e^{it\Delta}I_{n-1-k}\|_{L_t^{q_0}L_x^{r_0}}.
\end{align}
By \eqref{Ink3-1} and \eqref{Ink3-3}, we have
\begin{align}\label{Ink3-4}
\norm{\langle \nabla\rangle^s e^{it \Delta} I_{n, k}^{(4)}}_{L_t^{\infty}L_x^{2}\cap L_t^{q_0}L_x^{r_0}}\lsm C^k N_0^{(k+1)(\frac d r-2)}\|\eta\|_{L_x^r}^{k+1}  \|  \langle \nabla\rangle^{s}e^{it\Delta}I_{n-1-k}\|_{L_t^{\infty}L_x^{2}\cap L_t^{q_0}L_x^{r_0}}.
\end{align}
{\emph{5) On $I_{n}^{(5)}$.}} By the definition of $T_N^n v_0$, we rearrange the sequence $M_1, M_2, \cdots,  M_n$ into $M, M_1, \cdots,$ $ M_{n-1}$ and obtain that
\begin{align}\label{Ink5}
T_N^n v_0=T_N^{n-1} \big(\sum_{M:  M\gg N, M\geq N_0}\eta |\nabla|^{-2}P_{M} v_0\big).
\end{align}
By Strichartz's estimate, we get
\begin{align}\label{Ink5-1}
\norm{\langle \nabla\rangle^s e^{it \Delta} I_{n}^{(5)}}_{L_t^{\infty}L_x^{2}\cap L_t^{q_0}L_x^{r_0}}\lsm \Big\| \langle \nabla\rangle^s  \sum_{M\gg  N  }P_N T_N^{n-1} (\eta |\nabla|^{-2}P_{M\geq N_0} v_0)\Big\|_{L_x^2}.
\end{align}
An argument parallel to \eqref{Ink3} yields
\begin{align}\label{In5-2}
\norm{\langle \nabla\rangle^s e^{it \Delta} I_{n}^{(5)}}_{L_t^{\infty}L_x^{2}\cap L_t^{q_0}L_x^{r_0}}\lsm  C^{n-1}N_0^{n(\frac d r-2)}\|\eta\|_{L_x^r}^{n} \|  \langle \nabla\rangle^{s}v_0\|_{L_x^2}.
\end{align}
Combining the estimates \eqref{In1}, \eqref{Ink2-3}, \eqref{Ink4-2}, \eqref{Ink3-4}, and \eqref{In5-2}, for any $n\geq 1$, we have that there exists $C_0=C_0(r)>1$, such that
\begin{align}\label{In-523}
\norm{\langle \nabla\rangle^s e^{it \Delta}I_n}_{L_t^{\infty}L_x^{2}\cap L_t^{q_0}L_x^{r_0}}
 \leq &  \sum_{k=0}^{n-1}C_0^{k+1} N_0^{(k+1)(\frac d r-2)}\|\eta\|_{L_x^r}^{k+1} \|  \langle \nabla\rangle^{s}e^{it\Delta}I_{n-1-k}\|_{L_t^{\infty}L_x^2\cap L_t^{q_0}L_x^{r_0}}\notag \\
 &+C_0^{n}N_0^{n(\frac d r-2)}\|\eta\|_{L_x^r}^{n} \| v_0\|_{H_x^s}.
\end{align}
When $n=1$, then $k=0$. By \eqref{I0}, we have
\begin{align}\label{es-I1}
\|\langle \nabla\rangle^s e^{it\Delta }I_{1}\|_{L_t^{\infty}L_x^{2}\cap L_t^{q_0}L_x^{r_0}}\leq &2C_0N_0^{\frac dr-2}\|\eta\|_{L_x^r}\| v_0\|_{H_x^s}.
\end{align}
Next, we use the induction method to prove \eqref{In-all-523}.
Now, for any $1\leq j\leq n-1$, we assume the following estimate holds:
\begin{align}\label{Inj-1}
\norm{\langle \nabla\rangle^s e^{it \Delta}I_j}_{L_t^{\infty}L_x^{2}\cap L_t^{q_0}L_x^{r_0}}
 \leq (2C_0 )^{j}N_0^{j(\frac dr-2)}\|\eta\|_{L_x^r}^j\| v_0\|_{H_x^s}.
\end{align}
By \eqref{I0}, \eqref{In-523}, and \eqref{Inj-1}, we have
\begin{align}\label{Inn1}
\norm{\langle \nabla\rangle^s e^{it \Delta}I_n}_{L_t^{\infty}L_x^{2}\cap L_t^{q_0}L_x^{r_0}}
 \leq & \sum_{k=0}^{n-1}C_0^{k+1}(2C_0)^{n-k-1}N_0^{n(\frac dr-2)}\|\eta\|_{L_x^r}^n\| v_0\|_{H_x^s}\notag \\
 &+C_0^{n}N_0^{n(\frac dr-2)}\|\eta\|_{L_x^r}^{n} \| v_0\|_{H_x^s}\notag \\
 \leq & \big(2^{n-1}C_0^{n}\sum_{k=0}^{n-1} 2^{-k}+C_0^{n}\big)N_0^{n(\frac dr-2)}\|\eta\|_{L_x^r}^{n} \| v_0\|_{H_x^s}\notag \\
 \leq & (2C_0)^{n}N_0^{n(\frac dr-2)}\|\eta\|_{L_x^r}^{n} \| v_0\|_{H_x^s}.
\end{align}
This finishes the proof of this lemma.
\end{proof}

Now, we are in a position to give the proof of Proposition \ref{prop:v-2d}.
\begin{proof}[Proof of Proposition \ref{prop:v-2d}]
Let $N\in \N$, and denote
$S_N=\sum_{n=0}^{N}e^{it\Delta}I_n$.
 By Lemma \ref{2d-v}, we have that for any $N_0\in 2^{\N}$, there exists $0<T\leq N_0^{-2}$, such that
\begin{align*}
\|S_N\|_{L_t^{\infty}H_x^s([0, T))}\leq& C_0\norm{v_0}_{H_x^s}+\sum_{n=1}^{N} (2C_0)^{n}N_0^{n(\frac dr-2)}\|\eta\|_{L_x^r}^{n} \| v_0\|_{H_x^s}.
\end{align*}
Noting that $\frac dr-2<0$, and taking $N_0=N_0(\|\eta\|_{L_x^r})\in 2^{\N}$ large enough, we have
$$
2C_0N_0^{\frac dr-2}\|\eta\|_{L_x^r}<1.
$$
Thus for any $N$,
\begin{align}\label{bound-Sn}
\|S_N\|_{L_t^{\infty}H_x^s([0, T))}\lsm  \norm{v_0}_{H_x^s}.
\end{align}
By the similar way as above, we also obtain that
\begin{align}\label{con-Sn}
\|S_N - S_{N'}\|_{L^\infty_t H^{s}_x} \to 0, \mbox{ as } N,N' \to \infty.
\end{align}
Hence, by \eqref{bound-Sn} and \eqref{con-Sn}, there exists $S\in L_t^{\infty}H_x^s([0, T)) $ such that
$$
S=\lim_{N\to \infty}S_N,  \mbox{ in }L_t^{\infty}H_x^s([0, T)).
$$
Next, we explain that $S$ is the unique solution of the equation \eqref{eq:NLS-222}.
Recalling the well-posedness of the equation \eqref{eq:NLS} in $H_x^{\alpha}$, where $\alpha$ is defined as follows
$$
	\alpha=\left\{ \aligned
	&4-\frac {4}r-2\varepsilon_0, \quad d=2,
	\\
	&\frac 92-\frac {6}r, \quad d=3,
	\endaligned
	\right.
$$
we can easily obtain that the solution $v$ of the equation \eqref{eq:NLS-222} belongs to $L_t^{q_1}W_x^{\alpha-2, r_1}$. Furthermore, we also conclude that
$$
\lim_{N\to \infty}S_N=v,  \mbox{ in }L_t^{q_1}W_x^{\alpha-2, r_1} + L_t^{\infty}H_x^{\alpha-2}.
$$
Noting $\alpha-2<s<0$, we have
\begin{align*}
L_t^{\infty}H_x^{s}\subset L_t^{\infty}H_x^{\alpha-2}\subset L_t^{q_1}W_x^{\alpha-2, r_1} + L_t^{\infty}H_x^{\alpha-2} .
\end{align*}
By the uniqueness of the limit, we conclude that
$$
v=S\in L_t^{\infty}H_x^s([0, T)).
$$
Thus, this completes the proof of this proposition.

%

\subsection{Ill-posedness in $H_x^{2+\frac d2-\frac dr +}(\R^d)$}

In this part, we aim to prove that there exists some $\eta\in L_x^{r}(\R^d)$ with $\frac d2<r\leq2$ and $d=2, 3$, such that the equation \eqref{eq:NLS} is ill-posed in $H_x^{2+\frac d2-\frac dr+}(\R^d)$.

For our purpose, we set the parameters $M, N\geq 1$, which shall be determined later. Next, for any $\gamma>2+\frac d2-\frac dr$, we choose the initial data
\EQ{
u_0(x):=\mathscr{F}^{-1}\big(N^{-\frac d2-\gamma }\chi_{N\leq |\cdot|\leq 2N}(\xi)\big)(x).
}
Then we have
\EQ{\|u_0\|_{H_x^{\gamma}}^2=\|\langle\xi\rangle^\gamma \widehat{u_0}(\xi)\|_{L_{\xi}^2}^2\lsm N^{-d}\int_{N}^{2N}\lambda^{d-1} d\lambda\sim 1.
}
On the other hand, we choose the potential
\EQ{
\eta(x)=M^{\frac dr}\mathscr{F}^{-1}\big(\chi_{1\leq |\cdot|\leq 2}(\xi)\big)(M x).
}
Hence, for any $\frac d2<r\leq 2$, we have
\EQ{
\norm{\eta}_{L_x^r}\lsm \norm{\mathscr{F}^{-1}\big(\chi_{1\leq |\cdot|\leq 2}(\xi)\big)}_{L_x^r}\lsm 1.
}
Moreover, we have
\EQ{
\widehat{\eta}(\xi)=M^{-d+\frac dr}\chi_{1\leq |\cdot|\leq 2}\big(\frac{\xi} M\big).
}
Now, we define
\EQ{
A (u_0)(t)\triangleq\int_0^t e^{-i\rho\Delta}(\eta e^{i\rho\Delta}u_0)d\rho.
}
We aim to prove that
\EQ{
\sup\limits_{t\in [0,1]}\|A(u_0)\|_{H_x^{\gamma}(\R^d)}\rightarrow \infty, \mbox{ as }M\rightarrow \infty.
}
Next, we set
$$
t\triangleq \frac 1{M^2},
$$
and
$$
\Omega=\{\xi: \sqrt{\frac {\pi}{3}}M\leq|\xi|\leq \sqrt{\frac {\pi}{2}}M\}.
$$
By the integration-by-parts and the choice of $u_0$ and $\eta$, we have
\begin{align}\label{621656-2D}
\widehat{A (u_0)}(\xi)=&M^{-d+\frac dr}N^{-\frac d2-\gamma }\int_0^t \int_{\xi=\xi_1+\xi_2}e^{is(|\xi|^2-|\xi_2|^2)}\chi_{1\leq |\cdot|\leq 2}\big(\frac{\xi_1} M\big)\chi_{N\leq |\cdot|\leq 2N}(\xi_2)d\xi_2ds\notag\\
=&M^{-d+\frac dr}N^{-\frac d2-\gamma }\int_{\xi=\xi_1+\xi_2}\frac{e^{it(|\xi|^2-|\xi_2|^2)}-1}{i(|\xi|^2-|\xi_2|^2)}\chi_{1\leq |\cdot|\leq 2}\big(\frac{\xi_1} M\big)\chi_{N\leq |\cdot|\leq 2N}(\xi_2)d\xi_2.
\end{align}
Hence, taking the real part of $\widehat{A (u_0)}(\xi)$, we have
\begin{align}\label{6121-2D}
\mbox{Re}\widehat{A (u_0)}(\xi)=M^{-d+\frac dr}N^{-\frac d2-\gamma }\int_{\xi=\xi_1+\xi_2}\frac{\sin[t(|\xi|^2-|\xi_2|^2)]}{|\xi|^2-|\xi_2|^2}\chi_{1\leq |\cdot|\leq 2}\big(\frac{\xi_1} M\big)\chi_{N\leq |\cdot|\leq 2N}(\xi_2)d\xi_2.
\end{align}
By the mean value theorem, we have
\EQ{
\sin[t(|\xi|^2-|\xi_2|^2)]=\sin (t|\xi|^2)+O(t|\xi_2|^2).
}
Now, we take $N \ll M$.  Noting that if $\xi \in \Omega$, then  $t|\xi|^2\in [ \frac{\pi}{3}, \frac{\pi}{2}]$, which further implies $\sin (t|\xi|^2)\geq \frac 12$. Moreover, by $N\ll M$, we have $t|\xi_2|^2\sim \frac {N^2}{M^2}\ll 1$. Hence, we can get that
\begin{align}\label{621-2D}
\sin[t(|\xi|^2-|\xi_2|^2)]\geq \frac 14.
\end{align}
By the estimates \eqref{6121-2D} and \eqref{621-2D}, we obtain
\begin{align}\label{6122-2D}
\mbox{Re}\widehat{A (u_0)}(\xi)\geq \frac 14 M^{-d+\frac dr}N^{-\frac d2-\gamma }\int_{\xi=\xi_1+\xi_2}\frac{1}{|\xi|^2-|\xi_2|^2}\chi_{1\leq |\cdot|\leq 2}\big(\frac{\xi_1} M\big)\chi_{N\leq |\cdot|\leq 2N}(\xi_2)d\xi_2.
\end{align}
Further, noting $\mbox{Re}\widehat{A (u_0)}(\xi)>0$, the above inequality yields that
\begin{align*}
\|A (u_0)\|_{H_x^{\gamma}(\R^d)}= \norm{\langle\xi\rangle^{\gamma}\widehat{A(u_0)}(\xi)}_{L_{\xi}^2(\R^d)}\geq \norm{\langle\xi\rangle^{\gamma}\mbox{Re}\widehat{A(u_0)}(\xi)}_{L_{\xi}^2(\R^d)}.
\end{align*}
Finally, combing the estimate \eqref{6122-2D}, we get
\begin{align*}
\|A (u_0)\|_{H_x^{\gamma}(\R^d)}\geq& C M^{-d+\frac dr+\gamma}N^{-\frac d2-\gamma }M^{-2}\norm{\int\chi_{N\leq |\cdot|\leq 2N}(\xi_2)d\xi_2}_{L_{\xi}^2(\Omega)}\\
\geq & CM^{-2-d+\frac dr+\gamma}N^{-\frac d2-\gamma}N^{d}M^{\frac d2}\\
\geq & C(N) M^{\gamma-(2+\frac d2-\frac dr)},
\end{align*}
where $C(N)>0$ is a finite constant. Hence, by $\gamma>2+\frac d2-\frac dr$, we obtain that for any $T>0$,
\begin{align}\label{Bu0-2d}
\sup\limits_{t\in [0,T]}\|A (u_0)\|_{H_x^{\gamma}(\R^d)}\rightarrow \infty, \mbox{ as }M\rightarrow \infty.
\end{align}
The proof of ill-posedness is done by applying Lemma \ref{ill-posed}. We finish the proof of Theorem \ref{theorem-one}.\end{proof}

\section{Critical case: the proof of Theorem \ref{theorem-two}}

In this section, we establish the global well-posedness of \eqref{eq:NLS} in $H_x^{\frac d2-}(\R^d)$ with critical potentials $\eta\in L_x^{\frac d2}$, where $d=3, 4$.
 We firstly prove the global well-posedness in $H_x^{\frac d2-1-}(\R^d)$, then apply the transform $v=\partial_t u$ to improve the above global well-posedness to $H_x^{\frac d2-}$.



\subsection{Low regularity for a critical index}\label{low-34}

	
We first establish the global well-posedness of \eqref{eq:NLS} in the space $H_x^{\alpha}(\R^d)$, where $\alpha=\frac d2-1-$. This constitutes a weak regularity result, as smaller index $\alpha < \gamma^* = \frac d2$.
\begin{prop}\label{weak-critical}
Let $d=3, 4$, $r=\frac d2$, and $\eta\in L_x^{r}(\R^d)$, then \eqref{eq:NLS} is globally well-posed in $H_x^{\frac d2-1-}(\R^d)$.
\end{prop}

Next we provide the key estimates to prove the above result. Define the auxiliary spaces $Y(I)$ for $I\subset\R^+$ by the following norm,
	\EQn{\label{1127-1652}
		\|u\|_{Y(I)}:=\|u\|_{L_t^{\infty}L_x^{2}(I)} + \|u\|_{L_t^{2}L_x^{\frac{2d}{d-2}}(I)}+\|u\|_{L_{t, x}^{\frac{2(d+2)}{d}}(I)}.
	}

	\subsubsection{Boundary terms}
	\begin{lem}[Boundary terms]\label{lem:nonlinear-estimate-boundary-1127}
		Let $r=\frac d2$, and $\alpha=\frac d2-1-$. Let $I \subset\R^+$ be an interval containing $0$. Then, for any $N_0\in2^\N$,
		\EQn{\label{32511-1127}
			\normb{e^{it\Delta}\mathcal{B}(\langle\nabla \rangle^{-2+\alpha}\eta, u_0)}_{Y(I)} \lsm \|P_{\geq N_0}\eta\|_{L_x^{r}}\|\langle\nabla \rangle^{\alpha}u\|_{Y(I)},
		}
		and
		\EQn{\label{32512-1127}
			\norm{\mathcal{B}(\langle\nabla \rangle^{-2+\alpha}\eta, u(t))}_{Y(I)} \lsm  \|P_{\geq N_0} \eta\|_{L_x^{r}}\|\langle\nabla \rangle^{\alpha}u\|_{Y(I)}.
		}
	\end{lem}
	\begin{proof}
First of all, by Sobolev's inequality, we have
\EQn{\label{3251-1127}
			 \norm{\langle\nabla \rangle^{-2+\alpha}P_{\geq N_0}\eta}_{L_x^{\frac d{\alpha}}}
\lsm & \norm{\langle\nabla \rangle^{-2+\alpha+\frac dr-\alpha}P_{\geq N_0}\eta}_{L_x^{r}}\\
\lsm & \norm{P_{\geq N_0}\eta}_{L_x^{r}}.
		}
		Noting $\alpha<\frac d2$, by Strichartz's estimates, Sobolev's inequality, Lemma \ref{lem:Coifman-Meyer}, we further derive that
		\EQn{\label{32513-1127}
			\normb{e^{it\Delta}\mathcal{B}(\langle\nabla \rangle^{-2+\alpha}\eta, u_0)}_{Y(I)} \lsm & \normb{\mathcal{B}(\langle\nabla \rangle^{-2+\alpha}\eta, u_0)}_{L_x^2}\\
\lsm & \norm{\langle\nabla \rangle^{-2+\alpha}P_{\geq N_0}\eta}_{L_x^{\frac d{\alpha}}}\norm{u_0}_{L_x^{\frac d{\frac d2-\alpha}}}\\
\lsm & \norm{P_{\geq N_0}\eta}_{L_x^{r}}\norm{u}_{L_t^{\infty}H_x^{\alpha}}.
		}
	This proves \eqref{32511-1127}. Next, we give the proof of \eqref{32512-1127}. Following the same approach as in \eqref{32513-1127}, we obtain
		\EQ{
			\norm{\mathcal{B}(\langle\nabla \rangle^{-2+\alpha}\eta, u(t))}_{L_t^{\infty}L_x^{2}} \lsm  \|P_{\geq N_0} \eta\|_{L_x^{r}}\|u\|_{L_t^{\infty}H_x^{\alpha}}.
		}
		Next, we consider the term $\norm{\mathcal{B}(\langle\nabla \rangle^{-2+\alpha}\eta, u(t))}_{L_t^{2}L_x^{\frac {2d}{d-2}}}$.
Noting that $\alpha\cdot \frac{2d}{d-2}<d$, by Lemma \ref{lem:Coifman-Meyer}, Sobolev's inequality, and \eqref{3251-1127},
\EQ{
			\norm{\mathcal{B}(\langle\nabla \rangle^{-2+\alpha}\eta, u(t))}_{L_t^{2}L_x^{\frac {2d}{d-2}}} \lsm & \norm{\langle\nabla \rangle^{-2+\alpha}P_{\geq N_0}\eta}_{L_x^{\frac d{\alpha}}}\norm{u}_{L_t^2L_x^{\frac {2d}{d-2-2\alpha}}}\\
\lsm & \norm{P_{\geq N_0}\eta}_{L_x^{r}}\norm{\langle\nabla \rangle^{\alpha} u}_{L_t^2L_x^{\frac {2d}{d-2}}}.
		}
Finally, we consider the term $\norm{\mathcal{B}(\langle\nabla \rangle^{-2+\alpha}\eta, u(t))}_{L_{t, x}^{\frac{2(d+2)}{d}}}$. Noting that $\alpha\cdot \frac{2(d+2)}{d}<d$, by the same way as above,
\EQ{
			\norm{\mathcal{B}(\langle\nabla \rangle^{-2+\alpha}\eta, u(t))}_{L_{t, x}^{\frac{2(d+2)}{d}}} \lsm & \norm{\langle\nabla \rangle^{-2+\alpha}P_{\geq N_0}\eta}_{L_x^{\frac d{\alpha}}}\norm{u}_{L_t^{\frac{2(d+2)}{d}}L_x^{\frac 1{\frac d{2(d+2)}-\frac {\alpha}d}}}\\
\lsm & \norm{P_{\geq N_0}\eta}_{L_x^{r}}\norm{ \langle\nabla \rangle^{\alpha} u}_{L_{t, x}^{\frac{2(d+2)}{d}}}.
		}
By the above three inequalities, we finish the proof of \eqref{32512-1127}.

	\end{proof}

\subsubsection{Resonance term and low frequency term}
	\begin{lem}\label{lem:nonlinear-estimate-resonance-1127}
Let $r=\frac d2$, and $\alpha=\frac d2-1-$. Let $I =[0, T)\subset\R^+$. Then, for any $N_0\in2^\N$,
		\EQ{
			\normbb{\int_{0}^t e^{i(t-\rho)\Delta}\langle\nabla \rangle^{\alpha}\mathcal{R}(\eta, u) d\rho}_{Y(I)} \lsm  (T^{\frac 1{2+d}}N_0\|\eta\|_{L_x^{r}}+\norm{P_{\geq N_0} \eta}_{L_x^r})\|\langle\nabla \rangle^{\alpha}u\|_{Y(I)}.
		}
	\end{lem}
	\begin{proof}
Recalling the definition of $\mathcal{R}(\eta, u)$ in Definition \ref{defn:normal-form}, we have
\EQnnsub{
\Big\|\int_{0}^t& e^{i(t-\rho)\Delta}\langle\nabla \rangle^{\alpha}\mathcal{R}(\eta, u) d\rho\Big\|_{Y(I)}\notag\\
\lsm &\Big\|\int_{0}^t e^{i(t-\rho)\Delta}\langle\nabla \rangle^{\alpha} P_{\leq N_0}(\eta u)d\rho\Big\|_{Y(I)}\label{1128-1}\\
&+\Big\|\int_{0}^t e^{i(t-\rho)\Delta}\langle\nabla \rangle^{\alpha} P_{\geq N_0}\sum_{M\gtrsim N}P_N(\eta P_Mu)d\rho\Big\|_{Y(I)}.\label{1128-2}
}
For \eqref{1128-1}, noting $\frac 2{2+d}<\alpha$, by H\"{o}lder's and Sobolev's inequalities, and Lemmas  \ref{lem:Bernstein}, \ref{lem:strichartz}, we get
\begin{align}\label{1128-5}
\eqref{1128-1}\lsm N_0^{\alpha}\|\eta u\|_{L_t^{2}L_x^{\frac {2d}{d+2}}}
\lsm &T^{\frac 1{2+d}}N_0^{\alpha} \|\eta\|_{L_x^{\frac d2}}\|\langle\nabla \rangle^{\frac 2{2+d}} u\|_{L_{t,x}^{\frac {2(d+2)}{d}}}\notag\\
\lsm &T^{\frac 1{2+d}}N_0^{\alpha} \|\eta\|_{L_x^{\frac d2}}\|\langle\nabla \rangle^{\alpha} u\|_{L_{t,x}^{\frac {2(d+2)}{d}}}.
\end{align}
For \eqref{1128-2}, analogously to \eqref{32514-22} and \eqref{32517-22}, we have
\begin{align}\label{1128-6}
\eqref{1128-2}\lsm &\norm{\eta \langle M\rangle^{\alpha} P_M u}_{l_M^2L_t^2L_x^{\frac {2d}{d+2}}}\notag\\
\lsm &\norm{P_{\geq N_0}\eta \langle M\rangle^{\alpha} P_M u}_{l_M^2L_t^2L_x^{\frac {2d}{d+2}}}\notag\\
&+\norm{P_{< N_0}\eta \langle M\rangle^{\alpha} P_M u}_{l_M^2L_t^2L_x^{\frac {2d}{d+2}}}.
\end{align}
For the first term in \eqref{1128-6}, by  H\"older's inequality and Lemma \ref{lem:littlewood-Paley}, we get
\begin{align}\label{1128-7}
\norm{P_{\geq N_0}\eta \langle M\rangle^{\alpha} P_M u}_{l_M^2L_t^2L_x^{\frac {2d}{d+2}}}\lsm \norm{P_{\geq N_0}\eta}_{L_x^{\frac d2}}\norm{\langle\nabla \rangle^{\alpha} u}_{L_t^2L_x^{\frac {2d}{d-2}}}.
\end{align}
For the second term in \eqref{1128-6}, by H\"older's inequality and Lemmas \ref{lem:Bernstein}, \ref{lem:littlewood-Paley}, we get
\begin{align}\label{1128-8}
\norm{P_{< N_0}\eta \langle M\rangle^{\alpha} P_M u}_{l_M^2L_t^2L_x^{\frac {2d}{d+2}}}\lsm &\norm{P_{< N_0}\eta}_{L_x^{d}}\norm{\langle\nabla \rangle^{\alpha} u}_{L_{t,x}^2}\notag\\
\lsm &T^{\frac 12}N_0\norm{\eta}_{L_x^{\frac d2}}\norm{\langle\nabla \rangle^{\alpha} u}_{L_t^{\infty}L_x^2}.
\end{align}
Hence, this lemma follows from \eqref{1128-5}-\eqref{1128-8}.
\end{proof}

	\subsubsection{High-order terms}
	\begin{lem}[Higher order terms]\label{lem:nonlinear-estimate-higher-order-22}
Let $r=\frac d2$, and $\alpha=\frac d2-1-$. Let $I \subset\R^+$ be an interval containing $0$. Then
\begin{align*}
\normbb{\int_0^te^{i(t-\rho)\Delta}\mathcal{B}(\langle\nabla \rangle^{-2+\alpha}\eta, \eta u)(\rho, x)d\rho}_{Y(I)} \lsm&  \|P_{\geq N_0}\eta\|_{L_x^r}\|\eta\|_{L_x^r}\|\langle\nabla \rangle^{\alpha}u\|_{Y(I)}.
\end{align*}
\end{lem}
\begin{proof}
Noting that $\alpha\cdot \frac{2d}{d-2}<d$, by Lemmas \ref{lem:strichartz}, \ref{lem:Coifman-Meyer}, Sobolev's inequality, and \eqref{3251-1127}, we have
\begin{align*}
\normbb{\int_0^te^{i(t-\rho)\Delta}\mathcal{B}(\langle\nabla \rangle^{-2+\alpha}\eta, \eta u)(\rho, x)d\rho}_{Y(I)}\lsm & \norm{\mathcal{B}(\langle\nabla \rangle^{-2+\alpha}\eta, \eta u)}_{L_t^2L_x^{\frac {2d}{d+2}}}\\
\lsm &  \norm{\langle\nabla \rangle^{-2+\alpha}P_{\geq N_0}\eta}_{L_x^{\frac d{\alpha}}}\norm{\eta}_{L_x^{\frac d2}}\norm{u}_{L_t^2L_x^{\frac {2d}{d-2-{2\alpha}}}}\\
\lsm &\norm{P_{\geq N_0}\eta}_{L_x^{\frac d{2}}}\norm{\eta}_{L_x^{\frac d2}}\norm{\langle\nabla \rangle^{\alpha}u}_{L_t^2L_x^{\frac {2d}{d-2}}}.
\end{align*}
This proves this Lemma.
\end{proof}

We are now in a position to prove Proposition \ref{weak-critical}.
\begin{proof}[\bf{Proof of Proposition \ref{weak-critical}}]
By the above several lemmas, we can establish the global well-posedness of \eqref{eq:NLS} in $H_x^{\frac d2-1-}$ and for any $T>0$,
\EQn{\label{1130-11}
\norm{u}_{L_t^{\infty}H_x^{\frac d2-1-}([0, T)\times \R^3)}\leq C(T) \norm{u_0}_{H_x^{\frac d2-1-}}\leq C(T)\norm{u_0}_{H_x^{\frac d2-}}.
}
The detailed proof follows the arguments presented in subsection \ref{low-one} and is therefore omitted here again.
\end{proof}

\subsection{Global well-posedness in $H_x^{\frac d2-}(\R^d), d=3, 4$}

As the normal form method described above cannot further improve the regularity, we adopt another way to improve the well-posedness of \eqref{eq:NLS} to $H_x^{\frac d2-}$. Fix $\epsilon_0>0$, and denote $\frac d2-=\frac d2-\epsilon_0$, and
$s=\frac d2-2-\epsilon_0$, where $d=3, 4$.
Let $v=\partial_t u$, from the equation \eqref{eq:NLS}, $v$ satisfies the following equation
\begin{equation}\label{eq:NLS-1130333}
		\left\{ \aligned
		&i\partial_t v+\Delta v+\eta v=0, \qquad t\in(0, T) \mbox{ and } x\in \R^d,
		\\
		&v(0,x)=i(\Delta u_0+\eta u_0)\triangleq v_0.
		\endaligned
		\right.
	\end{equation}
 Next, we give two necessary facts.

\noindent $\bullet$ Claim 1: $v_0\in H_x^{s}$.

 In fact, by $ u_0\in H_x^{\frac d2-\epsilon_0}$ (that is $ u_0\in H_x^{2+s}$), $\eta \in L_x^{\frac d2}$, the Sobolev and H\"{o}lder inequalities, we get
\begin{align}\label{1201-1332}
\norm{v_0}_{H_x^{s}}&=\norm{\Delta u_0+\eta u_0}_{H_x^{s}}\notag\\
&\lsm \norm{u_0}_{H_x^{{2+s}}}+\norm{\eta u_0}_{L_x^{\frac d{\frac d2-s}}}\notag\\
&\lsm  \norm{u_0}_{H_x^{2+s}}+\norm{\eta }_{L_x^{\frac d2}}\norm{ u_0}_{L_x^{\frac d {\frac d2-2-s}}}\notag\\
&\lsm \norm{u_0}_{H_x^{2+s}}+\norm{\eta }_{L_x^{\frac d2}}\norm{ u_0}_{H_x^{2+s}}.
\end{align}
\noindent $\bullet$ Claim 2: $v\in C([0, T); H_x^{s})$ implies $u\in C([0, T); H_x^{\frac d2-\epsilon_0})$.

Indeed, by the high and low frequency decomposition,
\begin{align}\label{high-low-1130}
\norm{u}_{L_t^{\infty}H_x^{\frac d2-\epsilon_0}}\lsm \norm{P_{<1}u}_{L_t^{\infty}H_x^{\frac d2-\epsilon_0}} +\norm{P_{\geq 1}u}_{L_t^{\infty}H_x^{\frac d2-\epsilon_0}}.
\end{align}
For $ \norm{P_{<1}u}_{L_t^{\infty}H_x^{\frac d2-\epsilon_0}} $, by Lemma \ref{lem:Bernstein} and \eqref{1130-11}, we have
\begin{align}\label{low-1130}
\norm{P_{<1}u}_{L_t^{\infty}H_x^{\frac d2-\epsilon_0}}\leq C \norm{u}_{L_t^{\infty}H_x^{\frac d2-1-\epsilon_0}}\leq C(T) \norm{u_0}_{H_x^{2+s}}.
\end{align}
Next, we consider the high frequency term $ \norm{P_{\geq 1}u}_{L_t^{\infty}H_x^{\frac d2-\epsilon_0}} $. Noting that
\EQ{
\Delta u=-i v-\eta u,
}
by the Sobolev inequality, we have
\begin{align}\label{1201-1210}
\norm{P_{\geq 1}u}_{L_t^{\infty}H_x^{\frac d2-\epsilon_0}}\leq & \norm{P_{\geq 1}v}_{L_t^{\infty}H_x^{\frac d2-2-\epsilon_0}}+\norm{P_{\geq 1}(\eta u)}_{L_t^{\infty}H_x^{\frac d2-2-\epsilon_0}}\notag\\
\leq  & C(T) \norm{v_0}_{H_x^{s}}+C\norm{\eta u}_{L_t^{\infty}L_x^{\frac d{2+\epsilon_0}}}.
\end{align}
For the second term in \eqref{1201-1210}, by Lemma \ref{lem:Bernstein}, Sobolev's inequality, and \eqref{1130-11},
\begin{align}\label{high-1201-1255}
\norm{\eta u}_{L_t^{\infty}L_x^{\frac d{2+\epsilon_0}}}
\lsm& \norm{P_{\geq N_0}\eta }_{L_x^{\frac d2}}\norm{u}_{L_t^{\infty}L_x^{\frac d{\epsilon_0}}}+
\norm{P_{<N_0}\eta }_{L_x^d}\norm{u}_{L_t^{\infty}L_x^{\frac d{1+\epsilon_0}}}\notag\\
\lsm &\norm{P_{\geq N_0}\eta }_{L_x^{\frac d2}}\norm{u}_{L_t^{\infty}H_x^{\frac d2-\epsilon_0}}+N_0\norm{\eta }_{L_x^{\frac d2}}\norm{u_0}_{H_x^{2+s}},
\end{align}
where $N_0\in 2^{\N}$ is a large constant decided later.

By \eqref{1201-1210} and \eqref{high-1201-1255}, we have
\begin{align}\label{1201-1309}
\norm{P_{\geq 1}u}_{L_t^{\infty}H_x^{\frac d2-\epsilon_0}}\leq & C(T)  \norm{v_0}_{H_x^{s}}+C\norm{P_{\geq N_0}\eta }_{L_x^{\frac d2}}\norm{u}_{L_t^{\infty}H_x^{\frac d2-\epsilon_0}}\notag\\
&+CN_0\norm{\eta }_{L_x^{\frac d2}}\norm{u_0}_{H_x^{2+s}}.
\end{align}
Hence, by \eqref{high-low-1130}, \eqref{low-1130} and \eqref{1201-1309}, we further get
\begin{align}\label{high-low-1328}
\norm{u}_{L_t^{\infty}H_x^{\frac d2-\epsilon_0}}\leq &C(T) \norm{u_0}_{H_x^{2+s}}+ C(T)\norm{v_0}_{H_x^{s}}\notag\\
&+C\norm{P_{\geq N_0}\eta }_{L_x^{\frac d2}}\norm{u}_{L_t^{\infty}H_x^{\frac d2-\epsilon_0}}+CN_0\norm{\eta }_{L_x^{\frac d2}}\norm{u_0}_{H_x^{2+s}}.
\end{align}
We take $N_0=N_0(\norm{\eta}_{L_x^{\frac d2}})$, such that
$$
C\norm{P_{\geq N_0}\eta }_{L_x^{\frac d2}}\leq \frac 12.
$$
Further, combining \eqref{1201-1332} and \eqref{high-low-1328}, we conclude that
\begin{align*}
\norm{u}_{L_t^{\infty}H_x^{\frac d2-\epsilon_0}}\leq C(T, \norm{\eta }_{L_x^{\frac d2}}, \norm{u_0}_{H_x^{2+s}}).
\end{align*}
This gives the proof of Claim 2.

Based on the above two facts, it suffices to prove the equation \eqref{eq:NLS-1130333} is globally well-posed in $H_x^{s}$.
For our purpose, we firstly give the following result via an iterated Duhamel construction.
\begin{prop}\label{prop:v-34d}
Let $ N\in \N$, and $S_N\triangleq\sum_{n=0}^{N}e^{it\Delta}I_n$, where the terms $I_n$  are defined recursively by
\begin{align*}
I_n=i\int_0^te^{-i\rho \Delta}(\eta e^{i\rho \Delta}I_{n-1})d \rho, {\mbox{ for }}n\geq1; \quad I_0=v_0.
\end{align*} Let $
	s=\frac d2-2-\epsilon_0
$, where $\epsilon_0>0$, $d=3, 4$. Then there exist $T=T(\|\eta\|_{L_x^r})>0$, and $v\in C([0, T); H_x^s(\R^d))$, such that
$$
\lim_{N\to \infty}S_N=v, \mbox{ in }H_x^s,
$$
where $v$ is the unique solution to equation \eqref{eq:NLS-1130333}.
\end{prop}

Based on this proposition, we can easily obtain the global well-posedness of \eqref{eq:NLS} in $H_x^{\frac d2-}(\R^d)$. Here we omit the details, which can be referred to Section \ref{low-regularity-830}.

\subsection{Proof of Proposition \ref{prop:v-34d}}\label{v-spacetime-830-34d}
Next, let us focus on the proof of Proposition \ref{prop:v-34d}. The idea of proof is similar to the proof of Proposition \ref{prop:v-2d}. However, when dealing with case critical potentials, we need to perform a detailed frequency decomposition, which is essential for achieving smallness.

Next, we give the estimates for the operator $T_{N}^k$ defined in lemma \ref{In-form}.  Recall the definitions of $T_N$ that
  $$
T_{N}f=\eta \sum_{M:  M\gg N, M\geq N_0}|\nabla|^{-2}P_{M} f,
$$
where $M, N, N_0\in 2^{\N}$. We also recall the definition of $T_N^k$ that
for any $k\in \N$,
\begin{align*}
T_N^k f=(T_N)^kf, \mbox{ with } T_N^0 f=f.
\end{align*}
Firstly, we have the following estimates of $T_N$.
\begin{lem}\label{lem:TN-34D}
Let $d=3,4$, $\eta\in L_x^{\frac d2}$ and $f\in L_x^{\frac {2d}{d+2}}$. Then for any $N\in 2^{\N}$, and any $M_0, N_0\in 2^{\N}$ satisfying $M_0\leq N_0$, the following inequalities hold,
\begin{align}\label{Op-T1-34d}
\norm{T_{N}f}_{L_x^{\frac {2d}{d+2}}}\lsm \big(\norm{P_{\geq M_0}\eta}_{L_x^{\frac d2}}+ M_0^{\frac 14} N_0^{-\frac 14}\norm{\eta}_{L_x^{\frac d2}}\big)\norm{f}_{L_x^{\frac {2d}{d+2}}},
\end{align}
and
\begin{align}\label{Op-T1-34d22}
\norm{T_{N}f}_{L_x^{\frac {d}{2}-}}\lsm \big(\norm{P_{\geq M_0}\eta}_{L_x^{\frac d2}}+ M_0 N_0^{-1}\norm{\eta}_{L_x^{\frac d2}}\big)\norm{f}_{L_x^{\frac {d}{2}-}}.
\end{align}
\begin{proof}
By the definition of $T_{N}$, H\"older's and Sobolev's inequalities, and Lemma \ref{lem:Bernstein}, we have
\begin{align*}
\norm{T_{N}f}_{L_x^{\frac {2d}{d+2}}}\lsm &\norm{P_{\geq M_0}\eta}_{L_x^{\frac d2}} \norm{|\nabla|^{-2}P_{\gg N}P_{\geq N_0} f}_{L_x^{\frac {2d}{d-2}}}\notag \\
&+\norm{P_{< M_0}\eta}_{L_x^{\infty}} \norm{|\nabla|^{-2}P_{\gg N}P_{\geq N_0} f}_{L_x^{\frac {2d}{d+2}}}\notag \\
\lsm &\norm{P_{\geq M_0}\eta}_{L_x^{\frac d2}} \norm{f}_{L_x^{\frac {2d}{d+2}}}+M_0^2\norm{\eta}_{L_x^{\frac d2}}\cdot N_0^{-2}\norm{f}_{L_x^{\frac {2d}{d+2}}} \notag \\
\lsm & \big(\norm{P_{\geq M_0}\eta}_{L_x^{\frac d2}}+ M_0^2N_0^{-2}\norm{\eta}_{L_x^{\frac d2}}\big)\norm{f}_{L_x^{\frac {2d}{d+2}}},
\end{align*}
and
\begin{align*}
\norm{T_{N}f}_{L_x^{\frac {d}{2}-}}\lsm &\norm{P_{\geq M_0}\eta}_{L_x^{\frac d2}} \norm{|\nabla|^{-2}P_{\gg N}P_{\geq N_0} f}_{L_x^{\infty-}}\notag \\
&+\norm{P_{< M_0}\eta}_{L_x^{\infty}} \norm{|\nabla|^{-2}P_{\gg N}P_{\geq N_0} f}_{L_x^{\frac {d}{2}-}}\notag \\
\lsm &\norm{P_{\geq M_0}\eta}_{L_x^{\frac d2}} \norm{f}_{L_x^{\frac {d}{2}-}}+M_0^2\norm{\eta}_{L_x^{\frac d2}}\cdot N_0^{-2}\norm{f}_{L_x^{\frac {d}{2}-}} \notag \\
\lsm & \big(\norm{P_{\geq M_0}\eta}_{L_x^{\frac d2}}+ M_0^2N_0^{-2}\norm{\eta}_{L_x^{\frac d2}}\big)\norm{f}_{L_x^{\frac {d}{2}-}}.
\end{align*}
We finish the proof of this lemma.\end{proof}

\end{lem}

Applying this lemma, and iteration, we have the following estimates directly.
\begin{lem}\label{lem:Tnk-es34d}
Under the same assumptions as in Lemma \ref{lem:TN-34D}, then for any $k\in \N$, the following estimates hold,
\begin{align}\label{Op-Tk-34d}
\norm{T_{N}^kf}_{L_x^{\frac {2d}{d+2}}}\lsm \big(\norm{P_{\geq M_0}\eta}_{L_x^{\frac d2}}+ M_0^{\frac 14} N_0^{-\frac 14}\norm{\eta}_{L_x^{\frac d2}}\big)^k\norm{f}_{L_x^{\frac {2d}{d+2}}},
\end{align}
and
\begin{align}\label{Op-Tk-34d22}
\norm{T_{N}^kf}_{L_x^{\frac {d}{2}-}}\lsm \big(\norm{P_{\geq M_0}\eta}_{L_x^{\frac d2}}+ M_0 N_0^{-1}\norm{\eta}_{L_x^{\frac d2}}\big)^k\norm{f}_{L_x^{\frac {d}{2}-}}.
\end{align}
\end{lem}

Now, we give the following estimates of $I_n$ with $n\geq 1$.

\begin{lem}\label{34d-v}
Let $s=\frac d2-2-\epsilon_0$ with $d=3, 4$, $\eta\in L_x^{\frac d2}$, and $v_0\in H_x^s$, then for any $M_0, N_0\in 2^{\N}$ satisfying $M_0\leq N_0$, and $T=T(N_0)>0$, the following inequalities hold,
\begin{align}\label{I0-34d}
\norm{\langle \nabla\rangle ^ s e^{it \Delta }I_0}_{L_t^{\infty}L_x^2 \cap L_t^{2}L_x^{\frac {2d}{d-2}}([0, T))}\lsm \norm{v_0}_{H_x^s};
\end{align}
and for any $n\geq 1$,
\begin{align}\label{In-all-34d}
\norm{\langle \nabla\rangle^s e^{it \Delta} I_{n}}_{L_t^{\infty}L_x^2\cap L_t^{2}L_x^{\frac{2d}{d-2}}([0, T))}\lsm  \big( \norm{P_{\geq M_0}\eta}_{L_x^{\frac d2}}+ M_0^{\frac 14}N_0^{-\frac 14}\norm{\eta}_{L_x^{\frac d2}}\big)^n \| v_0\|_{H_x^s}.
\end{align}
\end{lem}

\begin{proof}
$\bullet$ { Estimates on $I_{0}$.} Recall that $I_0=v_0$, by Strichartz's estimates, the validity of \eqref{I0-34d} follows immediately.

$\bullet$ { Estimates on $I_{n}, n\geq 1$.} In what follows, for notational brevity, we always omit $\sup\limits_{h: \norm{h}_{L_x^2}\leq 1}$ in the front of dual's identity and denote $\norm{\cdot}_{L_x^2}:=\langle\cdot, h\rangle$. Similarly $\norm{\cdot}_{L_x^{\frac{2d}{d-2}}}:=\langle\cdot, h\rangle$.

Moreover, for notational brevity, we also denote
$$
\varepsilon_0(M_0, N_0):=\norm{P_{\geq M_0}\eta}_{L_x^{\frac d2}}+ M_0^{\frac 14}N_0^{-\frac 14}\norm{\eta}_{L_x^{\frac d2}}.
$$
We remark that $\varepsilon_0(M_0, N_0)\ll 1$, if $M_0\gg 1$ and $N_0\gg M_0$.

Recall the structural form of $I_n$ in Lemma \ref{In-form}, that for any $n\geq 1$,
\begin{align}\label{In-es}
I_n=\sum_{k=0}^{n-1}( I_{n, k}^{(1)}+I_{n, k}^{(2)}+I_{n, k}^{(3)}+I_{n, k}^{(4)})+I_{n}^{(5)},
\end{align}
where
\begin{align*}
I_{n, k}^{(1)}=&i\int_0^te^{-i\rho\Delta }\sum_{N}P_{N}T_N^k(\eta e^{i\rho\Delta }P_{\leq N_0}I_{n-1-k})d\rho;\notag\\
I_{n, k}^{(2)}=&i\int_0^te^{-i\rho\Delta }\sum_{M\lsm N}P_{N}T_N^k(\eta  \cdot e^{i\rho\Delta }P_{M\geq N_0}I_{n-1-k})d\rho;\notag\\
I_{n, k}^{(3)}=&i|\nabla|^2\int_0^te^{-i\rho\Delta }\sum_{M\gg N}P_{N}T_N^k(\eta  \cdot |\nabla|^{-2}e^{i\rho\Delta }P_{M\geq N_0}I_{n-1-k})d\rho;\notag\\
I_{n, k}^{(4)}=&-e^{-it\Delta }\sum_{M\gg N}P_{N}T_N^k(\eta  \cdot |\nabla|^{-2} e^{it\Delta }P_{M\geq N_0}I_{n-1-k});\notag\\
I_{n}^{(5)}=&\sum_N P_N T_N^n v_0.
\end{align*}
{\emph{1) On $I_{n, k}^{(1)}$.}} By Strichartz's estimates, we get
\begin{align}
\|\langle \nabla\rangle^s e^{it \Delta} I_{n, k}^{(1)}\|_{L_t^{\infty}L_x^2\cap L_t^{2}L_x^{\frac {2d}{d-2}}}\lsm & \|\langle \nabla\rangle^s\sum_{N: N<N_0}P_{N} T_N^k(\eta e^{it\Delta }P_{\leq N_0}I_{n-1-k})\|_{L_t^{\tilde{q}'}L_x^{\tilde{r}'}}\label{In167-34d}\\
&+\|\langle \nabla\rangle^s\sum_{N: N\geq N_0}P_{N} T_N^k(\eta e^{it\Delta }P_{\leq N_0}I_{n-1-k})\|_{L_t^{2}L_x^{\frac{2d}{d+2}}}\label{In167-34d2},
\end{align}
where $(\tilde{q}, \tilde{r})=(\frac 2{\frac d2-1-\epsilon_0}, \frac {d}{1+\epsilon_0})$ is the Schr\"odinger admissible pair.

By the H\"older and Sobolev inequalities, and Lemmas \ref{lem:Bernstein} and \ref{lem:Tnk-es34d}, we have
\begin{align}\label{1018-134d}
\eqref{In167-34d}\lsm &\sum_{N: N<N_0}\|T_N^k( \eta e^{it\Delta }P_{\leq N_0}I_{n-1-k})\|_{L_t^{\tilde{q}'}L_x^{\frac{2d}{d+2}}}\notag\\
\lsm &\sum_{N: N<N_0} T^{\frac 1{\tilde{q}'}-\frac 12}\varepsilon_0^k(M_0, N_0)\| \eta e^{it\Delta }P_{\leq N_0}I_{n-1-k}\|_{L_t^{2}L_x^{\frac{2d}{d+2}}}\notag\\
\lsm & N_0T^{\frac 1{\tilde{q}'}-\frac 12}\varepsilon_0^k(M_0, N_0) \|\eta e^{it\Delta }P_{\leq N_0}I_{n-1-k}\|_{L_t^{2}L_x^{\frac{2d}{d+2}}}.
\end{align}
By Lemma \ref{lem:Bernstein}, we get
\begin{align}\label{34d-1021}
 &\|\eta e^{it\Delta }P_{\leq N_0}I_{n-1-k}\|_{L_t^{2}L_x^{\frac{2d}{d+2}}}\notag\\
 \lsm &\norm{P_{\geq M_0}\eta}_{L_x^{\frac d2}}\| e^{it\Delta }P_{\leq N_0}I_{n-1-k}\|_{L_t^{2}L_x^{\frac{2d}{d-2}}}\notag\\
& +\norm{P_{< M_0}\eta}_{L_x^{d}}\| e^{it\Delta }P_{\leq N_0}I_{n-1-k}\|_{L_{t, x}^{2}}\notag\\
\lsm  &N_0^{-s}\big(\norm{P_{\geq M_0}\eta}_{L_x^{\frac d2}}+T^{\frac 12}M_0 \norm{\eta}_{L_x^{\frac d2}}\big)  \| \langle \nabla\rangle^se^{it\Delta }I_{n-1-k}\|_{ L_t^{\infty}L_x^2 \cap L_t^{2}L_x^{\frac{2d}{d-2}}}.
\end{align}
Hence, by the above two estimates, we get
\begin{align*}
 \eqref{In167-34d}\lsm  T^{\frac 1{\tilde{q}'}-\frac 12} & N_0^{1-s}\varepsilon_0^k(M_0, N_0)\big(\norm{P_{\geq M_0}\eta}_{L_x^{\frac d2}}\\
 &+T^{\frac 12}N_0 M_0 N_0^{-1}\norm{\eta}_{L_x^{\frac d2}}\big)  \| \langle \nabla\rangle^se^{it\Delta }I_{n-1-k}\|_{ L_t^{\infty}L_x^2 \cap L_t^{2}L_x^{\frac{2d}{d-2}}}.
\end{align*}
Now, we temporarily take $T=T(N_0)>0$, such that
\begin{align}\label{741-34d}
T^{\frac 1{\tilde{q}'}-\frac 12}N_0^{1-s}+T^{\frac 12}N_0\leq 1.
\end{align}
Further, we conclude that
\begin{align}\label{1018-1}
\eqref{In167-34d}\lsm \varepsilon_0^{k+1}(M_0, N_0) \|\langle \nabla\rangle^s e^{it\Delta }I_{n-1-k}\|_{ L_t^{\infty}L_x^2\cap L_t^{2}L_x^{\frac{2d}{d-2}}}.
\end{align}
Similarly, by H\"older's inequality, and Lemmas \ref{lem:Bernstein}, \ref{lem:Tnk-es34d}, and \eqref{34d-1021}, \eqref{741-34d}, we have
\begin{align}\label{1018-2}
\eqref{In167-34d2}\lsm &\sum_{N: N\geq N_0}N^s\|T_N^k( \eta e^{it\Delta }P_{\leq N_0}I_{n-1-k})\|_{L_t^{2}L_x^{\frac{2d}{d+2}}}\notag\\
\lsm &\sum_{N: N\geq N_0}N^s \varepsilon_0^k(M_0, N_0)\| \eta e^{it\Delta }P_{\leq N_0}I_{n-1-k}\|_{L_t^{2}L_x^{\frac{2d}{d+2}}}\notag\\
\lsm & N_0^s\varepsilon_0^k(M_0, N_0) \|\eta e^{it\Delta }P_{\leq N_0}I_{n-1-k}\|_{L_t^{2}L_x^{\frac{2d}{d+2}}}\notag\\
\lsm & \varepsilon_0^{k+1}(M_0, N_0) \|\langle \nabla\rangle^s e^{it\Delta }I_{n-1-k}\|_{ L_t^{\infty}L_x^2\cap L_t^{2}L_x^{\frac{2d}{d-2}}}.
\end{align}
Hence, by \eqref{In167-34d}, \eqref{In167-34d2}, \eqref{1018-1}, and \eqref{1018-2}, we obtain
\begin{align}\label{In1-34d}
\|\langle \nabla\rangle^s e^{it \Delta} I_{n, k}^{(1)}\|_{L_t^{\infty}L_x^2\cap L_t^{2}L_x^{\frac {2d}{d-2}}}\lsm \varepsilon_0^{k+1}(M_0, N_0)\|\langle \nabla\rangle^s e^{it\Delta }I_{n-1-k}\|_{L_t^{\infty}L_x^2\cap L_t^2 L_x^{\frac{2d}{d-2}}}.
\end{align}

{\emph{2) On $I_{n, k}^{(2)}$.}}
By the duality, Strichartz's estimates, and Lemmas \ref{lem:Schur}, \ref{lem:Tnk-es34d}, we have
\begin{align}\label{Ink2-34d10}
\norm{\langle \nabla\rangle^s e^{it \Delta} I_{n, k}^{(2)}}_{L_t^{\infty}L_x^2}\lsm& \normbb{\Big\langle\int_0^t \langle \nabla\rangle^se^{i(t-\rho) \Delta}\sum _{M\lsm N}P_{N}T_N^k(\eta   e^{i\rho\Delta }P_{M\geq N_0}I_{n-1-k})d\rho, h\Big\rangle}_{L_t^{\infty}}\notag\\
\lsm &\sum _{M\lsm N}\frac{\langle N \rangle^{s}}{\langle M \rangle^{s}} \norm{T_N^k(\eta   e^{it\Delta }\langle M \rangle^{s}P_{M\geq N_0}I_{n-1-k})}_{L_t^{2}L_x^{\frac{2d}{d+2}}}\|P_N h\|_{L_x^2}\notag\\
\lsm &\sum _{M\lsm N}\frac{\langle N \rangle^{s}}{\langle M \rangle^{s}}\varepsilon_0^k(M_0, N_0)
\norm{\eta e^{it\Delta }\langle M \rangle^{s}P_{M\geq N_0}I_{n-1-k}}_{L_t^{2}L_x^{\frac {2d}{d+2}}}\|P_N h\|_{L_x^2}\notag\\
\lsm &\varepsilon_0^k(M_0, N_0) \norm{\eta e^{it \Delta}\langle M\rangle^{s}P_{M\geq N_0} I_{n-1-k}}_{l_M^2L_t^{2}L_x^{\frac{2d}{d+2}}}.
\end{align}
By Lemmas \ref{lem:Bernstein}, \ref{lem:littlewood-Paley}, Sobolev's inequality, and \eqref{741-34d}, we have
\begin{align*}
&\|  \eta  e^{it\Delta }\langle M \rangle^{s} P_{M\geq N_0}I_{n-1-k}\|_{l_M^2 L_t^2L_x^{\frac {2d}{d+2}}}\\
\lsm &\| P_{\geq M_0} \eta\|_{L_x^{\frac d2}} \| e^{it\Delta }\langle \nabla\rangle^{s} P_{\geq N_0}I_{n-1-k}\|_{ L_t^2L_x^{\frac {2d}{d-2}}}\\
&+\| P_{< M_0} \eta\|_{L_x^{d}} \|  e^{it\Delta }\langle \nabla \rangle^{s} P_{\geq N_0}I_{n-1-k}\|_{ L_{t, x}^{2}}\\
\lsm &\big(\| P_{\geq M_0} \eta\|_{L_x^{\frac d2}}+ T^{\frac 12}M_0\|\eta\|_{L_x^{\frac d2}}\big)\| \langle \nabla\rangle^{s}  e^{it\Delta }I_{n-1-k}\|_{ L_t^{\infty}L_x^2 \cap L_t^{2}L_x^{\frac{2d}{d-2}}}\\
\lsm &\big(\| P_{\geq M_0} \eta\|_{L_x^{\frac d2}}+ M_0N_0^{-1}\|\eta\|_{L_x^{\frac d2}}\big)\| \langle \nabla\rangle^{s}  e^{it\Delta }I_{n-1-k}\|_{ L_t^{\infty}L_x^2 \cap L_t^{2}L_x^{\frac{2d}{d-2}}}.
\end{align*}
Combining the above two estimates, we conclude that
\begin{align}\label{Ink2-34d}
\norm{\langle \nabla\rangle^s e^{it \Delta} I_{n, k}^{(2)}}_{L_t^{\infty}L_x^2}\lsm \varepsilon_0^{k+1}(M_0, N_0)\| \langle \nabla\rangle^{s}  e^{it\Delta }I_{n-1-k}\|_{ L_t^{\infty}L_x^2 \cap L_t^{2}L_x^{\frac{2d}{d-2}}}.
\end{align}
Similarly, $\norm{\langle \nabla\rangle^s e^{it \Delta} I_{n, k}^{(2)}}_{L_t^{2}L_x^{\frac{2d}{d-2}}}$ and $\norm{\langle \nabla\rangle^s e^{it \Delta} I_{n, k}^{(2)}}_{L_t^{\infty}L_x^2}$ can be controlled by the same bound.
Hence, we get
\begin{align}\label{Ink2-34d}
\norm{\langle \nabla\rangle^s e^{it \Delta} I_{n, k}^{(2)}}_{L_t^{\infty}L_x^2\cap L_t^{2}L_x^{\frac{2d}{d-2}}}\lsm \varepsilon_0^{k+1}(M_0, N_0) \| \langle \nabla\rangle^{s}  e^{it\Delta }I_{n-1-k}\|_{ L_t^{\infty}L_x^2 \cap L_t^{2}L_x^{\frac{2d}{d-2}}}.
\end{align}
{\emph{3) On $I_{n, k}^{(3)}$.}}
 Noting that $2+s>0$, by the duality, Strichartz's estimates, and Lemma \ref{lem:Schur}, we have
\begin{align}\label{Ink4-34d}
\norm{\langle \nabla\rangle^s e^{it \Delta} I_{n, k}^{(3)}}_{L_t^{\infty}L_x^2}\lsm&\normbb{\Big\langle\langle \nabla\rangle^{2+s}\int_0^t e^{i(t-\rho) \Delta}\sum _{M\gg N}P_{N}T_N^k(\eta   |\nabla|^{-2}e^{i\rho\Delta }P_{M\geq N_0}I_{n-1-k})d\rho, h\Big\rangle}_{L_t^{\infty}}\notag\\
\lsm &\sum _{M\gg N} \frac{\langle N\rangle^{2+s}}{\langle M\rangle^{2+s}}\Big\|\int_0^t e^{i(t-\rho) \Delta}T_N^k(\eta \langle M\rangle^{2+s}\notag\\
&\qquad\qquad\cdot |\nabla|^{-2}e^{i\rho\Delta } P_{M\geq N_0}I_{n-1-k})d\rho\Big\|_{L_t^{\infty}L_x^2}\|P_N h\|_{L_x^2}\notag\\
\lsm &\sum _{M\gg N} \frac{\langle N\rangle^{2+s}}{\langle M\rangle^{2+s}} \|T_N^k(\eta \langle M\rangle^{2+s} |\nabla|^{-2}e^{it\Delta }P_{M\geq N_0}I_{n-1-k})\|_{L_t^{2}L_x^{\frac{2d}{d+2}}}\|P_N h\|_{L_x^2}\notag\\
\lsm &\varepsilon_0^k(M_0, N_0) \|\eta  |\nabla|^{-2}e^{it\Delta }\langle M\rangle^{2+s}P_{M\geq N_0}I_{n-1-k}\|_{l_M^2L_t^{2}L_x^{\frac{2d}{d+2}}}.
\end{align}
By Lemmas \ref{lem:Bernstein}, \ref{lem:littlewood-Paley}, and Sobolev's inequality, we have
\begin{align*}
&\|  \eta  |\nabla|^{-2} e^{it\Delta }\langle M \rangle^{2+s} P_{M\geq N_0}I_{n-1-k}\|_{l_M^2L_t^2 L_x^{\frac {2d}{d+2}}}\\
\lsm &\| P_{\geq M_0} \eta\|_{L_x^{\frac d2}} \| |\nabla|^{-2} e^{it\Delta }\langle \nabla\rangle^{2+s} P_{\geq N_0}I_{n-1-k}\|_{L_t^2 L_x^{\frac {2d}{d-2}}}\\
&+\| P_{< M_0} \eta\|_{L_x^{d}} \| |\nabla|^{-2} e^{it\Delta }\langle \nabla \rangle^{2+s} P_{\geq N_0}I_{n-1-k}\|_{ L_{t, x}^{2}}\\
\lsm &\big(\| P_{\geq M_0} \eta\|_{L_x^{\frac d2}}+ T^{\frac 12}M_0\|\eta\|_{L_x^{\frac d2}}\big)\| \langle \nabla\rangle^{s}  e^{it\Delta }I_{n-1-k}\|_{ L_t^{\infty}L_x^2 \cap L_t^{2}L_x^{\frac{2d}{d-2}}}.
\end{align*}
Noting  $T^{\frac 12}<N_0^{-1}$, by the above two estimates, we conclude that
\begin{align}\label{Ink4-34d}
\norm{\langle \nabla\rangle^s e^{it \Delta} I_{n, k}^{(3)}}_{L_t^{\infty}L_x^2}\lsm \varepsilon_0^{k+1}(M_0, N_0)  \| \langle \nabla\rangle^{s}  e^{it\Delta }I_{n-1-k}\|_{ L_t^{\infty}L_x^2 \cap L_t^{2}L_x^{\frac{2d}{d-2}}}.
\end{align}
Similarly, $\norm{\langle \nabla\rangle^s e^{it \Delta} I_{n, k}^{(3)}}_{L_t^{2}L_x^{\frac{2d}{d-2}}}$ and $\norm{\langle \nabla\rangle^s e^{it \Delta} I_{n, k}^{(3)}}_{L_t^{\infty}L_x^2}$ can be controlled by the same bound.
Hence, we have
\begin{align}\label{Ink4-34d}
\norm{\langle \nabla\rangle^s e^{it \Delta} I_{n, k}^{(3)}}_{L_t^{\infty}L_x^{2}\cap L_t^{2}L_x^{\frac{2d}{d-2}}}\lsm \varepsilon_0^{k+1}(M_0, N_0)\|  \langle \nabla\rangle^{s}e^{it\Delta}I_{n-1-k}\|_{ L_t^{\infty}L_x^2 \cap L_t^{2}L_x^{\frac{2d}{d-2}}}.
\end{align}
{\emph{4) On $I_{n, k}^{(4)}$.}}
 By the duality, Lemmas \ref{lem:Bernstein}, \ref{lem:Schur}, and \ref{lem:Tnk-es34d}, we have
\begin{align}\label{Ink3-34d0}
\norm{\langle \nabla\rangle^s e^{it \Delta} I_{n, k}^{(4)}}_{L_x^2}\lsm &\Big\langle  \langle \nabla\rangle^s\sum _{M\gg N}P_{N}T_N^k(\eta   |\nabla|^{-2} e^{it\Delta }P_{M\geq N_0}I_{n-1-k}), h \Big\rangle\notag\\
\lsm &  \sum _{M\gg N} \frac{\langle N \rangle^{1+s}}{\langle M \rangle^{1+s}}\|  \langle \nabla\rangle^{-1}P_{N}T_N^k(\eta   |\nabla|^{-2} e^{it\Delta }\langle M \rangle^{1+s} P_{M\geq N_0}I_{n-1-k})\|_{L_x^2} \|P_Nh\|_{L_x^2}\notag\\
\lsm&\sum _{M\gg N} \frac{\langle N \rangle^{1+s}}{\langle M \rangle^{1+s}} \|  T_N^k(\eta   |\nabla|^{-2} e^{it\Delta }\langle M \rangle^{1+s} P_{M\geq N_0}I_{n-1-k})\|_{L_x^{\frac{2d}{d+2}}} \|P_Nh\|_{L_x^2}\notag\\
\lsm & \varepsilon_0^k(M_0, N_0)
 \sum _{M\gg N} \frac{\langle N \rangle^{1+s}}{\langle M \rangle^{1+s}} \|  \eta   |\nabla|^{-2} e^{it\Delta }\langle M \rangle^{1+s} P_{M\geq N_0}I_{n-1-k}\|_{L_x^{\frac{2d}{d+2}}} \|P_Nh\|_{L_x^2}\notag\\
\lsm& \varepsilon_0^k(M_0, N_0)\|  \eta  |\nabla|^{-2} e^{it\Delta }\langle M \rangle^{1+s} P_{M\geq N_0}I_{n-1-k}\|_{l_M^2 L_x^{\frac {2d}{d+2}}}.
\end{align}
By Lemmas \ref{lem:Bernstein}, \ref{lem:littlewood-Paley}, and Sobolev's inequality, we have
\begin{align*}
&\|  \eta  |\nabla|^{-2} e^{it\Delta }\langle M \rangle^{1+s} P_{M\geq N_0}I_{n-1-k}\|_{l_M^2 L_x^{\frac {2d}{d+2}}}\\
\lsm &\| P_{\geq M_0} \eta\|_{L_x^{\frac d2}} \| |\nabla|^{-2} e^{it\Delta }\langle \nabla\rangle^{1+s} P_{\geq N_0}I_{n-1-k}\|_{ L_x^{\frac {2d}{d-2}}}\\
&+\| P_{< M_0} \eta\|_{L_x^{d}} \| |\nabla|^{-2} e^{it\Delta }\langle \nabla \rangle^{1+s} P_{\geq N_0}I_{n-1-k}\|_{ L_x^{2}}\\
\lsm &\big(\| P_{\geq M_0} \eta\|_{L_x^{\frac d2}}+ M_0N_0^{-1}\|\eta\|_{L_x^{\frac d2}}\big)\| \langle \nabla\rangle^{s}  e^{it\Delta }I_{n-1-k}\|_{ L_x^{2}}.
\end{align*}
Hence, combining the above two estimates, we conclude that
\begin{align}\label{Ink3-34d1}
\norm{\langle \nabla\rangle^s e^{it \Delta} I_{n, k}^{(4)}}_{L_t^{\infty}L_x^2} \lsm  \varepsilon_0^{k+1}(M_0, N_0)\|  \langle \nabla\rangle^{s}e^{it\Delta}I_{n-1-k}\|_{L_t^{\infty}L_x^{2}}.
\end{align}
Similarly, by the duality, Lemmas \ref{lem:Bernstein}, \ref{lem:Schur}, \ref{lem:Tnk-es34d}, and Sobolev's inequality, we have
\begin{align*}
\norm{\langle \nabla\rangle^s e^{it \Delta}  I_{n, k}^{(4)}}_{L_x^{\frac{2d}{d-2}}}\lsm & \Big\langle  \langle \nabla\rangle^s\sum _{M\gg N} P_{N}T_N^k(\eta   |\nabla|^{-2} e^{it\Delta }P_{M\geq N_0}I_{n-1-k}), h \Big\rangle\notag\\
\lsm &   \sum _{M\gg N} \frac{\langle N \rangle^{s-\frac d2+3+}}{\langle M \rangle^{s-\frac d2+3+}}\|  \langle \nabla\rangle^{\frac d2-3-} P_{N}T_N^k(\eta   \langle M\rangle^{s-\frac d2+3+} \\
&\qquad\qquad\qquad  \cdot|\nabla|^{-2} e^{it\Delta }P_{M\geq N_0}I_{n-1-k})\|_{L_x^{\frac {2d}{d-2}}} \|P_Nh\|_{L_x^{\frac{2d}{d+2}}}\notag\\
\lsm& \sum _{M\gg N} \frac{\langle N \rangle^{s-\frac d2+3+}}{\langle M \rangle^{s-\frac d2+3+}}\| P_{N}T_N^k(\eta   \langle M\rangle^{s-\frac d2+3+}\notag\\
&\qquad\qquad\qquad \cdot |\nabla|^{-2} e^{it\Delta }P_{M\geq N_0}I_{n-1-k})\|_{L_x^{\frac {d}{2}-}}\|P_Nh\|_{L_x^{\frac {2d}{d+2}}}\notag\\
\lsm & \varepsilon_0^k(M_0, N_0)
 \|  \eta   \langle M\rangle ^{s-\frac d2+3+} |\nabla|^{-2} e^{it\Delta }P_{M\geq N_0}I_{n-1-k}\|_{l_M^2L_x^{\frac {d}{2}-}}.
\end{align*}
By Lemmas \ref{lem:Bernstein}, \ref{lem:littlewood-Paley}, and Sobolev's inequality, we have
\begin{align*}
&\|  \eta   \langle M\rangle ^{s-\frac d2+3+} |\nabla|^{-2} e^{it\Delta }P_{M\geq N_0}I_{n-1-k}\|_{l_M^2L_x^{\frac {d}{2}-}}\\
\lsm &\| P_{\geq M_0} \eta\|_{L_x^{\frac d2}} \| |\nabla|^{-2} e^{it\Delta }\langle \nabla\rangle^{s-\frac d2+3+} P_{\geq N_0}I_{n-1-k}\|_{ L_x^{\infty-}}\\
&+\| P_{< M_0} \eta\|_{L_x^{\frac{2d}{6-d}-}} \| |\nabla|^{-2} e^{it\Delta }\langle \nabla \rangle^{s-\frac d2+3+} P_{\geq N_0}I_{n-1-k}\|_{ L_x^{\frac {2d}{d-2}}}\\
\lsm &\big(\| P_{\geq M_0} \eta\|_{L_x^{\frac d2}}+ M_0^{\frac d2-1-}N_0^{-\frac d2+1+}\|\eta\|_{L_x^{\frac d2}}\big)\| \langle \nabla\rangle^{s}  e^{it\Delta }I_{n-1-k}\|_{ L_x^{\frac{2d}{d-2}}}.
\end{align*}
Noting $M_0\leq N_0$ and $\frac d2-1->\frac 14$, the above two estimates yield that
\begin{align}\label{Ink3-34d2}
\norm{\langle \nabla\rangle^s e^{it \Delta} I_{n, k}^{(4)}}_{L_t^{2}L_x^{\frac{2d}{d-2}}} \lsm \varepsilon_0^{k+1}(M_0, N_0)\|  \langle \nabla\rangle^{s}e^{it\Delta}I_{n-1-k}\|_{L_t^2L_x^{\frac{2d}{d-2}}}.
\end{align}
By \eqref{Ink3-34d1} and \eqref{Ink3-34d2}, we have
\begin{align}\label{Ink3-34d3}
\norm{\langle \nabla\rangle^s e^{it \Delta} I_{n, k}^{(4)}}_{L_t^{\infty}L_x^{2}\cap L_t^{2}L_x^{\frac {2d}{d-2}}}\lsm \varepsilon_0^{k+1}(M_0, N_0) \|  \langle \nabla\rangle^{s}e^{it\Delta}I_{n-1-k}\|_{L_t^{\infty}L_x^{2}\cap L_t^{2}L_x^{\frac {2d}{d-2}}}.
\end{align}
{\emph{5) On $I_{n}^{(5)}$.}} By the definition of $T_N^n v_0$, we can rewrite it as follows,
\begin{align}\label{Ink5-34d1}
T_N^n v_0=T_N^{n-1} \big(\sum_{M:  M\gg N, M\geq N_0}\eta |\nabla|^{-2}P_{M} v_0\big).
\end{align}
By Strichartz's estimate, we get
\begin{align}\label{Ink5-134d2}
\norm{\langle \nabla\rangle^s e^{it \Delta} I_{n}^{(5)}}_{L_t^{\infty}L_x^{2}\cap L_t^{2}L_x^{\frac{2d}{d-2}}}\lsm \Big\| \langle \nabla\rangle^s  \sum_{M\gg  N  }P_N T_N^{n-1} (\eta |\nabla|^{-2}P_{M\geq N_0} v_0)\Big\|_{L_x^2}.
\end{align}
An argument parallel to \eqref{Ink3-34d1} yields
\begin{align}\label{In5-234d}
\norm{\langle \nabla\rangle^s e^{it \Delta} I_{n}^{(5)}}_{L_t^{\infty}L_x^{2}\cap L_t^{2}L_x^{\frac{2d}{d-2}}}\lsm  \varepsilon_0^n(M_0, N_0)\| v_0\|_{H_x^s}.
\end{align}
Combining the estimates \eqref{In1-34d}, \eqref{Ink2-34d}, \eqref{Ink4-34d}, \eqref{Ink3-34d3}, and \eqref{In5-234d}, for any $n\geq 1$, we have that
\begin{align}\label{In-52334d}
\norm{\langle \nabla\rangle^s e^{it \Delta}I_n}_{L_t^{\infty}L_x^{2}\cap L_t^{2}L_x^{\frac{2d}{d-2}}}
 \lsm &  \sum_{k=0}^{n-1} \varepsilon_0^{k+1}(M_0, N_0)\|  \langle \nabla\rangle^{s}e^{it\Delta}I_{n-1-k}\|_{L_t^{\infty}L_x^2\cap L_t^{2}L_x^{\frac{2d}{d-2}}}\notag \\
 &+ \varepsilon_0^n(M_0, N_0)\| v_0\|_{H_x^s}.
\end{align}
Hence, by the induction method, we conclude that for any $n\geq 1$,
\begin{align*}
\norm{\langle \nabla\rangle^s e^{it \Delta}I_n}_{L_t^{\infty}L_x^{2}\cap L_t^{2}L_x^{\frac{2d}{d-2}}}
 \lsm \varepsilon_0^n(M_0, N_0)\| v_0\|_{H_x^s}.
\end{align*}
Here we omit the details.
This finishes the proof of this lemma.
\end{proof}

Now, we are in a position to give the proof of Proposition \ref{prop:v-34d}.
\begin{proof}
This proof process is the same as that of Proposition \ref{prop:v-2d}. We omit the details.
\end{proof}

\subsection{Ill-posedness in $H_x^{\frac d2}(\R^d)$}

In this part, we prove that there exists some $\eta\in L_x^{\frac d2}$ with $d=3, 4$, such that the equation \eqref{eq:NLS} is ill-posed in $H_x^{\frac d2}(\R^d)$.
On one hand, we choose the initial data
\EQ{
u_0(x):=\mathscr{F}^{-1}\big(\frac 1{|\xi|^d}\frac 1{\ln |\xi|}\chi_{2\leq |\cdot|\leq M}(\xi)\big)(x),
}
where $M>2$ shall be determined later. Then we have
\EQ{
\|u_0\|_{H_x^{\frac d2}}^2= & \|\langle\xi\rangle^{\frac d2} \widehat{u_0}(\xi)\|_{L_{\xi}^2}^2\\
\lsm & \||\xi|^{-\frac d2}\frac 1{\ln |\xi|}\chi_{2\leq |\cdot|\leq M}(\xi)\|_{L_{\xi}^2}^2\\
\lsm & \int_2^M r^{-d}\frac 1{\ln^2r}r^{d-1}dr\lsm 1.
}
On the other hand, we choose the potential
\EQ{
\eta(x)=M^{2}\mathscr{F}^{-1}\big(\chi_{\frac 12\leq |\cdot|\leq 2}(\xi)\big)(M x).
}
Then we have $\eta\in L_x^{\frac d2}$, and
\EQ{
\widehat{\eta}(\xi)=M^{2-d}\chi_{\frac 12\leq |\cdot|\leq 2}\big(\frac{\xi}{M}\big).
}
Define
\EQ{
A (u_0)(t)\triangleq\int_0^t e^{-i\rho\Delta}(\eta e^{i\rho\Delta}u_0)d\rho.
}
Define
$$
t\triangleq \frac 1{M^2},
$$
and
$$
\Omega=\{\xi: \sqrt{\frac {\pi}{3}}M\leq|\xi|\leq \sqrt{\frac {\pi}{2}}M\}.
$$
Following exactly the same process as in Theorem \ref{theorem-main 1}, we can obtain for any $T>0$ and large enough $M$,
\EQ{
\sup\limits_{t\in [0,T]}\|A[u_0]\|_{H_x^{\frac d2}(\R^d)}\geq \frac 18 \ln \ln M.
}
Therefore, this implies
\EQ{
\sup\limits_{t\in [0,T]}\|A[u_0]\|_{H_x^{\frac d2}(\R^d)}\rightarrow \infty, \mbox{ as }M\rightarrow \infty.
}
The proof of ill-posedness is done by applying Lemma \ref{ill-posed}. We finish the proof of Theorem \ref{theorem-two}.

\section{Subcritical case: the proof of Theorem \ref{theorem-three}}\label{H2-reg}

In this section, we first establish the global well-posedness in $L_x^2$, then we improve this result to $H_x^2$ through the transformation $v=\partial_t u$.

\subsection{Global well-posedness in $L_x^2(\R^d)$}\label{GWP-L2}
We first establish the global well-posedness of \eqref{eq:NLS} in the space $L_x^{2}(\R^d)$. This constitutes a weak regularity result, as the expected critical regularity requires in $H_x^2$.
\begin{prop}\label{weak-subcritical-2}
Let $d\geq 2$, $r\geq\frac d2$ and $r>2$, and $\eta\in L_x^{r}(\R^d)$, then \eqref{eq:NLS} is globally well-posed in $L_x^{2}(\R^d)$.
\end{prop}
We begin by presenting the required inhomogeneous estimates to prove the above result.
\begin{lem}[$d= 2, 3, 4$]\label{lem:space-time-d}
 Let $\eta \in L_x^{r}(\R^d)$, $r> 2$. Let $I=[0, T)\subset \R^+$, then for $d=2$,
 \begin{align*}
 \norm{\int_0^t e^{i(t-\rho)\Delta}(\eta u)(\rho)d\rho}_{L_t^{\infty} L_x^2(I)}\lsm T^{1-\frac 1r}\norm{\eta}_{L_x^r}\norm{u}_{L_t^{\infty}L_x^2},
 \end{align*}
 for $d=3, 4$,
  \begin{align*}
 \norm{\int_0^t e^{i(t-\rho)\Delta}(\eta u)(\rho)d\rho}_{L_t^{\infty}L_x^2\cap L_t^2L_x^{\frac{2d}{d-2}}(I)}\lsm T^{1-\frac d{2r}}\norm{\eta}_{L_x^r}\norm{u}_{L_t^{\infty}L_x^2\cap L_t^2L_x^{\frac{2d}{d-2}}}.
 \end{align*}
\end{lem}
\begin{proof}
By the Strichartz estimates and H\"older's inequality, we have that for $d=2$,
\begin{align*}
 \norm{\int_0^t e^{i(t-\rho)\Delta}(\eta u)(\rho)d\rho}_{L_t^{\infty}L_x^2}\lsm \norm{\eta u}_{L_t^{\frac{r}{r-1}} L_x^{\frac{2r}{2+r}}} \lsm T^{1-\frac 1r}\norm{\eta}_{L_x^r}\norm{u}_{L_t^{\infty}L_x^2}.
 \end{align*}
 For $d=3, 4$, we have that when $2<r\leq d$, then
  \begin{align}\label{34d-S1}
 \norm{\int_0^t e^{i(t-\rho)\Delta}(\eta u)(\rho)d\rho}_{L_t^{\infty}L_x^2\cap L_t^2L_x^{\frac{2d}{d-2}}}\lsm  \norm{\eta u}_{L_t^{\frac {2r}{3r-d}}L_x^{\frac {2rd}{2d-2r+rd}}}\lsm T^{1-\frac d{2r}}\norm{\eta}_{L_x^r}\norm{u}_{L_t^2L_x^{\frac{2d}{d-2}}};
\end{align}
when $r>d$, then
  \begin{align}\label{34d-S2}
\norm{\int_0^t e^{i(t-\rho)\Delta}(\eta u)(\rho)d\rho}_{L_t^{\infty}L_x^2\cap L_t^2L_x^{\frac{2d}{d-2}}}\lsm  \norm{\eta u}_{L_t^{\frac {2r}{2r-d}}L_x^{\frac {2r}{2+r}}}\lsm T^{1-\frac d{2r}}\norm{\eta}_{L_x^r}\norm{u}_{L_t^{\infty}L_x^{2}}.
\end{align}
 This gives the proof of this lemma.
\end{proof}

\begin{lem}[$d\geq 5$]\label{5d}
 Let $\eta \in L_x^{r}(\R^d)$, $r\geq \frac d2$. Let $I=[0, T)\subset \R^+$, then for any $N_0\in 2^{\N}$, we have that for $r=\frac d2$,
\begin{align}\label{34d-S3}
\norm{\int_0^t e^{i(t-\rho)\Delta}\eta u(\rho)d\rho}_{L_t^{\infty}L_x^2\cap L_t^2L_x^{\frac {2d}{d-2}}(I)}\lsm \norm{P_{\geq N_0} \eta}_{L_x^{\frac d2}}\norm{ u}_{L_t^{2}L_x^{\frac{2d}{d-2}}}
+T^{\frac 12}N_0\norm{ \eta}_{L_x^{\frac d2}}\norm{ u}_{L_{t}^{\infty}L_x^2};
\end{align}
for $r>\frac d2$,
\begin{align}\label{34d-S4}
\norm{\int_0^t e^{i(t-\rho)\Delta}\eta u(\rho)d\rho}_{L_t^{\infty}L_x^2\cap L_t^2L_x^{\frac{2d}{d-2}}(I)}\lsm
T^{1-\frac d{2r}}\norm{ \eta}_{L_x^r}\norm{ u}_{L_{t}^{\infty}L_x^{2}\cap L_t^2L_x^{\frac{2d}{d-2}}}.
\end{align}
\end{lem}
\begin{proof}
When $r=\frac d2$, by Lemmas \ref{lem:Bernstein}, \ref{lem:strichartz}, we get
\begin{align*}
\norm{\int_0^t e^{i(t-\rho)\Delta}\eta u(\rho)d\rho}_{L_t^{\infty}L_x^2\cap L_t^2L_x^{\frac{2d}{d-2}}}\lsm & \norm{\eta u}_{L_t^{2}L_x^{\frac{2d}{d+2}}}\\
\lsm & \norm{P_{\geq N_0} \eta}_{L_x^{\frac d2}}\norm{ u}_{L_t^{2}L_x^{\frac{2d}{d-2}}}
+\norm{P_{< N_0} \eta}_{L_x^{d}}\norm{ u}_{L_{t, x}^{2}}\\
\lsm & \norm{P_{\geq N_0} \eta}_{L_x^{\frac d2}}\norm{ u}_{L_t^{2}L_x^{\frac{2d}{d-2}}}
+T^{\frac 12}N_0\norm{ \eta}_{L_x^{\frac d2}}\norm{ u}_{L_{t}^{\infty}L_x^2}.
\end{align*}
This gives \eqref{34d-S3}. Noting $r>\frac d2>2$ for $d\geq 5$, \eqref{34d-S4} is followed by \eqref{34d-S1} and \eqref{34d-S2}.
This finishes the proof.
\end{proof}

\begin{proof}[\bf{Proof of Proposition \ref{weak-subcritical-2}}]
In Lemmas \ref{lem:space-time-d}, \ref{5d}, the factors $\norm{P_{\geq N_0} \eta}_{L_x^{\frac d2}}$ and $T^{\gamma}$ for some $\gamma>0$ provide smallness. Using the standard contraction mapping principle, we can easily obtain the local well-posedness for the equation \eqref{eq:NLS} in $L_x^2$. Besides, since the local lifespan depends only $\norm{\eta}_{L_x^r}$, the local solution can be extended globally.
\end{proof}

\begin{remark} In establishing the global well-posedness, we can obtain that for any $T>0$,
\begin{align}\label{127-1447}
\norm{u}_{L_t^{\infty}L_x^2([0, T)\times \R^d)}\leq C(T) \norm {u_0}_{L_x^2}.
\end{align}
This bound can not be derived from the mass conservation, since for complex-valued potentials $\eta$, the equation \eqref{eq:NLS} no longer preserves $L_x^2$-norm.
\end{remark}

\subsection{Global well-posedness in $H_x^2(\R^d)$}\label{GWP-H2}
Employing the Strichartz's estimates does not suffice to further improve the regularity, we adopt an alternative approach to achieve it.
Let $v=\partial_t u$, from \eqref{eq:NLS}, $v$ satisfies the following equation
\begin{equation}\label{eq:NLS-125}
		\left\{ \aligned
		&i\partial_t v+\Delta v+\eta v=0, \qquad t\in [0, T) \mbox{ and } x\in \R^d,
		\\
		&v(0,x)=i(\Delta u_0+\eta u_0)\triangleq v_0.
		\endaligned
		\right.
	\end{equation}
We now have the following two key observations:

\noindent $\bullet$ Claim 1: $v_0\in L_x^2$.

Indeed, by $ u_0\in H_x^{2}$, $\eta\in L_x^{ r}$ with $r>2$ and $r\geq \frac d2$, the H\"older and Sobolev inequalities, we obtain
\begin{align}\label{12711}
\norm{v_0}_{L_x^2}\lsm &\norm {u_0}_{H_x^2}+\norm{\eta}_{L_x^{r}}\norm{u_0}_{L_x^{\frac {2r}{r-2}}}\notag\\
\lsm &\norm {u_0}_{H_x^2}+\norm{\eta}_{L_x^{r}}\norm{u_0}_{H_x^{2}}.
\end{align}
\noindent $\bullet$ Claim 2: $v\in C([0, T); L_x^2(\R^d))$ implies $u\in C([0, T); H_x^2(\R^d))$.

Indeed, by the high and low frequency decomposition, Lemma \ref{lem:Bernstein}, and \eqref{127-1447}, we have
\begin{align}\label{127-1449}
\norm{u}_{L_t^{\infty}H_x^2}\leq & \norm{P_{<1}u}_{L_t^{\infty}H_x^2}+ \norm{P_{\geq 1}u}_{L_t^{\infty}H_x^2}\notag \\
\leq & C(T) \norm{u_0}_{H_x^2}+ \norm{P_{\geq 1}u}_{L_t^{\infty}H_x^2}.
\end{align}
It is reduced to consider $\norm{P_{\geq 1}u}_{L_t^{\infty}H_x^2}$ in \eqref{127-1449}. Noting that $\Delta u=-i v-\eta u$, by the H\"older and Sobolev inequalities, Lemma \ref{lem:Bernstein}, and \eqref{127-1447}, we have
\begin{align}\label{127-1450}
\norm{P_{\geq 1}u}_{L_t^{\infty}H_x^2}\lsm & \norm {v}_{L_x^2}+\norm{ u P_{\geq N_0} \eta  }_{L_t^{\infty}L_x^2}+\norm{ u P_{< N_0} \eta   }_{L_t^{\infty}L_x^2}\notag \\
\lsm &C(T)\norm {v_0}_{L_x^2}+\norm{ P_{\geq N_0} \eta }_{L_x^{r}}\norm {u}_{L_t^{\infty}L_x^{\frac {2r}{r-2}}}+\norm{ P_{< N_0} \eta }_{L_x^{\infty}}\norm{ u}_{L_t^{\infty}L_x^2}\notag \\
\lsm &C(T)\norm {v_0}_{L_x^2}+\norm{ P_{\geq N_0} \eta }_{L_x^{r}}\norm {u}_{L_t^{\infty}H_x^{2}}+C(T)N_0^{\frac dr}\norm{ \eta }_{L_x^{r}}\norm{ u_0}_{H_x^2},
\end{align}
where $N_0\in 2^{\N}$ is a large constant decided later.

Hence, by \eqref{12711}-\eqref{127-1450}, we obtain
\begin{align}\label{127-1451}
\norm{u}_{L_t^{\infty}H_x^2} \leq & C(T)(1+N_0^{\frac dr}\norm{ \eta }_{L_x^{r}}) \norm{u_0}_{H_x^2}+ C\norm{ P_{\geq N_0} \eta }_{L_x^{r}}\norm {u}_{L_t^{\infty}H_x^{2}}.
\end{align}
Now, we take $N_0=N_0(\norm{  \eta }_{L_x^{r}})$ large enough, such that
\begin{align}\label{127-1452}
C\norm{ P_{\geq N_0} \eta }_{L_x^{r}}\leq \frac 12.
\end{align}
Hence, by \eqref{127-1451} and \eqref{127-1452}, we conclude that
\begin{align}\label{127-1453}
\norm{u}_{L_t^{\infty}H_x^2} \leq  C( T, \norm{  \eta }_{L_x^{r}}, \norm{ u_0}_{H_x^2}).
\end{align}
This completes the proof of this claim.

\begin{proof}[\bf{Proof of global well-posedness in $H_x^2$}]
By Claim 2, the global well-posedness of \eqref{eq:NLS} in $H_x^2$ reduces to the global well-posedness of \eqref{eq:NLS-125} in $L_x^2$. Note that \eqref{eq:NLS-125} shares the same structure as \eqref{eq:NLS}. Hence, the space-time estimates in Lemmas \ref{lem:space-time-d}, \ref{5d} for \eqref{eq:NLS} also hold for \eqref{eq:NLS-125}. Combining Claim 1, we can obtain the the global well-posedness of \eqref{eq:NLS-125} in $L_x^2$. This finishes the proof.
\end{proof}

\subsection{Ill-posedness in $H_x^{2+}(\R^d)$}
In this part, we aim to prove the result that for any $\gamma>2$, there exists some $\eta\in L_x^{r}(\R^d)$ with $r>2$, $r\geq \frac d2$, and $d\geq 5$, such that the equation \eqref{eq:NLS} is ill-posed in $H_x^{\gamma}(\R^d)$.

We set the parameters $M, N, L\geq 1$, which shall be determined later. Next, on one hand, we choose the initial data
\EQ{
u_0(x):=\mathscr{F}^{-1}\big(L^{-\frac d2-\gamma }\prod_{i=1}^d\chi_{\frac L2\leq |\cdot|\leq 2L}(\xi^{(i)})\big)(x).
}
Then we have
\EQ{\|u_0\|_{H_x^{\gamma}(\R^d)}^2=\|\langle\xi\rangle^\gamma \widehat{u_0}(\xi)\|_{L_{\xi}^2(\R^d)}^2\sim 1,
}
where $\xi=(\xi^{(1)}, \xi^{(2)}, \cdots, \xi^{(d)})$.
On the other hand, we choose the potential
\EQ{
\eta(x)=N^{-d+\frac dr}\mathscr{F}^{-1}\big(\chi_{\sqrt{\frac {\pi}{3}}M\leq |\cdot|\leq \sqrt{\frac {\pi}{3}}M+N}(\xi^{(1)})\cdot\prod_{i=2}^d\chi_{\frac N2\leq |\cdot|\leq 2N}(\xi^{(i)})\big)(x).
}
Then we have
\EQ{
\widehat{\eta}(\xi)=N^{-d+\frac dr}\chi_{\sqrt{\frac {\pi}{3}}M\leq |\cdot|\leq \sqrt{\frac {\pi}{3}}M+N}(\xi^{(1)})\cdot \prod_{i=2}^d\chi_{\frac N2\leq |\cdot|\leq 2N}(\xi^{(i)}).
}
Moreover, noting $\chi_{\sqrt{\frac {\pi}{3}}M\leq |\cdot|\leq \sqrt{\frac {\pi}{3}}M+N}(\xi^{(1)})$ and $\chi_{\frac N2\leq |\cdot|\leq 2N}(\xi^{(i)})$ ($i=2,3, \cdots, d$) are Schwartz functions, hence for any $r>2$, we have
\EQ{
\norm{\eta}_{L_x^r}\lsm \norm{\widehat{\eta}}_{L_{\xi}^{r'}}\lsm N^{-d+\frac dr} N^{d(1-\frac 1r)} =1,
}
where $r'$ satisfies $\frac 1r+\frac 1{r'}=1$.

Define
\EQ{
A (u_0)(t)\triangleq\int_0^t e^{-i\rho\Delta}(\eta e^{i\rho\Delta}u_0)d\rho.
}
We aim to prove that for any $\gamma>2$,
\EQ{
\sup\limits_{t\in [0,1]}\|A(u_0)(t)\|_{H_x^{\gamma}(\R^d)}\rightarrow \infty, \mbox{ as }M\rightarrow \infty.
}
Define
$$
t\triangleq \frac 1{M^2},
$$
and
\begin{align*}
\Omega=\left\{\xi: \sqrt{\frac {\pi}{3}}M+\frac N4\leq\xi^{(1)}\leq \sqrt{\frac {\pi}{3}}M+\frac {3N}4, \frac {3N}4\leq\xi^{(i)}\leq\frac {7N}4 (\mbox{where }i=2, 3, \cdots, d)\right\}.
\end{align*}
Following exactly the same process as in Theorem \ref{theorem-main 1}, we get
\begin{align*}
\|A(u_0)\|_{H_x^{\gamma}(\R^d)}
\geq  C(N, L) M^{\gamma-2},
\end{align*}
where $C(N, L)>0$ is a finite constant. Hence, by $\gamma>2$, we conclude that ant $T>0$,
\begin{align}\label{Bu0-hd}
\sup\limits_{t\in [0,T]}\|A(u_0)\|_{H_x^{\gamma}(\R^d)}\rightarrow \infty, \mbox{ as }M\rightarrow \infty.
\end{align}
The proof of ill-posedness is done by applying Lemma \ref{ill-posed}.

	\vskip 0.2cm

\end{document}